\documentclass[11pt, a4paper]{amsart}
\usepackage{bbm, bm, amssymb,stmaryrd, mathrsfs, comment}
\usepackage{verbatim}
\usepackage{xcolor}
\newtheorem{theorem}{Theorem}[section]
\newtheorem{lemma}[theorem]{Lemma}

\newtheorem{corollary}[theorem]{Corollary}
\newtheorem{proposition}[theorem]{Proposition}

\theoremstyle{definition}
\newtheorem{definition}[theorem]{Definition}

\newtheorem{example}[theorem]{Example}

\newtheorem{remark}[theorem]{Remark}

\newcommand{\I}{{\mathbbm{1}}}
\newcommand{\M}{\mathcal{M}}
\newcommand{\A}{\mathcal{A}}

\newcommand{\B}{\mathcal{B}}
\newcommand{\C}{\mathcal{C}}
\newcommand{\D}{\mathcal{D}}
\newcommand{\N}{\mathcal{N}}
\newcommand{\Rdb}{{\mathbb R}}

\newcommand{\Ndb}{{\mathbb N}}
\newcommand{\Cdb}{{\mathbb C}}

\newcommand{\Pdb}{{\mathbb P}}
\newcommand{\p}[1]{\mathbb{P}({#1})}

\newcommand{\Balg}{{\mathfrak B}}
\newcommand{\Malg}{{\mathfrak A}}

\newcommand{\vf}{\varphi}
\newcommand{\vniso}{\bm{\Phi}}
\newcommand{\dec}[1]{\mathfrak{d}(#1)}
\newcommand{\decp}[2]{\mathfrak{d}_{#1}(#2)}

\newcommand{\kk}[1]{\mathfrak{k}(#1)}
\newcommand{\magn}[2]{\mathfrak{m}_{#1}(#2)}
\newcommand{\ve}{\varepsilon}
\renewcommand{\b}[1]{\overline{#1}}
\newcommand{\cU}{\mathcal{U}}
\newcommand{\cP}{\mathcal{P}}

\begin{document}
\title[Abelian von Neumann algebras, measure, and $L^\infty$]{Abelian von Neumann algebras, measure algebras and $L^\infty$-spaces}


\author{David P. Blecher}
\address{Department of Mathematics, University of Houston, Houston, TX
77204-3008}
\email[David P. Blecher]{dpbleche@central.uh.edu}
\author{Stanis{\l}aw Goldstein}
\address{Faculty of Mathematics and Computer Science\\
   University of {\L}\'od\'z\\
   Banacha 22\\90-238 {\L}\'od\'z\\Poland}
\email[Stanis{\l}aw Goldstein]{stanislaw.goldstein@wmii.uni.lodz.pl}
\author{Louis E. Labuschagne}
\address{DSI-NRF CoE in Math. and Stat. Sci, Pure and Applied Analytics\\ Internal Box 209, School of Math. \& Stat. Sci.,
NWU, PVT. BAG X6001, 2520 Potchefstroom, South Africa}
\email[Louis Labuschagne]{louis.labuschagne@nwu.ac.za}

\thanks{DB is supported by a Simons Foundation Collaboration Grant (527078).}

\begin{abstract}
  We give a fresh account of the astonishing interplay between abelian von Neumann algebras, $L^\infty$-spaces and measure algebras, including an exposition of Maharam's theorem from the
von Neumann algebra perspective.  \end{abstract}


\maketitle

\medskip

\section{Introduction}

It is a beautiful fact, seemingly attributable to Segal \cite{Segal}, that the class of localizable measure spaces ``corresponds'' to the abelian von Neumann algebras, with the correspondence given by taking such a measure space to
its $L^\infty$ algebra (realized as multiplication operators acting on its $L^2$ space).
Contributions from other authors (see for example\ \cite{Sakai, tak1}) show that the classes of decomposable and Radon measure spaces fulfill essentially the same role, but with localizable spaces being the most general class among these three.
Segal demonstrated in this 
(now quite dated) paper that
localizable spaces ``constitute the most significant general class of measure spaces properly containing the finite ones, and are of maximum generality consistent with usefulness in any large sector of abstract analysis''.  He did this by identifying a number of important measure-theoretic and functional analytic properties that he showed to each be equivalent to localizability.

Segal also generalized to localizable spaces  Dorothy Maharam's stunning structural classification
of finite measure spaces \cite{maharam}.
He showed that such spaces can be decomposed into
finite pieces to which Maharam's result applies (see the lines above
\cite[Theorem 3.3]{Segal}).
Instead of measure spaces, Maharam's theorem 
naturally belongs in the setting of {\em measure algebras}, objects which have been thoroughly investigated by Fremlin (see \cite[Volume III]{Fremlin} and \cite{FremlinMA}) and which
 may be derived from measure spaces. See \cite[332B]{Fremlin} or \cite[Theorem 3.9]{FremlinMA} for the striking measure algebra
approach to the (generalized) Maharam's theorem. One sees that localizable measure algebras are essentially just simple products of measure algebras of the spaces $\Omega_\kappa = \{ 0, 1 \}^\kappa$ for various cardinals $\kappa$, with its usual measure (perhaps scaled by some factor).  The von Neumann algebraic formulation of
Maharam's theorem is usually phrased
 as the statement that every abelian von Neumann algebra is isomorphic to a direct sum of
$L^\infty$ spaces of such measure spaces $\Omega_\kappa$.
Although the latter is true
(see Corollary \ref{vncase} in conjunction with \ref{AkapaL}), it does not convey the full
import of the result, as we shall see.

It is a well known saying that `von Neumann algebras are noncommutative measure theory'.   Indeed much of von Neumann algebra theory may be viewed as a vast generalization
of classical integration theory, and is an exquisitely appropriate formalism for
``noncommutative'' analysis.    For both classical and noncommutative analysts it is therefore important to have a very firm grip on
the commutative case, and it is thus the purpose of our paper to assemble
the most significant aspects of this theory in one place.
We give a fresh account of the astonishing interplay between abelian von Neumann algebras, $L^\infty$-spaces and measure algebras. The exposition is directed towards non-experts, and somewhat towards graduate students, although we hope that even specialists 
will find some of the contents enlightening.
There are very many things 
 throughout our paper that are either not in the literature, or are not easy to find
there, at least in the form given here.  It is worth saying that there are many routes through much of this material, and while the route taken here seemed an optimal one to obtain all the results in the paper, it may not necessarily be the shortest route to the reader's favorite result.

We  develop all the theory essentially from `scratch', along a fresh path,
so that almost no background of $C^*$-  and von Neumann algebra theory is necessary.   Certainly not much beyond say the definition of a $C^*$-algebra, the simple exercise that $C(K)$ spaces and $L^\infty$ spaces are
commutative (abelian) 
unital $C^*$-algebras, and conversely any of the latter algebras
are $*$-isomorphic to a $C(K)$ space.
We will assume that the reader has taken a good real variables course, and we refer
to \cite{Fremlin} for definitions and results in measure and integration theory and basic functional analysis.
In the final section \ref{MahS}, where we offer a von Neumann algebraic proof of Maharam's theorem, we will assume  a
little more knowledge of $C^*$-  and von Neumann algebra theory, as
may be quickly gleaned from
 early sections of any one of the texts \cite{Black,KR,Ped, Sakai, tak1}.

Early results of Dixmier  and others characterize abelian von Neumann algebras in various ways,
most notably in terms of monotone completeness or hyperstonian spaces \cite{DixS}, Radon measure spaces \cite[Theorem III.1.18]{tak1}, 
 `decomposable spaces', and
$C(K)$ spaces having a predual:
Sakai's abelian $W^*$-algebras, see \cite[Theorem 1.18.1]{Sakai}.    (Definitions of these terms
are given below. 
  Although
Sakai uses the term `localizable' in the cited theorem, it is clear that what he has in mind there are what we here call decomposable spaces.)   A $W^*$-{\em algebra} is a $C^*$-algebra with a Banach space predual,
and Sakai's famous theorem (which we prove here in Corollary \ref{awvn} in the abelian setting)
states that these are precisely the von Neumann algebras, up to weak* homeomorphic
$*$-isomorphism \cite[Theorem 1.16.7]{Sakai}.
The reader will easily see that by the usual duality of $L^p$ spaces from a graduate real variables  class,
the $L^\infty$ spaces of $\sigma$-finite measure spaces are abelian $W^*$-algebras.
Conversely
one may rephrase most of the first lines of this paragraph as
saying that the abelian $W^*$-algebras `are exactly'
the $L^\infty$ spaces of localizable, or Radon, or decomposable, measure spaces.

Thus the above three classes of measure spaces (localizable, Radon, and decomposable)
are all  part of the same whole (see Theorem \ref{Radmeasalg}).  In fact it is precisely
this analysis which we develop carefully here over several segments of our paper,  and in particular
in the first few subsections of Section \ref{chmav}. 
We 
 show there for example that the category of abelian von Neumann algebras equipped with a faithful normal semifinite trace is functorially equivalent to the category of localizable measure algebras (see Theorem \ref{vna=ma}).

As seen from the above, there is a lot of information on abelian von Neumann algebras in the literature. Nevertheless, often much of this material is not easily accessible to a graduate student, 
and many questions are left unanswered.
For example, it is sometimes claimed that Maharam's theorem amounts to a classification of abelian von Neumann algebras.  This has entrenched itself as a folk theorem, yet a precise description of what exactly is meant by this claim -- let alone a detailed proof -- is hard to find in the literature. One of the aims of our paper is to remedy this fact. The connection between localizable measure algebras and abelian von Neumann algebras equipped with a faithful normal semifinite trace
stems from the fact that the projection lattice of such a von Neumann algebra is a localizable measure algebra when equipped with the restriction of the trace (see Example \ref{vnama}).
From this it is easy to make precise the sense in which Maharam's theorem
`is truly' a classification of abelian von Neumann algebras.
The aforementioned von Neumann algebraic proof of Maharam's theorem in Section
\ref{MahS}
gives important extra insights into the classification of abelian von Neumann algebras,
and some of the steps are actually
more natural, clearer, and more illuminating for a functional analyst.
In this regard Proposition \ref{akhom} is a case in point:
tensor products work well in the category of von Neumann measure algebras, but
this  cannot be said about free products in measure algebras.
There is one main
step which is simpler to perform in the measure algebra context, and for
this
we content ourselves with merely referencing
the relevant measure algebraic result.

We briefly mention a few other unique features of this paper.   In Sections
\ref{wol} and \ref{rns} we discuss the relations between weights and measures and the Radon-Nikodym theorem.
Although Segal did include a variant of the Radon-Nikodym theorem in his characterization of localizable spaces,
it is not the most general one possible, nor will it meet all of our needs.   In its place
we  give a new
and exceedingly general Radon-Nikodym theorem
(Corollary \ref{rn2} and Theorem \ref{rnd2}; see Theorem \ref{frec} (vi) for the equivalence
with localizability).   Sections
\ref{LinftyW1}, \ref{Ws2}, and  \ref{cmsvn}  study  $L^\infty$ spaces as $W^*$-algebras
from very many perspectives, and with an increasingly
sophisticated perspective, and with
new proofs in the main.   
Section \ref{Rad} contains  a structural analysis of the class of Radon measure spaces we need in the latter part of this paper, again with some new angles.    
At the end of Section \ref{Rad} we give an application both of 
the topics of that section and of our general Radon-Nikodym theorem from Section \ref{rns}: we give a new proof of the 
uniqueness of Haar measure.  For this we use an idea of von Neumann, which we make to work in the general case. Sections \ref{chmav} and \ref{MahS} have already been discussed, but contain many
other important results, such as the relationship with Boolean algebras and Stone's theorem
(see Subsection \ref{rmal}), and new proofs of
abelian variants of important theorems of Kadison, Pedersen, and Bade characterizing
 `concrete' commutative von Neumann algebras (see Subsection \ref{kpb}). 
 
This paper covers quite a swathe of topics, and readers interested in a specific topic need not wade through the entire paper 
or parts that they find a little too dense.   For example readers only interested in seeing, and appreciating the background to, an extremely general Radon-Nikodym 
theorem, need read no further than Section \ref{rns}. Those only interested in the main 
relationships between localizable measure spaces, $L^\infty$ 
which are $W^*$ algebras, topics related to 
Segal's objectives summarized in the second half of the first paragraph of 
our paper, etc, may skip the details of Section 4, 
and read Sections \ref{LinftyW1}, \ref{wol} and \ref{Ws2} for the semifinite setting, 
and Section \ref{cmsvn} for a more general result. 
Readers interested in 
the relationship between abelian von Neumann algebras and measure algebras, may take the earlier theory at face value, and 
focus on \S \ref{rmal}--\ref{ssmavn}.   Sections 
\ref{kpb} and \ref{cmsvn} could be read immediately after Section \ref{Ws2} if desired. 
Section \ref{Rad} is independent of earlier parts of the paper, as is  Section \ref{MahS} 
apart from some definitions from \S  \ref{rmal} and the contextualization afforded by Theorems \ref{vna=ma} and 
\ref{eqca}.

Turning to notation, the underlying field is the complex numbers $\Cdb$ unless stated to the contrary.
The `$C^*$-algebra approach' is a constant  motif in this paper.
Since $C^*$-algebras have an ordering $\leq$ on the selfadjoint part, suprema (least upper bounds) of bounded sets make sense in the selfadjoint part of any $C^*$-algebra, although for some  $C^*$-algebras they  may not exist.    We write $f_t \nearrow f$ to denote that $f$ is the supremum of the increasing net $(f_t)$.   It is an
exercise in spectral theory 
 that an algebraic isomorphism between abelian $C^*$-algebras (or $C(K)$ spaces) is an isometric
$*$-isomorphism (this follows  from e.g.\ Proposition I.4.5 and the proof of I.5.3 and I.5.4 in \cite{tak1}).

We denote the set of projections in a $C^*$-algebra $\M$ by $\p{\M}$.  Unless stated otherwise, the  $C^*$-algebras in this paper are abelian and unital,
and in this case the  projections in $\M$ are just the idempotent elements, and they
form an orthocomplemented lattice, indeed a Boolean algebra (defined later since we shall not need this concept in Sections \ref{LinftyW1}--\ref{Rad}), with
infimum $p \wedge q = pq$ and supremum $p\vee q= p+q - pq$. 
By \emph{orthogonal family in} $\M$  we mean a family of non-zero mutually orthogonal projections from $\M.$ 
The projections on a Hilbert space are the selfadjoint idempotent operators.

The term \emph{direct sum of $C^*$-algebras} is always understood
as an $\ell^\infty$-direct sum.
We will use the simple fact that such a direct sum of (commutative) $C^*$-algebras is a $C^*$-algebra with norm $\sup_i \, \| x_i \|$.
The reader should be warned that there is another  direct sum of $C^*$-algebras,
namely the $c_0$-direct sum, as in \cite[II.8.1.2]{Black} which reserves the name \emph{direct product} for our direct sum.

If $H$ is a Hilbert space then the set $B(H)$ of bounded linear operators of $H$ is a $C^*$-algebra, and indeed is
a $W^*$-algebra since it has a Banach space predual.  From a Banach space perspective the most elementary way to see the latter fact, and to identify the weak* continuous functionals,
 is via a rudimentary point about the {\em projective tensor product} $\hat{\otimes}$.
We will not take the time to define $\hat{\otimes}$, which may be found in many 
texts (for example IV.2 in \cite{tak1}).   Suffice it to say that $(H \hat{\otimes} \bar{H})^* = B(H)$ isometrically.  It is an exercise from this, and from the definition of $\hat{\otimes}$,  that the weak* continuous linear functionals on $B(H)$ are precisely the functionals
$\sum_{k=1}^\infty \langle \, \cdot \, \zeta_k , \eta_k \rangle$,
where $\zeta_k , \eta_k \in H$ with $\sum_{k=1}^\infty \, \| \zeta_k \| \, \| \eta_k \|
< \infty$.   A simple scaling trick shows that we may assume here that
$\sum_{k=1}^\infty \, \| \zeta_k \|^2$ and $\sum_{k=1}^\infty \, \| \zeta_k \|^2$ are finite. 
These facts are usually presented to students
via trace duality: viewing $B(H)$ as the dual of the
trace class operators (see for example \cite{DJT}.) 
  The weak* topology on $B(H)$ is often called the $\sigma$-{\em weak topology}.

A von Neumann algebra is a
weak* closed
 $*$-subalgebra $\M$ of $B(H)$
containing
the identity operator.
This is not the usual definition of a von Neumann algebra, so we
take a moment to describe the latter, although we do not really
use it.   By von Neumann's famous double commutant
theorem \cite{KR,Ped,Sakai,tak1} one may show that a  unital $*$-subalgebra $\M$ of $B(H)$ is
weak* closed, and hence is a von Neumann algebra, if and only if $\M$ equals its bicommutant.
Indeed, if and only if $\M$ is closed in the strong (resp.\ weak operator) topology. We recall that  
a net $T_t \to T$ strongly (resp.\ weak operator) if and only if $T_t \zeta \to T \zeta$
(resp.\ $\langle T_t \zeta , \eta \rangle \to \langle T \zeta , \eta \rangle$) for all $\zeta, \eta \in H$.   By a routine functional analytic argument,
on the unit ball the weak operator topology coincides with the $\sigma$-weak topology.
For increasing bounded nets convergence for these three locally convex topologies coincide.

Thus by the above and basic duality principles in functional analysis,
the predual (resp.\ positive part of the predual) of a von Neumann algebra $\M$
on $H$, which may be identified with the weak* continuous (resp.\ positive weak* continuous) linear functionals on $\M$,  consists of the functionals above of form
$\sum_{k=1}^\infty \langle \, \cdot \zeta_k , \eta_k \rangle$ (resp.\ with
$\zeta_k = \eta_k$).

A 
{\em normal} $*$-{\em representation} of a von Neumann algebra is a weak* continuous
$*$-homomorphism from $\M$ into $B(H)$ for a Hilbert space $H$.
Every von Neumann algebra is clearly a $W^*$-algebra (by the definitions and by Banach space duality theory), and as we mentioned above Sakai's theorem
gives the  converse: Every $W^*$-algebra is representable
via a normal $*$-isomorphism 
as a von Neumann algebra
(see Theorem 1.16.7 in \cite{Sakai}).  Thus sometimes the term von Neumann algebra is used loosely
without specifying a fixed Hilbert space representation.

\section{$L^\infty$-spaces as $W^*$-algebras - 1} \label{LinftyW1}

In this section we determine which $L^\infty$-spaces living on a semifinite measure space  $(X,\Sigma,\mu)$, are $W^*$-algebras.   As is usual we sometimes write $(X,\mu)$ for $(X,\Sigma,\mu)$ when $\Sigma$ is understood.

We will say that a measure space $(X,\Sigma,\mu)$ is  {\em decomposable} (sometimes called {\em strictly localizable})
if we can partition $X$ into measurable $\mu$-finite subsets $X_t$ such that
$$\Sigma = \{ E \subset X : E \cap X_t \in \Sigma \; \textrm{for all} \; t \},$$
and $\mu(E) = \sum_t \, \mu(E \cap X_t)$ for all $E \in \Sigma$.
For a measure space $(X,\Sigma,\mu)$ it is elementary to verify that $\M = L^\infty(X,\Sigma,\mu)$ is a commutative $C^*$-algebra; as is $C(K)$ for compact $K$.  The sets in $\Sigma$ correspond (but not bijectively)
to the idempotents (projections) in $L^\infty(X,\Sigma,\mu)$, namely the a.e.\ equivalence classes $[\chi_E]$  of the characteristic function of sets $E \in \Sigma$.
The projection lattice $\Pdb(\M)$ of $\M= L^\infty(X,\Sigma,\mu)$ is comprised of the above projections, although as is common in integration theory
we will sometimes blur the distinction between
$\chi_E$ and  $[\chi_E]$.
Then the {\em selfadjoint part}  $\M_{\rm sa}$ corresponds to the real valued elements of $\M$.

We say that $(X,\Sigma,\mu)$ is a
{\em Dedekind measure space} if suprema of increasing nets in $\Pdb(\M)$
always exist in $\Pdb(\M)$.  (Some other equivalent conditions for a measure space
to be Dedekind are listed in the paragraphs below
 Theorem \ref{ded}.)
Since  $\Pdb(\M)$ is closed under suprema of any {\em finite} subset of $\Pdb(\M)$, the same is
true for any  subset of $\Pdb(\M)$ (by taking increasing nets of finite subsets).
 A measure space is {\em localizable} if it is  semifinite and Dedekind.   We will give several characterizations of these later.   It is interesting that it is not
 so easy to find examples of nonlocalizable measure space (for example all finite and $\sigma$-finite measure spaces are
localizable).

We will often `decompose' a measure space $(X,\Sigma,\mu)$  into pieces with finite measure. Parts (iii) and (iv) of Theorem \ref{loc} contains a preliminary such `decomposition result' for semifinite Dedekind measure spaces.    It is partly for this reason that such spaces
are called localizable. 
For now we describe the `disjoint sum' of measure spaces
$(\Omega_i, \Sigma_i,\mu_i)$.
Let $\Omega$ be the disjoint union of the sets $\Omega_i$.  Define a $\sigma$-algebra $\Sigma$ on $\Omega$ by
$E \in \Sigma$ if and only if
$E \cap \Omega_{i} \in \Sigma_i$ for all $i$.
It is an exercise that $\Sigma$  is indeed a $\sigma$-algebra.  Define a measure $\nu$ on $\Sigma$ by
$\nu(E) = \sum_{i} \, \mu_i(E \cap \Omega_{i})$.
One may now check that  $(\Omega, \Sigma, \mu)$ is a measure space.
Indeed it is {\em decomposable} in the sense defined above
if $\mu_i(\Omega_{i}) < \infty$ for each $i$.
It is an exercise from the definition of $\Sigma$  that a scalar valued function $f$ on $\Omega$ is $\Sigma$-measurable if and only if its restriction to $\Omega_{i}$ is measurable on $\Omega_{i}$ for all $i$ (i.e.\ if and only if
$f$ is {\em locally measurable}).

We define a measure space $(X,\mu)$ to be
{\em dualizable} if $L^\infty(X,\mu) = L^1(X,\mu)^*$ isometrically via the canonical
map from $L^\infty(X,\mu)$ to $L^1(X,\mu)^*$. In particular, we learn in
a graduate real variable class that the latter holds
for $\sigma$-finite measures.

 In later results  (see for example Theorem \ref{frec}) we shall
prove that `dualizable', and indeed most of the assertions of the next result, are
each  {\em equivalent} to being `localizable'. 
By definition of the term `$W^*$-algebra',
$L^\infty(X,\Sigma,\mu)$ is a $W^*$-algebra for a dualizable measure space.  The interested reader will
note that then the last assertions of  the next result follow from the commutative
case of Sakai's theorem mentioned at the end of the introduction. 
We will not however use Sakai's theorem; instead we give a direct proof.

\begin{theorem} \label{lsigf1}  Let $(X,\Sigma,\mu)$  be a  dualizable measure space.
Then   $\M = L^\infty(X,\Sigma,\mu)$ is isometrically $*$-isomorphic and
 weak* homeomorphic to a commutative von Neumann algebra on $L^2(X,\Sigma,\mu)$
(this is the canonical representation as multiplication operators: $M_f g = fg$ for $f \in \M, g \in L^2$).
 Also,  $(X,\Sigma,\mu)$  is localizable, indeed  
every bounded increasing net in $\M_+$ has a supremum in $\M$, which coincides
with the weak* limit. \end{theorem}

\begin{proof}  If $E$ is a measurable set of infinite measure then $\chi_E \neq 0$.  Hence there is a function $f \in L^1(X,\mu)$ of norm $1$
with $$0 < |\int_X \, \chi_E \,  f \, d \mu|  = |\int_E f \, d \mu| \leq \int_E |f|  \, d \mu \leq 1.$$ If $s$ is a nonnegative simple function
with $s \leq |f| \, \chi_E$, and if $F \subset E$ is one of the (by definition mutually disjoint) sets in the standard representation of $s$, then we may assume that $0 < \mu(F) <  K \int_E |f|  \, d \mu < \infty$
for a constant $K$. 
So $\mu$ is semifinite.

It is well known  that  $\mu$ is semifinite   if and only if  the canonical representation $\pi$ of $\M$ on $L^2(X,\mu)$ is
faithful.   Indeed if it is faithful and $\mu(E) = \infty$, then there must exist
a function $g \in L^2$, which we may take to be the characteristic function of a $\mu$-finite set $F$,
such that $\chi_E \, g =  \chi_E  \, \chi_F =  \chi_{E \cap F} \neq 0$.  So $0 < \mu(E \cap F) < \infty$.
Conversely suppose that $\mu$ is semifinite and $\epsilon > 0$ is given.
Then $|f| \geq \| f \|_\infty - \epsilon$ on a non-null set $E \in \Sigma$.
If $F \subset E$ with $0 < \mu(F) < \infty$ then $\| f \, \chi_F \|_2^2 \geq (\| f \|_\infty - \epsilon)^2 \|  \, \chi_F \|_2^2$.
Thus the canonical $*$-representation of $\M$ on $L^2 = L^2(X,\mu)$ is isometric.

  As we said at the end of the introduction,
elements of the predual of $B(L^2)$
 correspond to  convergent series of
pairings
 of $L^2$ functions $g_n$ and $h_n$.
  The sum $k$ of the actual product of these $L^2$ functions
 is in   $L^1$.    Since $L^\infty = (L^1)^*$,
 $\pi$ is also weak* continuous: if $f_t \to f$ weak* in $\M$ then
 $$\sum_n \, \langle \, f_t g_n , h_n \rangle = \langle f_t , k \rangle \to  \langle f , k \rangle
 = \sum_n \, \langle \, f_t g_n , h_n \rangle.$$    Hence by basic duality theory in
functional analysis,  the $*$-homomorphism
$\pi$ is a weak* homeomorphism onto its weak*  closed range.   Thus the latter  is a von Neumann algebra.

  The remaining assertions immediately follow from the corresponding facts in elementary von Neumann algebra
theory.
Alternatively these assertions may be proved  directly as an exercise, for example using the fact that
$f_t \to f$ weak*  if and only if $\int_E \, f_t \, d \mu \to \int_E \, f \, d \mu$ for every
$\mu$-finite $E \in \Sigma$.   Also $f \leq g$ in $\M_+$   if and only if $\int_E \, f \, d \mu \leq \int_E \, g \, d \mu$ for every
such  set $E$.  
See also the proof of Lemma \ref{WStone} for 
yet another route.
\end{proof}

Given a 
 measure space $(X,\Sigma,\mu)$, the measure $\mu$ clearly induces a function $\bar{\mu}$ on the projection lattice $\Pdb(\M)$ of $\M=L^\infty(X,\Sigma,\mu)$.
Namely, if $p\in \Pdb(\M)$ corresponds to  $\chi_E$ where $E\in \Sigma$,  simply define $\bar{\mu}(p) =\mu(E)$. We will not use this, but it is useful to note that one may `recover' the $L^p$ spaces from the function $\bar{\mu}$ on $\Pdb(L^\infty(X,\mu))$, 
for $1 \leq p \leq \infty$.
For example for a simple function $s = \sum_{k=1}^n \, c_k \, p_k$ with $(p_k)$ mutually
orthogonal projections, $\| s \|_1 = \| |s| \|_1 = \sum_{k=1}^n \, |c_k| \, \bar{\mu}(p_k)$.
And  $L^1(X,\mu)$ is the completion of such simple functions in this norm.  (See Proposition \ref{Wextn}  for a related result.) This makes the next result
almost formal:

 \begin{lemma} \label{isdual} Let $(X,\mu)$ and $(Y,\nu)$ be measure spaces with $(Y,\nu)$ dualizable.
 Suppose that $L^\infty(X,\mu) \cong L^\infty(Y,\nu)$  via
 a $*$-homomorphism $\theta$ satisfying $\bar{\mu}(p) = \bar{\nu}(\theta(p))$
 for all $p \in \Pdb(\M)$.   Then $(X,\mu)$ is dualizable. \end{lemma}

\begin{proof}
The measurable simple functions in $L^p$ are norm dense in $L^p$.
By the equality above the lemma $\theta$ restricts to an isometry $\theta'$ on the space of  measurable simple integrable functions.
Hence  $L^1(X,\mu) \cong L^1(Y,\nu)$ isometrically via a (unique) continuous map $\iota$ extending
$\theta'$.  Taking the dual of $\iota$
we get a chain of isometries
$$L^\infty(X,\mu) \overset{\theta}{\cong}  L^\infty(Y,\nu)  \cong L^1(Y,\nu)^* \overset{\iota^*}{\cong} L^1(X,\mu)^* .$$
  One may check that the composition $\rho$ of these
isomorphisms is the canonical map $L^\infty(X,\mu) \to L^1(X,\mu)^*$. (Indeed by linearity and density it suffices to test
this on $\rho(q)(p)$ where $q$ (resp.\ $p$) is the characteristic function of a measurable (resp.\  $\mu$-finite measurable) set
(viewing $p  \in L^1(X,\mu)$).) \end{proof}

We shall say that a $*$-homomorphism $\theta$ satisfying the condition $\bar{\mu}(p) = \bar{\nu}(\theta(p))$
in the lemma is {\em measure preserving}.

\begin{theorem} \label{loc} Let $(X,\Sigma,\mu)$  be a measure space, and let $\M = L^\infty(X,\Sigma,\mu)$.
The following are equivalent:
\begin{itemize} \item [(i)]  $(X,\Sigma,\mu)$ is dualizable.
\item [(ii)] $(X,\Sigma,\mu)$ is localizable
(i.e.\ semifinite Dedekind). 
\item [(iii)]
 $\M \cong \oplus_i \, \M_i$  via a
  $*$-isomorphism $\theta$, with each $\M_i$ of the form $L^\infty(\Omega_i,\mu_i)$ for some finite measure
 space $(\Omega_i, \Sigma_i,\mu_i)$, and with $\bar{\mu}(p) = \sum_i \, \overline{\mu_i}(\theta(p)_i)$
 for all $p \in \Pdb(\M)$.
 \item [(iv)] There is a decomposable measure
 space $(Y,\Sigma_0,\nu)$ such that
 $\M \cong L^\infty(Y,\Sigma_0,\nu)$  via a  measure preserving $*$-isomorphism.
 \end{itemize}
 Also the  $*$-isomorphisms in {\rm (iii)} and {\rm (iv)}  are order-isomorphisms, hence  preserves suprema in
 the real part when these suprema exist.
  \end{theorem}

\begin{proof}  (ii) $\Rightarrow$ (iii) \
Let $(p_i)$ be a maximal set of orthogonal projections in $\M$ such that $\bar{\mu}(p_i) < \infty$.
 Such a maximal set exists by Zorn's lemma.  Let $p = \sup_i \, p_i$.   If $p \neq 1$ then by semifiniteness
 there exists
a projection $q \leq 1-p$ such that $0 < \bar{\mu}(q) < \infty$.  This contradiction shows that
$p = 1$.

If $E_i$ is a $\mu$-finite set in
$\Sigma$ corresponding to $p_i$ then   $\M p_i \cong L^\infty(E_i,\mu_{| E_i})$, and this is
a $W^*$-algebra since $\mu_i = \mu_{| E_i}$ is finite.

Let  $\theta: \M \to \oplus_i \, \M p_i$ be the contractive $*$-homomorphism
$\theta(x) = (x p_i)$ for $x \in \M$.
 Suppose that $x \in \M_+$ and that $0 \leq t < \| x \|$.
Then there exists  nonzero $q \in \Pdb(\M)$ with $tq \leq x$.
Since $q = \sup_i \, p_i  q \neq 0$, we have
 $$\sup_i \, \| x p_i \| \geq \sup_i \, \| x p_i q \|  \geq  \sup_i \, \| t p_i q \| = t .$$
Thus
$\theta$  is an isometry, and hence has norm closed range.

To see that $\theta$  is also surjective suppose that
 $q = (q_i)$ is a projection in $\oplus_i \, \M p_i$.
If $r = \sup_i \, q_i$ in $\Pdb(\M)$, then $p_i r
= \inf \{ p_i , r \}
= q_i$ and
  $\theta(r) = q$.
The projections in the 
$W^*$-algebra $\oplus_i \, \M p_i$ are norm densely spanning.
Alternatively, this follows from the facts that simple functions are norm dense in $L^\infty$, and 
that $\oplus_i \, \M p_i$ is an $L^{\infty}$ space as we shall prove in the next paragraph (the argument is not circular).
Thus $\pi$ has norm dense and closed range, hence is surjective.

If $q_i$ and $r$ above correspond to sets $F_i$ and $F$ in $\Sigma$ then 
$\sum_i \, \overline{\mu_i}(q_i) = \sum_i \, \mu(F_i)$, 
the partial sums of which are dominated by $\mu(F) =  \bar{\mu}(r)$.
So if the sum equals $\infty$ we are done.
If the sum is finite then all but a countable number of its terms are 
zero.
So we may assume that $\{ q_i \}$ is countable.  Then $F = \cup_i \, F_i$ a.e., and
$\sum_i \, \mu(F_i)
= \mu(F)$.
  Thus
$\sum_i \, \overline{\mu_i}(\theta(r)_i) = \bar{\mu}(r)$.

(iii) $\Rightarrow$ (iv) \ Let $(\Omega, \Sigma, \mu)$ be the  `disjoint sum' of the measure spaces
$(\Omega_i, \Sigma_i,\mu_i)$ (introduced a few paragraphs above Theorem \ref{lsigf1}). We leave it as an exercise that $L^{1}(\Omega,\Sigma,\nu) \cong \oplus_{i}^1 \,
L^{1}(\Omega_i,\mu_i)$ isometrically, where the latter is the $L^1$ direct sum of Banach spaces,
and similarly $L^{\infty}(\Omega,\Sigma,\nu) \cong \oplus_{i} \,
L^{\infty}(\Omega_i,\mu_i)$ $*$-isomorphically.
By Banach space duality, $$L^{1}(\Omega,\Sigma,\nu)^* \cong
(\oplus_{i}^1 \,
L^{1}(\Omega_i,\mu_i))^* \cong \oplus_{i} \,
L^{\infty}(\Omega_i,\mu_i)
 \cong L^{\infty}(\Omega,\Sigma,\nu).$$
 Tracing through these identifications, one can check that the resulting map $L^{\infty}(\Omega,\Sigma,\nu) \to
L^{1}(\Omega,\Sigma,\nu)^*$ is the canonical one. 
So $(\Omega,\Sigma,\nu)$ is dualizable.
Moreover, we have  $$\M \cong \oplus_{i} \, \M_{i}  \cong \oplus_{i} \,
L^{\infty}(\Omega_i,\mu_i)  \cong L^{\infty}(\Omega,\Sigma,\nu).$$
The $*$-isomorphisms here are all easily checked
 to be measure preserving.

 (iv)  $\Rightarrow$ (i) \ We may repeat the argument in last paragraph, but with $(\Omega, \Sigma, \mu)$ and $\Omega_i$
 replaced by $(Y,\Sigma_0,\nu)$ and the space $X_t$ in the decomposition guaranteed by the definition of a `decomposable measure space' above.   One sees that $(Y,\Sigma_0,\nu)$  is dualizable, and
 the $*$-isomorphisms employed in the proof again are all measure preserving.
 So (i) follows from Lemma \ref{isdual}.

Theorem \ref{lsigf1} shows that (i) implies (ii).
Finally, any $*$-isomorphism of $C^*$-algebras  is an order-isomorphism, and hence preserves suprema.
\end{proof}

\begin{example} \label{nonl}  Perhaps the simplest example of a semifinite but non-Dedekind (so non-localizable)
measure, is the following:
Let $X$ be an uncountably infinite set. Let $\Sigma$ be the set
of subsets of $X$ which are either countable or their complement is
countable, and let $\mu$ be  counting measure.
Then $(X,\Sigma,\mu)$ is a complete semifinite measure space, but it is not Dedekind.
 Indeed if $E$ is an uncountable  subset with uncountable complement
 then the countable subsets of $E$ have no supremum in $\Sigma$,
 and the matching projections in $L^\infty(X,\mu)$ have no supremum there.
 This measure `has support' $X$, in the sense that the empty set is the largest set in $\Sigma$ of measure zero .
 (This example will also be interesting to a subset of our readers in that if ${\mathcal P}$ is the power set of
 $X$, then $\M = L^\infty(X,{\mathcal P},\mu)$ contains $\N = L^\infty(X,\Sigma,\mu)$
 as a $C^*$-subalgebra, but there is no
 measure-preserving conditional
 expectation of $\M$ onto $\N$.)

Another semifinite non-Dedekind space may be found in  216D in \cite{Fremlin}.   
Note that Fubini's theorem is not valid for localizable measure spaces 
(see e.g.\ 252K and Volume 3 of \cite{Fremlin}), 
although there is a good variant for Radon measure spaces  (see e.g.\ \cite[Volume 4]{Fremlin}). 
See also \cite{OR} for some other classes of localizable measures and examples of such.  \end{example}

\section{Weights and measures} \label{wol}

 In this section all measures are positive measures unless stated to the contrary, and
  $\M = L^\infty(X,\Sigma,\mu)$ for a measure space $(X,\Sigma,\mu)$.
  For a measure
  $\nu \ll \mu$ on $(X,\Sigma)$ we write $\psi_\nu(f) = \int_X \, f \, d \nu$ for all $f \in \M_+$.

A {\em weight} on a
 $C^*$-algebra $N$ is a $[0,\infty]$-valued map on $N_+$ which is
additive, and satisfies $\omega(tx) = t \omega(x)$ for $t \in [0,\infty]$ (interpreting $0\cdot \infty = 0$).
We say that a weight $\omega$ on $N_+$ is {\em normal}
(resp.\ $\sigma$-{\em normal})
if $\omega(\sup_t \, x_t) = \sup_t \, \omega(x_t)$ whenever
$(x_t)$ is an increasing net (resp.\ sequence) in $N_+$ which possesses a supremum in $N_+$.
We say that  $\omega$ is
{\em completely additive} (resp.\ {\em countably additive} or $\sigma$-{\em additive})
if  whenever we have a set  (resp.\ countably indexed set) $\{ f_i : i \in I \}$
 in $N_+$  with the partial sums of $\sum_i \, f_i$ uniformly bounded
and having a supremum $f$ in $N$,
then $\omega(f) = \sum_i \,\omega (f_i)$.
We say that  $\omega$ is {\em completely additive on projections} if the latter holds
when  in addition $\{ f_i : i \in I \}$ are mutually orthogonal projections.
 Clearly a normal weight is $\sigma$-normal and completely additive,
 and a completely additive weight is $\sigma$-additive and is completely additive on projections.
 If $\M$ is a $W^*$-algebra then it is well known that functionals on $\M$ are normal  if and only if they
are weak* continuous (for example see Theorem \ref{abW}),
 but we will not use this
until after it has been proved.

If $\nu$ is a measure  on $\Sigma$ with $\nu \ll \mu$ then $\psi_\nu(f) = \int \, f \, d \nu$
is a $\sigma$-additive weight  on $\M_+$.
It is then an elementary exercise to show that $\nu$
is a semifinite measure if and only if  $\psi = \psi_\nu$
is a {\em semifinite weight} on $\M_+$, in the sense that whenever $\psi(p) = \infty$ for a projection $p$ in $\M$,
then $p$ has a subprojection $q$ in $\M$ with $0 < \psi(q) < \infty$.   For example, if the latter holds and if $\nu(E) = \infty$ then $\nu(E) = \psi (\chi_E ) = \infty$, so that there exists a projection $q = \chi_F \leq \chi_E$
 with $0 \neq \nu(F) = \psi ( \chi_F ) < \infty.$  So $\nu$ is semifinite.

 Since
the pointwise limit $f$ of an increasing  bounded sequence  of positive measurable functions
on $(X,\A,\mu)$  is measurable, $f$ clearly
corresponds to the supremum  of the sequence in $\M_{\rm sa}$.
 Note that a weight $\psi$ on $\M_+$ is countably additive if and only if  it is $\sigma$-normal: if
  $f_n \nearrow f \in \M_+$
   then
 $\sum_{n=0}^\infty \, (f_{n+1} - f_n) = f$,
 where $f_0 = 0$.   So
 countably additivity gives $\sum_{n=0}^\infty \, \psi(f_{n+1} - f_n) = \psi(f)  = \lim_n \, \psi(f_n)$.
 The converse is easier.

 If $(X,\Sigma,\mu)$ is dualizable/localizable  and $(f_t)$ is an increasing net in $\M_+$
then $f_t \nearrow f$  (that is, $f$ is the supremum in $\M_+$ of  $(f_t)$)
if and only if $f_t \to f$ weak*, by Theorem \ref{lsigf1}.

The following is the simple abelian case of part of Haagerup's characterization of normal weights \cite{haag-Nw}:

\begin{lemma} \label{lsigf0}   Suppose that  $(X,\Sigma,\mu)$ is localizable,
 and $\M = L^\infty(X,\Sigma,\mu)$.
 A weight $\omega$ on $\M_+$
is completely  additive if and only if $\omega$ is normal.
\end{lemma}

\begin{proof}
One direction of our lemma is obvious by applying `normality' to the increasing net of partial sums.
For the other direction, by Theorem \ref{loc}  (ii) we may assume that $\M$ is a direct sum
of `finite' pieces $\M p_i$ for $i \in I$, each of the form  $L^\infty(X,\nu)$ for a finite measure $\nu$.
On $\N = L^\infty(X,\nu)$ for a finite measure $\nu$, we claim that any $\sigma$-additive weight $\omega$ is normal.
Indeed suppose that $f_t \nearrow f$ in $\N_+$.  Then $\psi_\nu(f_t) \nearrow \psi_\nu(f)$ by the last assertion of Theorem \ref{lsigf1} (note
$1 \in L^1(X,\Sigma,\nu)$ so that $\psi_\nu$
is weak* continuous).
Choose  $t_1 \leq t_2 \leq \cdots$ with
$\psi_\nu(f_{t_k}) \nearrow \psi_\nu(f)$.  Setting  $g = \sup_k \, f_{t_k}$ in $\N_+$ we have $\psi_\nu(g - f) = 0$, so that $f = g$.
Then $f_{t_{1}} + \sum_k \, (f_{t_{k+1}}  - f_{t_{k}}) = f$, so that
$$\omega(f) = \omega(f_{t_1}) + \sum_k \, \omega(f_{t_{k+1}}  - f_{t_{k}}) = \sup_k \, \omega(f_{t_k}) = \sup_t \, \omega(f_{t}). $$
 Thus $\omega$ is normal.

Finally, if $f_t \nearrow f$ in $\M_+$
then, using complete  additivity twice, we have $$\sup_t \omega(f_t) = \sup_t  \, \sup_J  \sum_{i \in J} \, \omega(p_i \, f_t) = \sup_J  \sum_{i \in J} \, \omega(p_i \, f) = \omega(f) .$$
Here $J$ ranges over finite subsets of $I$.
\end{proof}

We write $\mu \cong \nu$ if $\nu  \ll  \mu$ and $\mu \ll \nu$; and say the measures {\em are equivalent}.

   \begin{lemma} \label{tol}  Set $\M = L^\infty(X,\mu)$ for a measure space $(X,\Sigma,\mu)$.
   If $\omega$  is a $\sigma$-normal
 (i.e.\ countably additive) weight on $\M_+$
 then there exists a measure $\nu$ on $(X,\Sigma)$ such that $\nu  \ll  \mu$ and
 $\omega(f) = \int_X \, f \, d \nu$ for all $f \in \M_+$.

 Indeed the above implements a bijection between $\sigma$-additive weights $\psi$ on $\M_+$ and measures $\nu$ on $(X,\Sigma)$
 such that $\nu  \ll  \mu$.
 Finally, $\psi_\nu$ is faithful on $\M_+$  if and only if $\mu \ll \nu$ also (so that $\mu \cong \nu$), and in this case
 $\psi_\nu$ is completely additive.
  \end{lemma}

   \begin{proof}  If $\omega$  is  $\sigma$-normal 
   (or even just countably additive on mutually orthogonal projections)
   then we may define a
   measure on $(X,\Sigma)$ by $\nu(E) = \omega(\chi_E)$.   If  $\{ E_k : k \in \Ndb \} \subset \Sigma$ are mutually
 disjoint with union $E$ then $\sum_k \, \chi_{E_k} = \chi_{E}$ pointwise,
 hence $\chi_{E}$
 is the a.e.\ supremum of the partial sums, so that
 $$\sum_k \, \nu(E_k) = \sum_k \, \omega(\chi_{E_k}) =  \omega(\chi_E) = \nu(E).$$
 Hence $\nu$ is a measure.
 It is easy to see that  $\nu  \ll  \mu$.     Let $\psi(f) = \int_X \, f \, d \nu$ for $f \in \M_+$.  Note that this is well defined since
 if $f, g  : X \to [0,\infty]$ are measurable but differ on a
 $\mu$-null (hence $\nu$-null) set then $\psi(f) = \psi(g)$.    Then $\psi = \omega$ on projections in $\M$, hence on
 simple functions in $\M_+$.   If $f \in \M_+$ and $0 \leq s_n \nearrow f$ pointwise a.e.\  then by Lebesgue's monotone convergence theorem we have that $\psi(s_n) = \omega(s_n) \nearrow \psi(f)$.
 Thus $\psi = \omega$ on $\M_+$.

 Conversely suppose that $\nu$ is another measure on $(X,\Sigma)$
 with $\nu  \ll  \mu$.     Let $\psi(f) = \int_X \, f \, d \nu$.   If $f, g  : X \to [0,\infty]$ are measurable but differ on a
 $\mu$-null set then $\psi(f) = \psi(g)$.   Thus $\psi$ may be viewed as a weight on $\M_+$, and by Lebesgue's monotone convergence theorem it is $\sigma$-normal
 and countably additive on $\M_+$.

  Next, note that $\psi$ is faithful on $\M_+$  if and only if $\mu \ll \nu$ also (so that $\mu \cong \nu$).
 Indeed suppose that also $\mu \ll \nu$, so that $\M = L^\infty(X,\nu)$. Now suppose that  $f \in \M_+$ with $\psi(f) = 0$.
 If $f \neq 0$ in $\M$ then there exists a nonzero simple function $s \leq f$ (because there exist simple functions $0 \leq s_n \nearrow f$
 with $\int_X \, s_n \, d \mu \to \int_X \, f \, d \mu \neq 0$).  Thus there exists a non $\mu$-null measurable subset $E$ of $F$ and $c > 0$ with
 $c \, \chi_E \leq f$,
 and $\nu(E) > 0$ since $\mu \ll \nu$.  Whence
 the contradiction $\psi(f) \geq c \, \psi(\chi_E) = c \, \nu(E) > 0$.   Thus $\psi$ is faithful on $\M_+$.

 Suppose now that $\psi$ is faithful on $\M_+$.  We show that
  $\psi$ is completely additive on $\M_+$.
  Suppose that $f_i$ are measurable nonnegative functions on $X$ with the partial sums of $\sum_i \, f_i$ uniformly bounded,
and suppose that this net of partial sums has supremum $g$ in $\M_{\rm sa}$.  We claim that
$\psi(g) = \sum_i \,\psi( f_i)$. Clearly $\sum_i \,\psi( f_i) \leq \psi(g)$.
Hence if $\sum_i \,\psi( f_i) = \infty$ we are done.
If $\sum_i \,\psi( f_i) < \infty$ then the sum is countable.   In this case it is easy to see
that $g$ may be represented as a measurable function on $X$ which is the function supremum of the partial sums, that
is the pointwise limit.  The claim now follows from Lebesgue's monotone convergence theorem.  \end{proof}

 \begin{corollary} \label{intno}   Suppose that  $(X,\Sigma,\mu)$ is localizable,
 and $\M = L^\infty(X,\Sigma,\mu)$.   Then  integration with respect to $\mu$  is a normal weight on $\M$.
\end{corollary}

\begin{proof}  This follows  from Lemmas \ref{lsigf0} and (the last assertion of) \ref{tol}.
 \end{proof}

 \section{The Radon-Nikodym theorem} \label{rns}
No analysis of measure spaces for which $L^\infty(X,\mu)=L^1(X,\mu)^*$ is complete without giving due attention
to the Radon-Nikodym theorem. However to do justice to the task initiated in Section \ref{LinftyW1},
we will want the most general Radon-Nikodym theorem achievable for semifinite or Dedekind measure spaces. This then is the focus of this section.

In this section again all measures are positive measures unless stated to the contrary, and
  $\M = L^\infty(X,\Sigma,\mu)$ for a measure space $(X,\Sigma,\mu)$.  We consider
  measures $\nu$ on $(X,\Sigma)$, and
 $\nu \ll \mu$
 signifies {\em absolute continuity}: $\nu(E) = 0$ whenever $\mu(E) = 0$.    As we learn in a real variables course, the complex
 or signed measure case of the Radon-Nikodym theorem follows immediately from the positive
 measure case, using the Hahn decomposition or via the variation measure $|\nu|$.

Fix a measure space $(X,\Sigma,\mu)$.  We will say that a measure $\nu$  on $(X,\Sigma)$ {\em has $\mu$-support}  if there is a set $E_0 \in \Sigma$ such that
$\nu(E) = \nu(E \cap E_0)$ for all $E \in \Sigma$, and $\nu(E) > 0$ for all $E \in \Sigma$ with $\mu(E) > 0$ and $E \subset E_0$.
We call $E_0$ the $\mu$-support of $\nu$ in this case.   If it exists then it is unique up to $\mu$-null sets.
Indeed if $E_1$ was another $\mu$-support of $\nu$, and if  for example 
$\mu(E_0 \setminus E_1) \neq 0$,
then $\nu(E_0  \setminus E_1) > 0$  while  $\nu(E_0  \setminus E_1) = \nu((E_0  \setminus E_1) \cap E_1) = 0$.
Indeed if further $\nu \ll \mu$ then adding or subtracting a $\mu$-null (hence $\nu$-null) set to or from $E_0$ will not change it being a $\mu$-support.

\medskip

{\bf Examples.} If $\lambda$ is Lebesgue measure and $\mu$ is counting measure on the Lebesgue $\sigma$-algebra of $[0,1]$ then $\lambda$   satisfies
$\lambda \ll \mu$ but $\lambda$ does not have $\mu$-support.   Any measure $\nu$ has $\nu$-support.
\medskip

Given measures $\mu, \nu$ on a $\sigma$-algebra $\Sigma$, we say that $\nu$ is $\mu$-{\em semifinite}
if whenever $\nu(E) > 0$ then $E$ has a $\mu$-finite measurable subset $F$ with $\nu(F) > 0$.
This condition was considered in some of the Radon-Nikodym theorems
in  \cite[Volume II]{Fremlin}.    It is automatic if $\mu$ is finite, and coincides
with semifiniteness if $\mu = \nu$.
We say that $\nu$ is {\em strongly} $\mu$-{\em semifinite} if  whenever $\nu(E) > 0$ then $E$ has a $\mu$-finite measurable subset $F$ with $0 <
\nu(F) < \infty$.   One may check that $\nu$ is strongly $\mu$-semifinite if and only if $\nu$ is $\mu$-semifinite and semifinite.

Below we write $\mu_f$ for $\mu_f(E) = \int_E \, f \, d \mu$.

\begin{lemma} \label{fis} Fix a measure space $(X,\Sigma,\mu)$.
\begin{itemize} \item [(1)]  If $\mu(X) < \infty$ then every
measure  $\nu$  on $(X,\Sigma)$ with $\nu \ll \mu$ has $\mu$-support.
\item [(2)]  If $f : X \to [0,\infty]$ is measurable then $\mu_f$ has $\mu$-support.
\end{itemize}
\end{lemma}
\begin{proof}   (2) \ Let $E_0 = f^{-1}((0,\infty])$.   If
 $f = 0$ $\mu$-a.e.\ on a set $E$, then $\mu_f(E) = 0$.
In particular $\mu_f(E_0^c) = 0$.
If $E \in \Sigma$ with $\mu(E) > 0$ and  $E \subset E_0$, that is if $f > 0$ on
$E$ $\mu$-a.e.,  then $\mu_f(E) > 0$ by the `vanishing principle' in measure theory.
So $\mu_f$ has $\mu$-support.

(1) \ This  can be seen  in many ways.   For example it follows from (2) by a well known Radon-Nikodym theorem (often met in graduate real variable classes \cite[Exercise 3.14]{Fol}) for finite
$\mu$.     It can also be proved using
Lemma \ref{lsigf0} and the fact that any mutually orthogonal collection of nonzero projections in $\M = L^\infty(X,\mu)$ must be countable, alongside for example Lemma \ref{mus} (2).    \end{proof}

\begin{lemma} \label{mu}   Suppose that  $(X,\Sigma,\mu)$ is a Dedekind
measure space and  $\M = L^\infty(X,\Sigma,\mu)$. 
\begin{itemize} \item [(1)]  There exists $S \in \Sigma$ such that $(S, \mu_{|S})$ is semifinite and
the projection lattice of $L^\infty(S, \mu_{|S})$ is Dedekind complete,
and such that $\mu(E) = \infty$ for any measurable  subset $E$ of $X \setminus S$ with $\mu(E) > 0$.
 \item [(2)]    Let $\nu$  be a measure on $(X,\Sigma)$ with $\nu \ll \mu$.
If $\nu$ is $\mu$-semifinite then $\nu$ has $\mu$-support.  Moreover $\nu$ is $\mu_{|S}$-semifinite on
the set $S$ in  {\rm(1),} and
$\nu(E) = 0$ for any  measurable subset $E$ of $X \setminus S$ \end{itemize}
\end{lemma}

\begin{proof}  We define the {\em essential supremum} of a collection ${\mathcal C}$ in $\Sigma$ to be a set
$B$ in $\Sigma$ (the equivalence class of) whose characteristic function is the supremum in $\Pdb(\M)$ of $\{ \chi_E : E \in {\mathcal C} \}$.  The essential supremum is clearly unique `up to measure zero': for any other  essential supremum $D$ we have $\mu(C \Delta D) = 0$.
 (1) \ The collection of $\mu$-finite sets that are not $\mu$-null
 have an essential supremum $S$ say.
We claim that $(S, \mu_{|S})$ is semifinite. 
  Indeed suppose that
$E \subset S$ with $\mu(E) = \infty$. Suppose that $E$ contained no $\mu$-finite sets that are not $\mu$-null.
Then $S \setminus E$  contains ($\mu$-a.e.)\ all $\mu$-finite sets that are not $\mu$-null, contradicting that
$S$ is the essential supremum.  So $E$ must contain a set $F \in \Sigma$ with $0 < \mu(F) < \infty$.  This
proves that $(S, \mu_{|S})$ is semifinite. 
The projection lattice assertion holds by the ideas above, since an essential supremum  in $X$ of any collection of measurable subsets of $S$, is $\mu$-a.e.\ contained in $S$.

For any measurable subset $E$ of $X \setminus S$ with $\mu(E) > 0$, we must have
$\mu(E) = \infty$ (or else $E \subset S$ $\mu$-a.e.).

(2) \ The set of $\nu$-null sets in $\Sigma$  that are not $\mu$-null have an essential supremum
$B \in \Sigma$.   Let $E_0 = B^c$.      We will show that $E_0$ is a $\mu$-support for $\nu$.
 If $E \in \Sigma$ with $\mu(E) > 0$ and $E \subset E_0$  then $\nu(E) > 0$ (or else $E \subset E_0^c$
up to a $\mu$-null set, that is $\mu(E \cap E_0) = \mu(E) = 0$).   If $\nu(B) = 0$ then $B^c$ is a $\mu$-support.
By way of contradiction suppose that $\nu(B) > 0$.
Then by hypothesis $B$ has a $\mu$-finite measurable subset $F$ with $\nu(F) > 0$. 
Now $(F, \mu_{|F})$ is a finite measure space, and $\nu_{|F}$ is a  measure with
$\nu_{|F} \ll \mu_{|F}$, and we may appeal to Lemma \ref{fis} to see that $\nu_{|F}$ has $\mu$-support in $F$.

Suppose that $G$ is the  $\mu$-support of $\nu_{|F}$ in $F$.   We must have $\mu(G) > 0$ or else
$\nu(G) = \nu(F \setminus G) = 0$.
Suppose that $H \subset F$ is a measurable set $\mu$-a.e.\ containing every $\nu$-null subset of $F$.
 Then $H \cup (B \setminus F)$ is a measurable set $\mu$-a.e.\ containing every $\nu$-null subset $T$ of $X$, since $T = (T \cap F) \cup (T \setminus F)$.
 So $B \subset H \cup (B \setminus F)$ $\mu$-a.e., and so $F = F \cap B \subset H$ $\mu$-a.e., and $\mu(F \setminus H) = 0$.
 Now set $H = F \setminus G$.   Every $\nu$-null subset $E$ of $F$  must have $\mu(E \cap G) = 0$, or else $\nu(E \cap G) > 0$.
 Hence $E \subset F  \setminus G = H$  $\mu$-a.e.  It follows that $\mu(F \setminus H) = \mu(G) = 0$, a contradiction.
 Thus  $\nu(B) = 0$, so that   $E_0$ is a $\mu$-support for $\mu$.

If $\nu \ll \mu$ and $S$ is as above, then $\nu$ is $\mu_{|S}$-semifinite on $S$.   Indeed if $E \subset S$
with $\nu(E) > 0$ then $E$ has a $\mu$-finite measurable subset $F$ with $\nu(F) > 0$.

Suppose that $E$ is a subset of the `anti-semifinite part' $S^c$ of $\mu$, with $\mu(E) = \infty$.   If $\nu(E) > 0$ then by
$\mu$-semifiniteness we obtain the
contradiction that   $E$ has a $\mu$-finite, hence $\mu$-null, measurable subset $F$ with $\nu(F) > 0$.
 \end{proof}

We call $(S, \mu_{|S})$ the {\em semifinite part}  of  $\mu$, and $(S^c, \mu_{|S^c})$ the {\em anti-semifinite part}.
It is unique up to $\mu$-null sets.    Indeed suppose that $A \in \Sigma$ is another, with  $(A, \mu_{|A})$  semifinite
and $\mu(E) = \infty$ for any measurable  subset $E$ of $X \setminus A$ with $\mu(E) > 0$.
If $\mu(S \setminus A) > 0$ then any measurable  subset $E$ of $S \setminus A$ with $\mu(E) >0$, satisfies
$\mu(E) = \infty$.  This contradicts that $(S, \mu_{|S})$ is semifinite.   So $\mu(S \setminus A) = 0$ and
similarly $\mu(A \setminus S) = 0$.

Note that (2) gives that $(S, \nu_{|S})$ is the `$\mu$-semifinite part'  of  $\nu$.

\begin{lemma} \label{mus}  Fix a measure space $(X,\Sigma,\mu)$
and a measure $\nu$  on $(X,\Sigma)$ with $\nu \ll \mu$.  Set $\M = L^\infty(X,\Sigma,\mu)$.
\begin{itemize} \item [(1)]  If $\nu$ has $\mu$-support then  $\psi_\nu$ is completely additive
 on $\M_+$, and indeed is normal if $(X,\Sigma,\mu)$ is localizable.
 \item [(2)]  If  $(X,\Sigma,\mu)$ is localizable and $\psi_\nu$ is completely additive on $\M_+$ then $\nu$ has $\mu$-support.
\item [(3)]  If  $\mu$ is semifinite
and $\nu$ has $\mu$-support then $\nu$ is $\mu$-semifinite.
\end{itemize}
  \end{lemma}
\begin{proof} (1) \ Let $E_0$ be the $\mu$-support of $\nu$.
Note that $(E_0,\mu_{| E_0})$ is a   measure space with the property that if $\nu(F) = 0$
for a measurable subset $F$ of $E_0$ then $\mu(F) = 0$.
Thus $\mu_{| E_0} \cong \nu_{| E_0}$, so that by Lemma \ref{tol}, $\psi_0(f) =   \int_{E_0} \, f \, d \nu$ is
completely  additive on $\N = L^\infty(E_0,\mu_{| E_0}) = \M e$ where $e = \chi_{E_0}$.
Moreover $\psi_\nu(f) = \psi_0(f e)$, since $\nu$ is concentrated on $E_0$.

To see that  $\psi_\nu$ is completely  additive,
suppose that $f_i \in \M$ with the partial sums of $\sum_i \, f_i$ uniformly bounded,
and suppose that this net of partial sums has supremum $g \in \M$.
We claim that  $ge$ is the supremum in $\M e$  of the partial sums of $\sum_i \, e f_i$.
Indeed if $k \in \M e$ with $k \geq \sum_{i \in J} \, e f_i$ for finite sets $J$, then
$k + g e^\perp \geq \sum_{i \in J} \, f_i$.  So $k + g e^\perp \geq g$, and so $k  \geq g e$.
This proves the claim.   It follows that $$\sum_i \, \psi_\nu (f_i) =
\sum_i \, \psi_0(f_i e) = \psi_0(ge) = \sum_i \, \psi_\nu (g) .$$ So $\psi_\nu$ is completely  additive.
Thus by Lemma \ref{lsigf0} $\psi_\nu$ is normal if $(X,\Sigma,\mu)$ is localizable.

(2) \ Conversely, if  $(X,\Sigma,\mu)$ is localizable and $\psi = \psi_\nu$ is completely additive on $\M_+$ then $\psi$ is normal by Lemma \ref{lsigf0}.    If $p$ is the supremum
of the projections in $\M_+ \cap {\rm Ker}(\psi)$ then $\psi(p) = 0$.
Then $q = 1-p$ is the support projection of  $\psi$.
Indeed if $p$ is a nonzero subprojection of $q$ in $\M$ then
$\psi(p) > 0$.   Choosing $E \in \Sigma$ with  $q = \chi_E$ it is an exercise to check that
$E$ is a  $\mu$-support of $\nu$.

(3) \  If $\nu$ has $\mu$-support $E_0$ and  $\mu$ is semifinite  and $\nu(E) > 0$ then
$\nu(E \cap E_0) > 0$.   Hence $\mu(E  \cap E_0) > 0$, so that
 $E  \cap E_0$ has a $\mu$-finite measurable subset $F$ with $\mu(F) > 0$.     Hence $\nu(F) > 0$ by definition of $\mu$-support.
 So $\nu$ is $\mu$-semifinite.
  \end{proof}

The equivalence of (i) and (iii) below is essentially 234O in \cite{Fremlin}, but with a quite different proof.

 \begin{theorem} \label{rn} {\rm (Radon-Nikodym for localizable
 measures I)} \ Fix a localizable
 measure space $(X,\Sigma,\mu)$  and a positive measure $\nu$  on $(X,\Sigma)$.  Set $\M = L^\infty(X,\Sigma,\mu)$.
  The following are equivalent:
\begin{itemize} \item [(i)]
There exists measurable $h : X \to [0,\infty)$  with $\nu(E) = \int_E \, h \, d \mu$ for all $E \in \Sigma$.
\item [(ii)]
$\nu \ll \mu$ and $\nu$ is semifinite and has $\mu$-support.
\item [(iii)]
$\nu \ll \mu$ and  $\nu$ is strongly $\mu$-semifinite (that is, whenever $\nu(E) > 0$ then $E$ has a $\mu$-finite measurable subset $F$ with $0 <
\nu(F) < \infty$).
\item [(iv)]  $\nu \ll \mu$ and $\psi_\nu$ is a normal semifinite weight  on $\M_+$.  \end{itemize}
Indeed every normal semifinite weight on $\M_+$  is of the form in {\rm (iv)}.
The function $h$ in {\rm (i)} is  unique up to $\mu$-a.e.\ equality.
Moreover if these hold then $\mu \ll \nu$ if and only if
$h$ in {\rm (i)} may be chosen to be strictly positive.
\end{theorem}

 \begin{proof}    (i) $\Rightarrow$ (ii) \
 Note that $h^{-1}((0,\infty])$ is the $\mu$-support by Lemma \ref{fis}.  It is well known and an exercise that $\mu_h(E) = \int_E \, h \, d \mu$ is semifinite if  $\mu$ is.

 (ii) $\Leftrightarrow$ (iii) \ This follows from Lemmas  \ref{mu} (2) and \ref{mus} (3), plus the fact that $\nu$ is strongly $\mu$-semifinite if and only if $\nu$ is $\mu$-semifinite and semifinite.

(ii) $\Rightarrow$ (iv) \  Suppose that $\nu \ll \mu$ and $\nu$ has $\mu$-support.
 Then $\psi = \psi_\nu$ is normal
 by Lemma \ref{mus} (1).
  Moreover, if $\nu$ is semifinite then $\psi$ is  a semifinite weight by a remark in the
third paragraph  of Section \ref{wol}.

 (iv) $\Rightarrow$ (ii) \  By the same remark just alluded to, $\nu$ is semifinite.
 The $\mu$-support exists by Lemma \ref{mus} (2).

 (iii) $\Rightarrow$ (i) \   If $\nu \leq \mu$, then integration with respect to $d \nu$ is a bounded linear functional on $L^1(\mu)$ since
$$| \int \, f \, d \nu | \leq  \int \, |f| \, d \nu  \leq \int \, |f| \, d \mu = \| f \|_{L^1(\mu)}.$$
Thus there exists $h \in L^\infty(\mu)$ with $$ \int \, f \, d \nu =  \int \, f h \, d \mu, \qquad f \in L^1(\mu) . $$
Then $\nu(E) = \psi_\nu (\chi_E) =  \int_E \,  h \, d \mu$ for
all $\mu$-finite $E \in \Sigma$.
Since this is positive and dominated by $\int_E \,  1 \, d \mu$,
a standard argument in elementary measure theory using the semifiniteness of $\mu$ 
(the kind of 
argument found for example in the last paragraph of the 
proof of Theorem \ref{frec} below, where it is shown that $h = 1$ a.e.\ there), 
shows that $0 \leq h \leq 1$ $\mu$-a.e.
If we assume that $\nu$ is semifinite and $\mu$-semifinite, then $\psi_\nu$ and $\psi_{\mu_h}$ are normal by the fact that (iv) in Theorem \ref{rn} is implied by the other conditions.
The semifiniteness now ensures that any $p \in \Pdb(\M)$ is a
 supremum
 of $\bar{\mu}$-finite projections.
It follows that $\nu =  \mu_h$ on  $\Pdb(\M)$.
Thus $\nu(E) = \mu_h(E)$ for all  $E \in \Sigma$.

In the general case if $\nu \ll \mu$ then we first argue that $(X,\Sigma,\mu + \nu)$ is localizable.
We are  assuming that $\nu$ is
strongly $\mu$-semifinite.   If $\nu(E) + \mu(E) > 0$ then $\mu(E) > 0$, and
so there is a set $F \subset E$ with $0 < \mu(F) < \infty$.   If $\nu(F) = 0$  then
 $0 < \nu(F) + \mu(F) < \infty$.      If $\nu(F) > 0$ then there is a subset $F' \subset F$ with $0 < \nu(F') < \infty$,
 in which case  $0 < \nu(F') + \mu(F') < \infty$.  So $\mu + \nu$ is semifinite.
 Note that $\mu$ and $\mu + \nu$ have the same null sets, so the corresponding projection lattices are equal.
 Thus $(X,\Sigma,\mu + \nu)$ is localizable.

 We have $\mu + \nu \ll \mu \ll \mu + \nu$. We finish in the standard way:
 First note that $\mu$ is $(\mu+\nu)$-semifinite: indeed if $\mu(E) > 0$ then
 there is a set $F \subset E$ with $0 < \mu(F) < \infty$.  If $\nu(F) = 0$ then $\mu(F) + \nu(F) < \infty$.
 If $\nu(F) > 0$ then there is a further $\mu$-finite subset $G \subset F$ with $0 < \nu(G) < \infty$
 so that $\mu(G) + \nu(G) < \infty$ and $\mu(G) > 0$.
  By the second last paragraph there
 exists $k : X \to [0,1]$ measurable with $\mu(E)  = \int_E \, k \, d (\mu + \nu)$.   Clearly $k> 0$ $\mu$-a.e., so we may suppose that
 $k$ is strictly positive.  Then
 $$\int_E \, k^{-1}  \, d \mu = \int_E \, k^{-1}  \, k \, d (\mu + \nu) = (\mu + \nu)(E) , \qquad E \in \Sigma .$$

 Note that $\nu$ is  $(\mu+\nu)$-semifinite: indeed if $\nu(E) > 0$ then
 there is a set $F \subset E$ with $0 < \nu(F) < \infty$,
 and there is a further $\mu$-finite subset $G \subset F$ with $\nu(G) > 0$.
 Clearly $\mu(G) + \nu(G) < \infty$.
  By the argument above there 
 exists $g : X \to [0,1]$ measurable with $\nu(E)  = \int_E \, g \, d (\mu + \nu)$.     Then
 $$\nu(E)  = \int_E \, g \, d (\mu + \nu) = \int_E \, g \, k^{-1}  \, d \mu , \qquad E \in \Sigma .$$
 So if we set $h = g \, k^{-1}$ then $d \nu = h d \mu$.

 The first of the final assertions follows from Lemma \ref{tol},
and the second is standard for semifinite measures.
 The last assertion is easy, because if for example $h$ is zero on a non-$\mu$-null set $E$, then
$\nu(E) = 0 \neq \mu(E)$.
 \end{proof}

The main part of the  last result  is simultaneously equivalent to Fremlin's Radon-Nikodym theorem  234O
in \cite{Fremlin}, and the Pedersen-Takesaki Radon-Nikodym theorem from \cite{PT} in the case of
 commutative von Neumann algebras.   Thus it shows that the latter two results are equivalent.
 Of course we need to know that the commutative von Neumann algebras correspond to localizable measures
 (see Theorems \ref{lsigf1} and \ref{frec}, and Corollary \ref{awvn}).    We also remark that some Radon-Nikodym theorems on certain measure algebras appear in connection with
 Lamperti's theorem in some recent papers of Gardella and Thiel, see e.g.\ \cite{GT}.

We have not seen the following result explicitly in the literature:

 \begin{corollary} \label{rn2} {\rm (Radon-Nikodym for localizable
 measures II)} \ Fix a localizable
 measure space $(X,\Sigma,\mu)$  and a positive measure $\nu$  on $(X,\Sigma)$.  Set $\M = L^\infty(X,\Sigma,\mu)$.
 The following are equivalent:
\begin{itemize} \item [(i)]
There exists a measurable $h : X \to [0,\infty]$  with $\nu(E) = \int_E \, h \, d \mu$ for all $E \in \Sigma$.
\item [(ii)]  $\nu \ll \mu$ and $\nu$  has $\mu$-support.
\item [(iii)]  $\nu \ll \mu$ and  $\nu$ is  $\mu$-semifinite.
\item [(iv)]  $\nu \ll \mu$ and $\psi_\nu$ is a normal weight  on $\M_+$.
  \end{itemize}
Indeed every normal weight on $\M_+$,
is of the form $\psi_\nu$ for $\nu$ as in {\rm (i)}.   The function $h$ in {\rm (i)} is  unique up to $\mu$-a.e.\ equality.
Moreover if these hold then in addition $\mu \ll \nu$ if and only if $h$ in {\rm (i)} may be chosen to be strictly positive.
\end{corollary}

 \begin{proof}    That (i) implies  (ii), and (iii) is equivalent to (ii), is as in Theorem \ref{rn}, for example using Lemma \ref{fis} again.

(ii) $\Rightarrow$ (iv) \  Suppose that $\nu \ll \mu$ and $\nu$ has $\mu$-support.
 Then $\psi = \psi_\nu$ is normal by Lemma \ref{mus} (1).

(iv) $\Rightarrow$ (i) \ If $p = \chi_S$ is the `semifinite' part projection of $\psi_\nu$ then $\psi_\nu$ is semifinite and normal
on $\M p = L^\infty(S,\mu_{|S})$. Also $\psi_\nu(r) = \infty$ for every nonzero projection $r \in \M (1-p)$.
By Theorem \ref{rn} there exists measurable $h : S \to [0,\infty)$  with $\nu(E) = \int_E \, h \, d \mu$ for all $E \in \Sigma$ with $E \subset S$.
Setting $h = \infty$ on $S^c$,  it is easy to see that $\nu(E) = \int_E \, h \, d \mu$ for all $E \in \Sigma$.

Most of the   final assertions  are as before, for example the first  follows from Lemma \ref{tol}.
It is an exercise
suitable for a graduate real variables course that if  $\mu$ is semifinite
and $f, g: X \to [0,\infty]$ are measurable
with $\mu_{f} = \mu_{g}$, then  $f = g$ $\mu$-a.e.
 For example suppose there is a set $E \in \Sigma$ where $g(x) = \infty$ but $f(x) < \infty$ a.e.
If  $\mu(E_n) > 0$ then  $E_n$ contains a set $F$ of strictly positive finite measure
and we get the contradiction $$\infty = \int_F \, g \, d \mu = \int_F \, f \, d \mu < \infty .$$
Thus $f^{-1}([0,\infty)) = g^{-1}([0,\infty))$
(and $f^{-1}(\{ \infty \}) = g^{-1}(\{ \infty \})$) $\mu$-a.e.
  \end{proof}

\begin{corollary} \label{sigmaf} Let $(X,\Sigma,\mu)$ be a  $\sigma$-finite measure space, and set $\M = L^\infty(X,\Sigma,\mu)$.
Then $\psi_\nu$ is a normal weight  on $\M_+$ for every  positive measure $\nu$  on $(X,\Sigma)$ with $\nu \ll \mu$.
\end{corollary}

 \begin{proof}  By the usual reduction of $\sigma$-finite to finite we may suppose that $\mu$ is finite.
 Then appeal to the proof of 
 Lemma \ref{fis} (1).
 \end{proof}

The following very general Radon-Nikodym theorem also does not seem to appear  in the literature.

\begin{theorem} \label{rnd2} {\rm (Radon-Nikodym for Dedekind measures)} \ Fix a
 Dedekind measure space $(X,\Sigma,\mu)$.
 Let $\nu$ be a positive  measure   on $(X,\Sigma)$ with $\nu \ll \mu$, and let $(S,\mu_{|S})$ be the semifinite part 
of $\mu$.   We assume that  either {\rm (i)} \ $\nu$ is $\mu$-semifinite; or  {\rm (ii)} \
   $\nu$ is $\mu_{|S}$-semifinite on $S$ and  $\nu = \infty$ on all nonempty 
measurable subsets of  $S^c$.
Then there exists a  measurable $h : X \to [0,\infty]$  with $\nu(E) = \int_E \, h \, d \mu$ for all $E \in \Sigma$.  Such $h$ is necessarily $\mu$-a.e.\ unique on $S$. 
Also, $\psi_\nu$ is normal on $L^\infty(X,\Sigma,\mu)_+$.

 Finally, if the above conditions all hold (including either (i) or (ii)), then
\begin{itemize} \item [(1)]   $\mu \ll \nu$   if and only if
$h$ may be chosen to be strictly positive,
 \item [(2)]  $\nu$ is semifinite on $S$  if and only if $h$ may be chosen to be finite on $S$.
  \end{itemize}
\end{theorem}

 \begin{proof}   By 
 Lemma \ref{mu} (2) and (1), 
 $\mu$ has a semifinite part $(S, \mu_{|S})$, and $(S, \mu_{|S})$ is Dedekind.
Also,  $\nu_{|S} \ll \mu_{|S}$.
 If $\nu$ is $\mu$-semifinite then  $\nu_{|S}$ is $\mu_{|S}$-semifinite.
 In this case it follows from Theorem \ref{rn2}  that there is a measurable function $h_S : S \to [0,\infty]$ such that
 $d \nu_{|S} = h_S \, d \mu_{|S}$.    On $X \setminus S$ we have that $\nu =0$ by Lemma \ref{mu} (2).
 Let $h =  h_S$ on $S$ and otherwise $0$.
  Then $\nu(E) = \int_E \, h \, d \mu$ for all $E \in \Sigma$.

Similarly in case (ii), except that now on $X \setminus S$ we have that $\nu = \infty$, and we set $h = \infty$ here.

  Integration with respect to $\nu$ on $S$ is normal on $\M = L^\infty(S, \mu_{|S})$ by Theorem \ref{rn2}
   (iv). 
   Clearly  in case (ii) integration with respect to $\nu$ on $S^c$ is (zero and) normal  on $\M(1-p) = L^\infty(S^c, \mu_{|S^c})$.
Hence one sees that $\psi_\nu$ is  normal.
Similarly in case (ii), except that now $\psi_\nu$ is  the trivial infinite weight on $\M(1-p)$.

Item (1)  is  proved just as in Theorem \ref{rn}, and 
(2) follows from the first part and from Theorem \ref{rn} applied on $S$.
  \end{proof}

\begin{remark}  Some closing remarks on the material in this section:

(1)\ For non-localizable measure spaces $(X,\Sigma,\mu)$, and a measure $\nu \ll \mu$,
  $\nu$ having $\mu$-support need not imply  $\mu$-semifiniteness, nor vice versa.  Nor need these be
  related to having $\nu(E) = \int_E \, h \, d \mu$ for some nonnegative measurable $h$.
       A first counterexample is a localizable measure $\nu$ on $(X,\Sigma)$ some of whose strictly positive values are changed to $\infty$ to yield a nonsemifinite measure $\mu$ with $\nu$-support.
Or consider Example \ref{nonl}, the  standard example of a semifinite non-Dedekind space,   an uncountable set   $X$ with counting measure $\mu$ on the
 countable/co-countable  $\sigma$-algebra. If $E$ is the uncountable  subset with uncountable complement  in that example,
 define $\nu(A) = \mu(E \cap A)$ for $A \in \Sigma$.    One can show easily that
 $\nu$ is
  $\mu$-semifinite but does not have $\mu$-support, nor satisfies the Radon-Nikodym theorem.

(2)\  
 We remark that an alternative proof of (i) in Theorem \ref{rn} is to  `decompose' $\mu$ and/or $\nu$ into finite pieces.
For example  if  $\mu_i$ are as in Theorem \ref{loc} (iii), then $\nu_{|E_i} \ll \mu_i$,
 so by the Radon-Nikodym theorem referred to in the proof of Lemma \ref{fis} (1) we get $d \nu_{|E_i}  = h_i d \mu_i$.  
Alternatively, $\psi_\nu$ may be `decomposed' as a sum of positive
functionals $\psi_i$ with mutually orthogonal supports  and corresponding to $h_i \in L_1(\mu)_+$.
 One may then assemble $h$ from the $h_i$.  There will be some work in showing that $h$ is measurable.
Since $\psi_\nu$ and $\psi_{\mu_h}$ are normal and agree on each `piece' of $\M$,  they must be equal. 

As usual we write the $h$ in (i) of Theorem \ref{rn}  as $\frac{d\nu}{d\mu}$, and call it the
  Radon-Nikodym derivative.
  As usual, $\frac{d(\nu + \sigma)}{d\mu} =
\frac{d\nu}{d\mu} + \frac{d\sigma}{d\mu}$ $\mu$-a.e.
Also if $\nu \ll \mu$ satisfies all the conditions of the theorem,
and if $\sigma  \ll \nu$ similarly satisfies all the conditions of that theorem (with $\mu$ replaced by $\nu$), then
$\frac{d\sigma}{d\nu} \frac{d\nu}{d\mu}
= \frac{d\sigma}{d\mu}$ $\mu$-a.e.
Indeed $$\sigma(E) = \int_E \, \frac{d\sigma}{d\nu} \, d \nu =
\int_E \,  \frac{d\sigma}{d\nu} \, \frac{d\nu}{d\mu} \, d \mu , \qquad E \in \Sigma.$$ So $\frac{d\sigma}{d\nu} \frac{d\nu}{d\mu}
= \frac{d\sigma}{d\mu}$ $\mu$-a.e.
It follows that $\sigma$ is strongly $\mu$-semifinite.

It follows that  if $\nu \ll \mu$ satisfies all the conditions of  Theorem  \ref{rn} and if also
 $\mu \ll \nu$,  then $\mu  \ll \nu$ similarly satisfies all the conditions of
 the theorem (with $\mu$ and  $\nu$ interchanged), and
$\frac{d\nu}{d\mu} = (\frac{d\mu}{d\nu})^{-1}$ $\mu$-a.e.   In particular,  $\mu$ is strongly $\nu$-semifinite.
To see this, note that if $\mu(E) > 0$ then $\nu(E) > 0$, and since $\nu$ is strongly $\mu$-semifinite
there is a $\nu$-finite set $F \subset E$ with $0 < \nu(F) < \infty$.   Clearly $\mu(F) > 0$, so $\mu$ is strongly $\nu$-semifinite.
Now apply the last paragraph to see that
$\frac{d\mu}{d\nu} \frac{d\nu}{d\mu}
= \frac{d\mu}{d\mu} = 1$ $\mu$-a.e.

(3)\ If $h$ is as in (i) of  Theorem \ref{rn} then $\mu_h$ is semifinite and in fact is Dedekind, 
hence localizable, on $\Sigma$ too.
To see this note that $\mu_h$  has  a $\mu$-support $E_0$ by Lemma \ref{fis} (2).
On $E_0$ we have that $\mu_h \cong \mu$,  hence
$\mu_h$ is Dedekind on $E_0$; and  $\mu_h$ is zero on $E_0^c$.
If $\D \subset \Sigma$ then $\{ D \cap E_0  : D \in \D \}$ has an $\mu$-essential supremum $S$ say
in the sense of the first lines of the proof of Lemma \ref{mu}, with $S \subset E_0$.    Then $S$ (and even $S \cup E_0^c$)
is an $\mu_h$-essential supremum  of $\D$.    For
$\mu_h(D \setminus S) =  \mu_h((D \cap E_0) \setminus S) = 0$ for all $D \in \D$.
Also, if $E \in \Sigma$ with $\mu_h(D \setminus E) = 0$ for all $D \in \D$,
then $\mu_h((D \cap E_0) \setminus (E \cap E_0)) = 0$.  Hence $E$ and $E \cap E_0$ contains $S$ $\mu$-a.e.

Note however that for set theoretic reasons (involving measurable cardinals) it would not be good to attempt to use
$\nu$ being localizable with  $\nu \ll \mu$ as an
equivalent condition in Theorem \ref{rn}.  Indeed 
(i) need not hold for finite measures
(which are all 
localizable), as we discuss momentarily.

One may ask if the condition that $\nu$ has $\mu$-support is necessary in  Theorem  \ref{rn}.
The answer is in the affirmative for infinite measures (see Example  \ref{nonl}).   If $\nu$ is also finite in the theorem then we may remove the
hypothesis in (ii) that $\nu$ is semifinite (it is automatic), and ask if the $\mu$-support
condition is necessary? 
In this case `strongly $\mu$-semifinite' in (iii) is
equivalent to $\mu$-semifiniteness, and we may ask whether the latter is automatic?
The answer to all of these questions depends on the existence of measurable cardinals.   A theorem to this effect was in an earlier draft of our paper: Every positive finite measure $\nu \ll \mu$, for every localizable measure space $(X, \mu)$, is of the form $\mu_h$ in the Radon-Nikodym theorem for a measurable function $h : 
X \to [0,\infty]$ (resp.\ has $\mu$-support, is $\mu$-semifinite), if and only if measurable cardinals do not exist.
And similarly with `finite'  replaced by `semifinite'.
We may present the proof of this elsewhere, although some of this is discussed in scattered 
parts of various volumes of \cite{Fremlin}.   

Similarly, one may ask if Theorem \ref{rn2} holds for every  positive measure $\nu$  on $(X,\Sigma)$ with $\nu \ll \mu$. Another counterexample, this time with $\nu(X) = \infty$,  is if  $\mu$
 is a localizable infinite measure some of whose strictly positive values
  are changed to $\infty$ to yield a nonsemifinite measure $\nu$.  Then $\nu \ll \mu$,
but $\nu$ is not $\mu$-semifinite, hence $\nu$ cannot be of the desired form $\mu_h$. 
However we can in some sense relax the condition in the Radon-Nikodym theorems above
 by dropping
 the condition that  $\nu$ is  $\mu$-semifinite, at the expense of the introduction of a `local $\mu$-finiteness'
related to
  234Yn in \cite{Fremlin}.

 (4)\  A simple but important example  due to Haagerup (\cite[Example 3.6]{GL} or \cite[Remark 1.12]{haag-Nw})
 is the weight $\omega$ on $l^\infty_+$ defined by $\omega((a_n)) = \sum_n \, a_n$ if $(a_n)$ is finitely
 supported, and is $\infty$ otherwise.
 This weight  is
completely additive on projections,
and the induced
measure $\nu(E) = \omega(\chi_E)$ agrees with counting measure
on $\Ndb$, but $\omega$ is not normal.

(5)\ Normality of $\psi_\nu$ does not imply  that $\nu$ is $\mu$-semifinite for general
Dedekind measures.   The Dixmier algebra example
at the end of Section \ref{cmsvn} shows this.
Similarly, 
 existence of a  measurable $h : X \to [0,\infty]$  with $\nu(E) = \int_E \, h \, d \mu$ for all $E \in \Sigma$,
does not imply  that $\nu$ is $\mu$-semifinite.
\end{remark}

\section{$L^\infty$ spaces as $W^*$-algebras - 2}  \label{Ws2}

We now complete the task initiated in Section \ref{LinftyW1}.
Theorem \ref{frec} is a key part of the characterization of commutative $W^*$-algebras in terms
of localizable measure spaces (see also Corollary \ref{awvn}),
but contains much more.

The assertions of  the next result  (and some other results in this section) follow
 easily from
the commutative case of Sakai's famous theorem (proved later in Corollary \ref{awvn}).
Indeed Lemma \ref{WStone}
is often used to prove that theorem, and may be skipped by  seasoned operator algebraists.

\begin{lemma} \label{WStone} Let $\M$ be an abelian $W^*$-algebra.
Then any bounded increasing net
in $\M_+$ has a supremum, which coincides with its weak* limit. In particular the supremum of an increasing net of projections will be a projection. In addition the weak* 
continuous states of $\M$ determine the order on $\M$.
\end{lemma}

\begin{proof} It will be helpful to
recall that $\M = C(K)$ for compact $K$, so that elements of $\M$ will sometimes silently
 be treated as functions (on $K$) below.

It is easy to see that $\M_{\rm sa}$ (resp.\ $\M_+$), or equivalently by
the Krein-Smulian theorem its intersection with Ball$(\M)$, is weak* closed.   This follows for example by considering
$x_t \to a + ib$ for $x_t , a, b \in \M_{\rm sa}$ with $\| x_t \| \leq 1$. Then  $\| x_t + i s \| \leq \sqrt{1+ s^2}$,
so that $\| b + s \| \leq \| a + i(b+s) \| \leq \sqrt{1+ s^2}$ for all real $s$, which forces $b = 0$
(see \cite[Lemma 6.1.1]{JP}). But then Ball$(\M)_+ = {\rm Ball}(\M)_{\rm sa} \cap (1 - {\rm Ball}(\M))$ is weak* closed.

By the bipolar theorem 
it follows that weak* continuous positive functionals determine the order
on $\M$. In particular, for any $0 \neq x \geq 0$ there exists a weak* continuous positive functional $\varphi$ with
$\varphi(x) > 0$.

Any weak* limit point $x$ of a bounded increasing net $(x_t)$ in $\M_+$ is a supremum for the net.
Indeed $x$ is an (increasing) weak* limit of a subnet $(x_{t_\lambda})$,
so for any $t$ we can find $t_{\lambda} \geq t$ with $x_t \leq x_{t_\lambda} \leq x$.
Also if $y \geq x_{t_\lambda}$ for all $\lambda$, then clearly $y  \geq x$.
Thus $x = \sup_t \, x_t$ and $x_t \to x$  weak*.

Now suppose that $(x_t)$ is an increasing net of projections in $\M$ with supremum $x$.
Then $f \leq 1$ and so for each $t$ we have  $f_t = f f_t$.
Hence  $(f+1) f= \sup_t \,  (f+1)  f_t = \sup_t \, 2 f_t  = 2f$.
 That is, $f$ is a projection.
 \end{proof}

\begin{theorem} \label{frec}  Let $\mu$ be a measure on a measurable space $(X,\Sigma)$, and set $\M = L^\infty(X,\Sigma,\mu)$.
The following are equivalent:
\begin{itemize}
\item [(i)] $(X,\Sigma,\mu)$ is dualizable (or equivalently, localizable).
\item [(ii)] $\mu$ is semifinite and $\M$ is a $W^*$-algebra (that is, has a predual).
\item [(iii)] $\mu$ is semifinite and $\M_{\rm sa}$ is  boundedly complete (that is, suprema in $\M_{\rm sa}$ of bounded subsets of $\M_{\rm sa}$ always exist).
\item [(iv)] The canonical representation of $\M$ on $L^2(X,\mu)$ is
faithful and has range which is a von Neumann algebra.
\item [(v)] The canonical representation of $\M$ on $L^2(X,\mu)$ is
faithful and has range which is a masa (that is, a maximal abelian subalgebra of $B(L^2(X,\mu))$).
\item [(vi)]   $\mu$ is semifinite and every 
strongly  $\mu$-semifinite positive
measure $\nu$ on $(X,\Sigma)$
with $\nu \ll \mu$
satisfies
the Radon-Nikodym theorem. That is, there exists measurable  $h : X \to [0,\infty)$ on $X$ with $\nu = \mu_h$.   \end{itemize}
\end{theorem}

\begin{proof}  That (i) implies (iii) and (iv)  was done in Theorem \ref{lsigf1}.   Similarly (i) implies (ii)
by definition of a $W^*$-algebra and of dualizability. 
That (ii) implies (i) follows for example from Lemma
  \ref{WStone}.
 That (i)  implies (vi)  was done in  Theorem \ref{rn},  and that (v) implies (iv) is obvious.

 (iv) $\Rightarrow$  (i) \ 
  In the proof of Theorem \ref{lsigf1} we showed
  that the canonical representation of $\M$ on $L^2(X,\mu)$ is
faithful  if and only if  $\mu$ is semifinite.    Since  von Neumann algebras are $W^*$-algebras
we can deduce from for example  Lemma
  \ref{WStone} that $(X,\Sigma,\mu)$ is Dedekind.
  
(iii) $\Rightarrow$  (i) \ This is similar to the last lines, but we invoke the last paragraph of the proof of Lemma
  \ref{WStone}.

(i) $\Rightarrow$ (v) \   By Theorem \ref{loc} (iii)  we have $\M \cong \N$ via a $*$-isomorphism $\theta$, where $\N = \oplus_i \, \M_i$   with  each $\M_i$ of the form $L^\infty(\mu_i)$ for some finite measure $\mu_i$.    We can represent $\M_i$ on $H_i = L^2(\mu_i)$ via the regular representation $\pi_i$, and this induces a faithful normal $*$-representation
$\pi : \N \to B(H)$ where $H = \oplus_i \, H_i$.   We have that $\pi((x_i)) = (\oplus_i \, \pi_i(x_i))$ acts `diagonally', that is as a diagonal matrix when written in block matrix form with respect to the subspaces
$H_i$.
Thus if $T \in B(H)$ commutes with
$\pi(\M)$ it is an exercise to see that $T$ acts diagonally too, indeed $T = (T_i)$ where $T_i \in \pi_i(\M_i)'$.   It is a standard argument to see that
$T_i \in \pi_i(\M_i)$.
(Indeed if $f = T_i(1)$ then for $h \in L^\infty(\mu_i)$ we have $T_i(h) = T \M_h 1 = \M_h T(1) = hf$.   Since $\| f h \|_2 \leq \| T \| \| h \|_2$ we have
$f \in L^\infty(\mu_i)$ and $T_i = \M_f \in \pi_i(\M_i)$.)   Thus $\pi(\M)' = \pi(\M)$, and this forces $\pi(\M)$ to be a masa.

(vi)  $\Rightarrow$ (i) \   We will only use the case of (vi) where $\nu \leq \mu$.
It is classical 
 (see e.g.\  \cite[Theorem 6.15]{Fol}) that the canonical map $L^\infty \to (L^1)^*$ is an isometry for semifinite measures.
It suffices to show that any $\varphi \in {\rm Ball}(L^1(\mu)^*)_+$ is `integration' with respect to the measure
$h \, d \mu$, for
some $h \in L^\infty(\mu)$.   Define $\nu'(E) = \varphi ( \chi_E)$ for $\mu$-finite $E \in \Sigma$.   This is countably additive on
the measurable $\mu$-finite sets, by the matching part in the 
 standard proof   that 
$L^\infty =  (L^1)^*$ for finite measures (for example see the first paragraph of the proof of 
 \cite[Theorem 6.15]{Fol}).  Define 
$\nu(E) = \sup \, \nu'(F)$ where the supremum is over $\mu$-finite measurable $F \subset E$. 
This is easily checked to be a measure 
with $$\nu(E) = \sup \, \nu'(F) = \sup \, \varphi ( \chi_F ) \leq \sup \,\|  \chi_F \|_1 \leq \mu(E) .$$  
Moreover if $\nu(E) > 0$ then there exists $\mu$-finite measurable $F \subset E$ with $0 < \nu'(F) = \nu(F) < \infty$. 
So $\nu$ is  
strongly $\mu$-semifinite. 
By (vi) there exists measurable  $h : X \to [0,\infty)$ on $X$ with $\nu = \mu_h$.  
Let  $D$ be a set on which $h > 1$.  By semifiniteness we may assume  that $\mu(D) < \infty$. 
Since $0 \leq \int_D \, (h-1) \, d \mu = \nu(D) - \mu(D) \leq 0$ we see that $h = 1$ $\mu$-a.e.\ on 
$D$.  So $\mu(D) = 0$, and 
we may 
assume that $h : X \to [0,1]$ and $h \in L^\infty(\mu)$.   We have 
$$\varphi(f) = \int_X \, f \, d \nu = \int_X \, f  h \, d \mu$$ when $f$ is the characteristic function of a measurable $\mu$-finite set, hence
also for simple functions which are
linear combinations of such.   Now let  $f \in L^1(\mu)_+$, and $(s_n)$ be a sequence of the latter simple functions
with $s_n \geq 0$ and $s_n \nearrow f$.
 It follows from Lebesgue's monotone convergence theorem that
 $\int_X \, (f-s_n) \, d \mu \to 0$, so that $\varphi(s_n) \to \varphi(f)$, and
 $\varphi(s_n) = \int_X \, s_n h \, d \mu \to \int_X \, f  h \, d \mu$.   Thus $\varphi(f) =  \int_X \, f  h \, d \mu$
 for all $f \in  L^1(\mu)$.  \end{proof}

 We will need the next  results   later.
 The following fairly elementary fact may be found in for example Propositions 1.3.1 \& 1.3.2 of \cite{Sakai}. 
 We recall that a topological space is called {\em extremely disconnected} if the closure of any open set is open,
 and {\em Stonean} if in addition it is compact.
 
\begin{lemma} \label{sa1}\ Let $K$ be a compact Hausdorff space  such that
every  bounded increasing net of positive elements in $C(K)$ has a supremum.  Then $K$ is 
Stonean, and the projections in $C(K)$ are densely spanning.
\end{lemma}

\begin{proposition} \label{sa3}  Let $\M$ be an abelian $W^*$-algebra.
Then the spectrum (i.e.\ maximal ideal  space) of $\M$ is Stonean, and the projections in $\M$ are densely spanning.
 \end{proposition}

\begin{proof} The Gelfand transform is a $*$-isomorphism from $\M$ to $C(K)$, where $K$ is the 
 spectrum of $\M$. Using this $*$-isomorphism, it then follows from Lemma \ref{WStone} that the hypothesis of
Lemma \ref{sa1} is satisfied for $C(K)$.
\end{proof}

\section{On Radon measures} \label{Rad}

One clear message from the previous section is that those semifinite measure spaces for which $L^\infty(X,\Sigma,\mu)$ is a $W^*$-algebra are the ones which are either decomposable or localizable. The question we now address is whether there is a large class of
even better understood, or more tractable,
 measure spaces that are
 identified by these criteria. It turns out that Radon measure spaces serve exactly this purpose,
 as we elucidate in
subsections 
 \ref{rmal} and \ref{avna}. Given the potential significance of this class for the objective we have in mind, we proceed to give a brief overview of Radon measure spaces.

Let $X$ be a set with a locally compact Hausdorff topology $\mathfrak{T}$. The Borel $\sigma$-algebra $\mathscr{B}(X)$ on $X$ is  the $\sigma$-algebra generated by $\mathfrak{T}$. A Borel measure $\mu$ on $X$ is said to be
\emph{outer regular on a measurable set} $E$ if 
$$\mu(E)=\sup\{\mu(U)\colon U\in\mathfrak{T}, E\subseteq U\},$$
and \emph{inner regular on} $E$ if on the other hand
$$\mu(E) =\sup\{\mu(C)\colon C\subseteq E, C\mbox{ is compact}\}. $$
We simply say that $\mu$ is  \emph{outer regular} (resp.\ \emph{inner regular}) if it is
outer (resp.\ inner) regular  on every measurable set.

Radon measures on such $X$
are a very special category of Borel measures on 
$X$, constructed from positive linear functionals on $\mathcal{K}(X)$ - the continuous functions of compact support.
 There are however several different definitions of Radon measures in the literature, all justified by some  particular factor
an author may need. Hence we pause to give some background regarding the issues that gave rise to some of these various versions. We shall start with the general process of constructing a Radon measure.

Let $X$ be as before and $I$ a positive linear functional on $\mathcal{K}(X)$. By the Riesz representation theorem, there exists a  Borel measure $\mu$ for which we have that $I(f)=\int f\,d\mu$ for each $f\in \mathcal{K}(X)$.
The measure $\mu$ is usually constructed in a real variables course  by first defining the set function $\mu^*$ on open subsets by means of the prescription $$\mu^*(U)=\sup\{I(f)\colon f\in\mathcal{K}(X), 0\leq f\leq \chi_U, \mathrm{supp}(f)\subseteq U\},$$ and then extending to all subsets of $X$ by the prescription
\begin{equation}\label{mustardef}\mu^*(A)=\inf\{\mu^*(U)\colon U\mbox{ open and }A\subseteq U\}. \end{equation}
This quantity turns out to be an outer measure for which
the $\mu^*$-measurable sets includes the Borel $\sigma$-algebra. Hence $\mu^*$ restricts to a measure $\mu$ on the Borel
$\sigma$-algebra $\mathscr{B}(X)$ which is finite on compact sets, for which we have that $I(f)=\int f\,d\mu$ for all
$f\in\mathcal{K}(X)$, and which is outer regular on all of $\mathscr{B}(X)$, but inner regular on open sets only. For some authors it is this measure $\mu$ which is a Radon measure. However it does have some shortcomings, namely that the
measure may not be inner regular on all Borel sets (see \cite[Exercise 3.23]{Sal}),
and indeed it may not be semifinite.

Recall from for example \cite[Exercise 1.15]{Fol} that an arbitrary measure $\nu$ on a measure space 
$(X,\Sigma)$ may be written as  $\nu=\nu_{sf}+\nu_\infty$, 
where $\nu_{sf}$ is semifinite and $\nu_\infty$ assumes only the values 0 and $\infty$.  In fact
$$\nu_{sf}(E)=\sup\{\nu(F): F \in \Sigma, \nu(F)<\infty, \, F\subset E\} ,$$
and will be referred to
below as the {\em semifinite part} of $\nu$.  All measurable sets with $\nu(E)<\infty$ are trivially semifinite, and so $\nu(E) = \nu_{sf}(E)$ for such sets.

Many authors insist on Radon measures being inner regular on all measurable sets. It is clear that any inner regular 
measure on $(X,\mathscr{B}(X))$ which is finite-valued on compact sets is semifinite. Thus semifiniteness presents itself 
as the possible criterion needed to ensure universal inner regularity. If the measure constructed as described immediately 
after Equation (\ref{mustardef}) is semifinite, then it is indeed universally 
inner regular. If however this measure is not 
semifinite, further modification is necessary. 
In such a case the `semifinite part' of the measure $\mu$ defined
above, which we shall denote by $\mu_0$, is indeed inner regular on all Borel sets, and may alternatively be
realized by the prescription $$\mu_0(E) =\sup\{\mu(C)\colon C\subseteq E, C\mbox{ is compact}\}.$$(See
\cite[Exercise 5, \S 7.2]{Cohn} for this fact.) This measure $\mu_0$ moreover also satisfies $I(f)=\int f\,d\mu_0$ for all $f\in\mathcal{K}(X)$, and hence may also be regarded as a measure fulfilling the objective of the Riesz representation theorem (see for example \cite{Sal} for an outstanding account
of this, and of the issues involved). It is such inner regular Borel measures that many other authors refer to as Radon measures. However the benefit of ensuring inner regularity (via 
 passage to $\mu_0$) is not without cost. Specifically the measure $\mu_0$ can only be outer regular if it agrees with $\mu$ \cite[Exercise 5(b), \S 7.2]{Cohn}. Thus ensuring inner regularity often means sacrificing outer regularity. Using the fact that its semifinite part is inner regular, there is one serendipitous observation we may make regarding $\mu$, and that is that $\mu$ is in fact inner regular on sets of finite $\mu$-measure (and more generally on all sets of $\sigma$-finite $\mu$-measure). This stems from the fact that $\mu$ agrees with its semifinite part on such sets.

Our ultimate interest as far as Radon measures are concerned, is to establish good connections with 
decomposable measure spaces, and with abelian von Neumann algebras. For this further structure is 
needed--the Borel $\sigma$-algebra is not large enough in general. To add such structure, we 
enlarge $\mathscr{B}(X)$ to respectively the $\mu^*$ and $\mu_0^*$-measurable sets. We write $\mu_0^*$ for the outer
 measure induced by $\mu_0$ on the power set $\mathcal{P}(X)$, and $\mathcal{M}_{\mu_0^*}$ for the $\mu_0^*$-measurable sets.
That is, $\mu_0^*(E)$ is the infimum of $\mu_0(B)$ over Borel sets $B$ containing $E$.
We shall retain the notation ${\mu}^*_0$ for the restriction of $\mu_0^*$ to $\mathcal{M}_{\mu_0^*}$.  This is a measure that extends
$\mu_0$ (by Caratheodory's theory), in particular its domain contains all Borel sets.  
The ${\mu}^*_0$-integral
clearly coincides with the $\mu$-integral on $\mathcal{K}(X)$.   Because of this, the correspondence between such inner regular measures and positive functionals on $\mathcal{K}(X)$ is bijective.  
The $\mu^*$-measurable sets will similarly be denoted by $\mathcal{M}_{\mu^*}$, and the restriction of $\mu^*$ to $\mathcal{M}_{\mu^*}$ by $\mu^*$.

Another well known approach to Radon measures is due to Bourbaki, but this turns
out to be very closely connected with $\mu_0^*$ to $\mathcal{M}_{\mu_0^*}$ above.
The Bourbaki approach is summarized for example
 on pages 215 and 216 of \cite{Cohn}.   By Theorem 7.5.5 in  \cite{Cohn}
 the Bourbaki approach gives the same scalar valued measurable functions,
 the same $L^1$ space as $L^1(X, \mathcal{M}_{\mu^*},\mu^*)$, and the same integral there.

In \cite[Proposition 7.5.1]{Cohn}, Cohn proves the variant of the following result for the measure space
$(X, \mathcal{M}_{\mu^*},\mu^*)$.  The proof carries over to the present setting, but some rather non-trivial
modifications need to be made. 
We therefore provide details for the reader's convenience.

\begin{proposition} \label{Cohn7.5.1} Let $X$ be a locally compact Hausdorff space. Then in terms of the notation above, for any subset $E$ of $X$, the following are equivalent:
\begin{itemize}
\item[a)] $E$ belongs to $\mathcal{M}_{\mu_0^*}$;
\item[b)] $E\cap U\in \mathcal{M}_{\mu^*_0}$ whenever $U$ is an open set with $\mu_0(U)<\infty$;
\item[c)] $E\cap K\in \mathcal{M}_{\mu^*_0}$ for any compact set $K$.
\end{itemize}
\end{proposition}

\begin{proof} The implication (a)$\Rightarrow$(c) is clear. Hence we need only show that (c)$\Rightarrow$(b)$\Rightarrow$(a).

We first show that (b)$\Rightarrow$(a).  Suppose that we are given a subset $E\subseteq X$ such that
$E\cap U\in \mathcal{M}_{\mu^*_0}$ whenever $U$ is an open set with $\mu_0(U)<\infty$. Then $E$ will belong to
$\mathcal{M}_{\mu_0^*}$ if we can show that $\mu_0^*(A) \geq \mu_0^*(A\cap E)+\mu_0^*(A\cap E^c)$ for any subset $A$ of $X$
for which we have that $\mu_0^*(A) < \infty$. So let $A$ be such a set. If $\mu_0^*(A)=0$, then by monotonicity of $\mu_0^*$
we also have that $0=\mu_0^*(A\cap E)=\mu_0^*(A\cap E^c)$, in which case the inequality then holds. In the case where
$\infty>\mu_0^*(A)>0$, we can for any given $\epsilon>0$, and select a Borel set $B$ such that $A\subseteq B$ and
$\mu_0^*(A)+\epsilon\geq \mu_0(B)$. By the definition of $\mu_0$ we may now select a Borel set $B_\epsilon \subseteq B$ with
$\mu_0(B)\geq\mu(B_\epsilon)> \mu_0(B)-\epsilon$. (Recall that $\mu$ is the Borel measure defined after
(\ref{mustardef}).) Since $\mu(B_\epsilon)<\infty$, we have $\mu(B_\epsilon)=\mu_0(B_\epsilon)$,
and hence that $\mu_0(B-B_\epsilon)\leq \epsilon$. This agreement of $\mu$ and $\mu_0$ on $B_\epsilon$ enables us to import
some of the properties of $\mu$ into the present proof.

Let $U$ be an open set which includes $B_\epsilon$ and which satisfies
$\mu^*(U) =\mu(U) < \infty$. The agreement of $\mu$ and $\mu_0$ on Borel sets with finite
$\mu$-measure ensures that $\mu_0(U)=\mu(U)<\infty$.  By assumption we then have that
$U \cap E\in \mathcal{M}_{\mu^*_0}$. Invoking the standard measurability criterion therefore shows that
\begin{eqnarray*}
\mu(U)&=&\mu_0(U)\\
&=& \mu_0^*(U\cap E)+\mu_0^*(U\cap(U\cap E)^c)\\
&=& \mu_0^*(U\cap E)+\mu_0^*(U\cap E^c)\\
&\geq& \mu_0^*(B_\epsilon\cap E)+\mu_0^*(B_\epsilon\cap E^c)
\end{eqnarray*}
If we now take the infimum over all such open sets, it then follows from equation (\ref{mustardef}) that
$\mu^*(B_\epsilon)\geq \mu_0^*(B_\epsilon\cap E)+\mu_0^*(B_\epsilon\cap E^c)$. Combining this with our earlier observation,
then shows that $$\mu_0^*(A)+\epsilon\geq\mu_0(B)\geq\mu_0(B_\epsilon)=\mu(B_\epsilon)
\geq \mu_0^*(B_\epsilon\cap E)+\mu_0^*(B_\epsilon\cap E^c).$$
Now observe that
\begin{eqnarray*}
\mu_0^*(B\cap E) &\leq& \mu_0^*((B \setminus B_\epsilon)\cap E)+\mu_0^*(B_\epsilon\cap E)\\
&\leq& \mu_0(B \setminus B_\epsilon)+\mu_0^*(B_\epsilon\cap E)\\
&\leq& \epsilon + \mu_0^*(B_\epsilon\cap E).
\end{eqnarray*}
Similarly $$\mu_0^*(B\cap E^c)\leq \epsilon + \mu_0^*(B_\epsilon\cap E^c).$$It therefore follows that
\begin{eqnarray*}
\mu_0^*(A)+\epsilon&\geq& \mu_0^*(B_\epsilon\cap E)+\mu_0^*(B_\epsilon\cap E^c)\\
&\geq& \mu_0^*(B\cap E)+\mu_0^*(B\cap E^c)-2\epsilon\\
&\geq& \mu_0^*(A\cap E)+\mu_0^*(A\cap E^c)-2\epsilon.
\end{eqnarray*}
Since $\epsilon>0$ was arbitrary, we have that $\mu_0^*(A)\geq \mu_0^*(A\cap E)+\mu_0^*(A\cap E^c)$ as was required.

It remains to prove that (c)$\Rightarrow$(b). Let $U$ be an open set  with $\mu_0(U) < \infty$, and $E$ a subset of $X$ satisfying condition (c). The inner regularity of the measure $\mu_0$  allows us to select a sequence $(K_n)$ of compact subsets of $U$ such that $\mu_0(U) = \sup_n \, \mu_0(K_n)$. Condition (c) then informs us that each $E \cap K_n$, and hence $\cup_n(E \cap K_n)$, belongs to $\mathcal{M}_{\mu^*_0}$. On the other hand since by construction $U - \cup_n K_n$ has measure 0, $E \cap (U - \cup_n K_n)$ will  be a $\mu_0$-null set. Since $E\cap U$ is then the union of the measurable set $\cup_n(E \cap K_n)$ and a $\mu_0$-null set, it must belong to $\mathcal{M}_{\mu^*_0}$. Thus (b) follows.
\end{proof}

We shall not use this but remark that Fremlin establishes the strong result that $E \in \mathcal{M}_{\mu_0^{*}}$ if and only if $E \cap B$ is in the completion of the Borel sets with respect to $\mu_0$, for every $\mu_0$-finite Borel set $B$.
See the last assertion of Proposition 416F in \cite{Fremlin}. 

When the following result is considered alongside part (a) of Theorem \ref{loc}, 
we see that of the two
measure spaces $(X, \mathcal{M}_{\mu_0^*},\mu^*_0)$ and $(X, \mathcal{M}_{\mu^*},\mu^*)$, it is the
framework afforded by $(X, \mathcal{M}_{\mu_0^*},\mu^*_0)$ that is well suited to the description of abelian von
Neumann algebras. In particular it is localizable.

\begin{theorem}\label{radlocal} The measure space $(X, \mathcal{M}_{\mu_0^*},\mu^*_0)$ is decomposable. Specifically there exists a
disjoint family $\mathscr{C}$ of compact sets for which following holds:
\begin{enumerate}
\item[a)] $\mu_0(K)>0$ for each $K\in\mathscr{C}$;
\item[b)] for any open set $U$ and any $K\in\mathscr{C}$, we have that $\mu_0(U\cap K)>0$ whenever $U\cap K\neq \emptyset$;
\item[c)] for any compact set $K_0$ we have that $K_0\cap K\neq\emptyset$ for at most countably many sets $K$ in
$\mathscr{C}$;
\item[d)] for any $E\in \mathcal{M}_{\mu_0^*}$ we have that $\mu^*_0(E)=\sum_{K\in\mathscr{C}}\mu^*_0(E\cap K)$;
\item[e)] a subset $E$ of $X$ belongs to $\mathcal{M}_{\mu_0^*}$ if and only if for each $K\in \mathscr{C}$,
$E\cap K \in \mathcal{M}_{\mu_0^*}$.
\end{enumerate}
\end{theorem}

There is a version of this result which holds for $(X, \mathcal{M}_{\mu^*},\mu^*)$, but with the trade-off that
whilst (c) will then hold for all sets $E\in \mathcal{M}_{\mu^*}$ with $\mu^*(E)<\infty$, (d) will only hold for sets
$E\in\mathcal{M}_{\mu^*}$ for which $\mu^*(E)<\infty$. For details see \cite[Proposition 7.5.3]{Cohn}.

We need a technical lemma to facilitate the promised proof. This lemma parallels \cite[Lemma 7.5.2]{Cohn} where it was proved for the measure space $(X, \mathcal{M}_{\mu^*},\mu^*)$.

\begin{lemma}\label{lemradlocal} Let $X$ be a locally compact Hausdorff space. If $K$ is a compact subset of $X$ such that $\mu_0(K) > 0$, then there is a compact subset $K_0$ of $K$ such that $\mu_0(K_0) = \mu_0(K)$, with in addition $\mu_0(U\cap K_0)>0$ for any open subset $U$ for which $U\cap K_0\neq \emptyset$.
\end{lemma}

\begin{proof}
Let $V_K$ be the union of all the open sets $V$ such that $\mu_0(V \cap K) = 0$. Any compact subset $\widetilde{K}$ of $V_K\cap K$ will therefore be covered by the union of these $V$'s and be covered by finitely many of these $V$'s. When combined with the subadditivity of the measure $\mu_0$, the fact that each $V\cap K$ has measure zero, then ensures that  $\widetilde{K}$ also does. The inner regularity of $\mu_0$ then leads to the conclusion that $\mu_0(V_K\cap K)=0$. The set $K_0 = K \cap (X-V_K)$ will then fulfill the criteria of the lemma.
Indeed 
if $\mu_0(U\cap K_0) = 0$ then $$0 = \mu_0(U\cap (K \setminus V_K)) =  \mu_0(U\cap K)$$
since $\mu_0(V_K\cap K)=0$.    So $U \subset V_K$ and hence $U \cap K_0$ is empty.
\end{proof}

\begin{proof}[Proof of Theorem \ref{radlocal}]
Let $\mathfrak{C}$ be the collection of all families $\mathscr{C}$ of compact subsets of $X$ satisfying the criteria that \begin{itemize}
\item the sets in $\mathscr{C}$ are mutually disjoint,
\item $\mu_0(K)>0$ for any $K \in \mathscr{C}$, and
\item for any open set $U$ and any $K\in \mathscr{C}$, we have $\mu_0(U\cap K)>0$ whenever
$K\cap U\neq \emptyset$.
\end{itemize}
The collection $\mathfrak{C}$ is clearly non-empty since it contains at least the empty set. On partially ordering this collection by inclusion, it is then not difficult to verify that for any chain in $\mathfrak{C}$, the union of that chain will again belong to $\mathfrak{C}$, and will be an upper bound for that chain. This then brings Zorn's lemma into play, which we may use to conclude that $\mathfrak{C}$ admits a maximal element. Let $\mathscr{C}_0$ be that maximal element. We shall ultimately see that this family satisfies the criteria of the theorem.

Properties (a) and (b) are or course immediate.

We turn to property (c). Given a compact set $K_0$, the local compactness ensures that $K_0$ is contained in an open set $U$
for which $\mu_0(U)<\infty$. (To see this notice that each $x\in K_0$ lives inside an open neighborhood with compact
closure, and that because of its compactness, $K_0$ may be covered by finitely many such neighborhoods.) Then each set $K$
in $\mathscr{C}_0$ which meets $K_0$, must also meet $U$, whence $\mu_0(U\cap K)<\infty$. Since $\mu_0(U)<\infty$ and the
elements of $\mathscr{C}_0$ are disjoint from each other, there can only be finitely many elements $K$ for which
$\mu_0(K\cap U)>\frac{1}{n}$. (If not the union of countably many such sets will be a measurable subset of $U$ with infinite
measure - a clear contradiction.) Thus there are at most countably many $K$'s for which $\mu_0(K\cap U)> 0$, and hence countably many for which $K\cap U\neq \emptyset$. Since
$K_0\subset U$, this establishes the validity of (c).

We pass to investigating (d). We first consider the case where $\sum_{K\in\mathscr{C}_0} \mu^*_0(K\cap E)=\infty$. If
this situation pertains, there must exist a countable subcollection ${K_n}$ with $\sum \mu^*_0(K_n\cap E)=\infty$.
The truth of this fact is obvious if only a countable number of the $\mu^*_0(K\cap E)$'s are non-zero. If on the
other hand there are uncountably many non-zero $\mu^*_0(K\cap E)$'s, then for some $k\in\mathbb{N}$ there must be
infinitely many $K$'s for which $\mu^*_0(K\cap E)>\frac{1}{k}$. Hence again the claim follows. Since
$\sum_{K\in\mathscr{C}_0} \mu^*_0(K_n \cap E)=\mu^*(E\cap(\cup_{n\geq 1} K_n))\leq \mu^*_0(E)$,
we will then clearly have that $\mu^*_0(E)=\infty$ and hence that $\sum \mu^*_0(K\cap E)= \mu^*_0(E)$.

It therefore remains to consider the case where $\sum_{K\in\mathscr{C}_0} \mu^*_0(K\cap E)<\infty$. In this case
there can clearly only be countably many $K$'s in $\mathscr{C}_0$ for which $\mu^*_0(K\cap E)\neq 0$. Let $\{C_n\}$
be those sets. It then follows that $$\sum_{K\in\mathscr{C}_0} \mu^*_0(K\cap E)=\sum_{n\geq 1} \mu^*_0(C_n\cap E) =\mu^*_0((\cup_{n\geq 1}C_n)\cap E)\leq \mu^*_0(E).$$ We proceed with showing
that the converse inequality also holds.

Let $C$ be a compact subset of $E$. By part (c), the set $C$ meets at most countably many of the $K$'s in $\mathscr{C}_0$. Let ${K_n}$ be that collection of compact sets. We claim that $\mu_0(C-\cup_{n\geq 1}(C\cap K_n))=0$. If this was not the case, then by the inner regularity of $\mu_0$, there exists a compact subset $L$ of $C
\setminus \cup_{n\geq 1}(C\cap K_n)$ with $\mu_0(C-(C\cap(\cup_{n\geq 1} K_n)))\geq\mu_0(L)>0$. By the lemma this compact set admits a compact subset $L_0$ for which $\mu_0(L_0)=\mu(L)$ with $\mu_0(L_0\cap U)>0$ for each open set $U$ which meets $L_0$. But the containment of this set in $C
\setminus \cup_{n\geq 1} (C\cap K_n)$ when combined with the fact that $\cup_{n\geq 1} (C\cap K_n)=\cup_{K\in \mathscr{C}_0} (C\cap K)$, ensures that $L_0$ does not meet any of the $K$'s in $\mathscr{C}_0$. But this would contradict the maximality of $\mathscr{C}_0$. We therefore have that $\mu_0(C
\setminus \cup_{n\geq 1}(C\cap K_n))=0$, or equivalently that $$\mu_0(C)=\mu_0(\cup_{n\geq 1}(C\cap K_n))=\sum_{n\geq 1}\mu_0(C\cap K_n)=\sum_{K\in \mathscr{C}_0}\mu_0(C\cap K).$$
 Since $\sum_{K\in \mathscr{C}_0}\mu_0(C\cap K)\leq \sum_{K\in \mathscr{C}_0}\mu^*_0(E\cap K)$, the inner regularity of $\mu^*_0$ therefore ensures that $\mu^*_0(E)\leq \sum_{K\in \mathscr{C}_0}\mu^*_0(E\cap K)$. Thus (d) holds.

As far as (e) is concerned the ``only if'' part follows from Proposition \ref{Cohn7.5.1}. For the converse, suppose we have a set $E$ for which $E\cap K \in \mathcal{M}_{\mu_0^*}$ for each $K\in \mathscr{C}_0$. By Proposition \ref{Cohn7.5.1} it will be enough to show that $E\cap L\in\mathcal{M}_{\mu_0^*}$ for an arbitrary compact subset $L$ of $X$. If we are given such a compact set $L$, we know from part (c) that $L$ then meets at most countably many of the members of $\mathscr{C}_0$. If $\{K_n\}$ is the collection of elements of $\mathscr{C}_0$ which meet $L$, it then follows from (d) that $\mu_0(L) =\sum_{n\geq 1}\mu_0(K_n\cap L)$, and hence that $\mu_0(L
\setminus \cup_{n\geq 1}K_n)=0$. The set $E\cap L$ is then the union of the countable collection of sets $\{E\cap L\cap K_n\}$ and a subset of the $\mu_0$-null set $L
\setminus \cup_{n\geq 1}K_n$. Since all $\mu_0$-null sets belong to $\mathcal{M}_{\mu_0^*}$ and each $E\cap L\cap K_n$ will by hypothesis also belong to $\mathcal{M}_{\mu_0^*}$, it follows that then $E\cap L\in \mathcal{M}_{\mu_0^*}$ as was required.

Finally, to see that $(X, \mathcal{M}_{\mu_0^*},\mu^*_0)$ is decomposable, we set
$N = X \setminus (\cup_{K\in \mathscr{C}_0} \, K)$, and ${\mathcal D} = \mathscr{C}_0
\cup \{  N \}$.   Note that every subset of $N$ is $\mu^*_0$-null.   It then easily follows from
d) and e)  that ${\mathcal D}$  is the desired  partition of $X$ into $\mu^*_0$-finite sets satisfying
the definition of decomposability.
\end{proof}

With Proposition \ref{Cohn7.5.1} as background, the framework below now emerges as a formalisation of the structure of
$(X, \mathcal{M}_{\mu_0^*},\mu^*_0)$. This formalised framework is what Fremlin calls a Radon measure space, and we
shall follow him in embracing this definition.

\begin{definition}\label{def-Radon} A \emph{Radon measure space} is a quadruple $(X,\mathfrak{T}, \mathscr{B},\mu)$ for which we have that
\begin{enumerate}
\item[(i)] $(X,\mathscr{B},\mu)$ is a complete measure space;
\item[(ii)] if $E\subseteq X$ and $E\cap F\in \mathscr{B}$
for all $F\in\mathscr{B}$
with  $\mu(F)<\infty$, then $E\in\mathscr{B}$;
\item[(iii)] $\mathfrak{T}$ is a Hausdorff topology on $X$;
\item[(iv)] $\mathfrak{T}\subset \mathscr{B}$;
\item[(v)] for every $E\in\mathscr{B}$ we have $\mu(E)=\sup\{\mu(C)\colon C\subseteq E, C\mbox{ is compact}\}$;
\item[(vi)] for every $x$ there is an open neighborhood $U$ of $x$ with $\mu(U)<\infty$.
\end{enumerate}
\end{definition}

We do not need this but we mention in passing the interesting fact (416F in 
\cite{Fremlin}) that a measure $\nu$ on the Borel sets of a Hausdorff space 
extends to a Radon measure in the sense above if and only if $\nu$ is inner 
regular and satisfies (vi) above.  In this case the extension is unique.  (A 
nice formula for the domain of this extension was mentioned below Proposition 
6.1 above.  Indeed this is precisely the `c.l.d.\ version' construction from 
213D,E in \cite{Fremlin}, applied to $\nu$.)

It follows from (vi) in the definition that compact sets are $\mu$-finite.  Conversely, the latter implies (vi) if $X$ is locally compact (because every $x$ has an open neighborhood with compact closure).  
Thus we obtain the Riesz representation theorem: A Radon measure $\mu$ induces a  positive functional on the continuous functions of compact support.
Conversely any such functional comes from integration
on  the Radon measure space $(X, \mathcal{M}_{\mu_0^{*}},\mu^{*}_0)$ above, if $X$ is locally compact.
See Volume 4 of~\cite{Fremlin} for more on Radon measure spaces in this sense. 
In particular \textsection 436 of that text treats the Riesz representation theorem, while 416E or 415I there give other ways to easily verify that the above correspondence 
between Radon measures in this sense, and positive functionals on $\mathcal{K}(X)$, is bijective.   

A minor modification of the earlier proof for the measure space $(X, \mathcal{M}_{\mu_0^*},\mu^*_0)$, now yields the following fact:

\begin{theorem}[Theorem 1.10 in \cite{FremlinMA}]\label{Fremlin 1.10} Any Radon measure space $(X,\mathfrak{T}, \mathscr{B},\mu)$ is decomposable.
\end{theorem}

\begin{proof}
The properties encoded in Fremlin's definition are precisely those we need for the
proof of Theorem \ref{radlocal} to work in our new setting. To see this, first note that the
 ideas in the proof of Lemma \ref{lemradlocal} go through verbatim for the measure spaces
described in Definition \ref{def-Radon}. Having noted this we may essentially clone Theorem \ref{radlocal} for
this class of measure spaces. The construction of the collection $\mathscr{C}$ of compact sets proceeds exactly as in Theorem
\ref{radlocal}. For proving (c)  notice that property (vi) of Definition \ref{def-Radon} ensures that
in our context each compact set lives inside an open set with finite measure. This fact then ensures that the proof of
(c) goes through. The assumed inner regularity of all measures described by Definition \ref{def-Radon} ensures that the proof
of part (d) also works in the present context.  That then leaves (e) to consider.
In proving the ``only if'' part there, we have (ii) of Definition \ref{def-Radon} to work with rather than part (c) of
Proposition \ref{Cohn7.5.1}.  In the argument used to prove the implication (c)$\Rightarrow$(b) in Proposition \ref
{Cohn7.5.1} we simply replace the open set $U$ with a measurable set $F$ of finite measure.  It is then an exercise to see that that
argument suffices to show that if
$E\cap K$ is measurable for any compact set, we will then also
have that $E\cap F$ is measurable for any measurable set $F$ of finite measure. Condition (ii) in Definition \ref{def-Radon}
then ensures that this is enough to complete the proof.
\end{proof}

In the unpublished manuscript \cite{Pedeu}, Pedersen uses the Radon-Nikodym theorem to give a very short proof (which he attributes to a paper of
von Neumann cited there) of the uniqueness of Haar measure for $\sigma$-compact locally compact groups (this is  the last 
proof there).  He is only able to treat $\sigma$-compact groups 
since he is using the usual $\sigma$-finite version of the Radon-Nikodym theorem. 
This raises the question of whether one can 
make this proof work for general locally compact groups $G$. We do this below in the setting of Fremlin's Radon measures,
thereby presenting an excellent application of both our general Radon-Nikodym theorem \ref{rn} and of 
the topics of the present section. See 442B in \cite{Fremlin} for Fremlin's proof of the uniqueness. Fremlin defines 
Haar measure to be a left invariant Radon measure.  He shows in  \cite[443J(b-i)]{Fremlin} that the domain $\sigma$-algebra
of a Haar measure is the completion of the Borel $\sigma$-algebra $\B(G)$, and we will include this in the 
definition of Haar measure.

\begin{theorem} \label{Haar}     Haar measure on a locally compact topological group $G$ is unique up to a positive scalar multiple.  \end{theorem}

\begin{proof}  Let $\mu, \nu$ be Haar measures on $G$, 
with corresponding functionals $\varphi_\mu, \varphi_\nu$ on $\mathcal{K}(G)$.   Then $\varphi_\mu + \varphi_\nu$ is an invariant functional on $\mathcal{K}(G),$ so corresponds by the Riesz representation theorem
(see the lines below Definition \ref{def-Radon}) to a unique Radon measure $\omega$ on $G$.   
Then $\omega$ is left invariant (e.g.\ by 
\cite[441L]{Fremlin}), i.e.\ is a Haar measure. 
 Let $\mathcal{A}$ be the domain $\sigma$-algebra of  $\omega$, which is the completion of $\B(G)$
 with respect to $\omega$.
By the formulae above relating functionals on $\mathcal{K}(G)$ and  Radon measures, or by 
416E in  \cite{Fremlin}, 
 it is easy to see that  $\mu \leq \omega$ on open and compact 
sets, and hence also on $\B(G)$. 
It follows that the domain of $\mu$ includes $\mathcal{A}$.   Let 
$\mu' = \mu_{|\mathcal{A}}$.  
If $\mu'(E) > 0$ then there is a compact subset $K \subset E$ with $\mu(K) > 0$.   
Since $\mu(K)$ and $\omega(K)$ are finite, 
$\mu'$ is strongly $\omega$-semifinite.  Thus we may apply the Radon-Nikodym theorem \ref{rn} to obtain an $\mathcal{A}$-measurable $h : G \to [0,\infty)$ with $d \mu' = h \, d \omega$ on $\mathcal{A}$. 
We have $h \leq 1$ a.e.,  this may be seen 
 e.g.\ as  in the proof that (vi) implies (i) of Theorem \ref{frec}.    Then 
$\int \, f \, h \, d \omega = \int \, f \, d \mu'$ for $f \in \mathcal{K}(G)$.  For each $f\in  \mathcal{K}(G)$ and each $z\in G$ we have that 
$$\int_G f(x) \, h(x)\,d\omega(x) =\int_G f(y)\,d\mu'(y) =\int_G f(zy)\,d\mu'(y)$$ $$=\int_G f(zx) \, h(x)\,d\omega(x) =\int_G f(x) \, h(z^{-1}x)\,d\omega(x).$$ So
$\int_G f(x)(h(x)-h(z^{-1}x))\,d\omega(x)=0$ for each $f\in \mathcal{K}(G)$, which ensures that
 $h = h_z$ $\omega$-a.e., for all $z \in G$.

We show that $h$ is constant $\omega$-a.e.  To this end, assume first that $h$ is Borel measurable.
Let $K_0$ and $K_1$ be arbitrary compact sets in $G$, then $0 = (h_z(x)-h(x)) \, \chi_{K_0}(x) \, \chi_{K_1}(z^{-1})$ 
 $\omega$-a.e.\ for all $z \in G$.  Now $h(z^{-1} x)$ is Borel measurable  on $G \times G$. 
   Hence $|h_z(x)-h(x)| \, \chi_{K_0}(x) \, \chi_{K_1}(z^{-1})$ is a non-negative measurable function.   
 By Fubini's theorem for finite measures, 
$$\int_{K_1^{-1}} \, \int_{K_0} \, |h_z(x)-h(x)|\ d \omega(x) \, d \omega(z) = 
 \int_{K_0} \, \int_{K_1^{-1}} \,  |h_z(x)-h(x)|\ d \omega(z) \, d \omega(x)$$ is zero.
Thus for a.e.\ $x \in K_0$ we have 
$$0 = \int_{G} \,  |h_z(x)-h(x)| \, \chi_{K_1}(z^{-1})  \,d \omega(z) = \int_G \, |h(z^{-1} ) - h(x)|\, \chi_{K_1 x}(z^{-1}) \, d \omega(z) .$$
We may therefore  select $x_0 \in K_0$ so that $h = h(x_0)$ $\omega$-a.e.\ on $K$, where $K = x_0^{-1} K_1^{-1}$.
Since $K_1$ was an  arbitrary  compact set, $h$ is constant $\omega$-a.e.\ on every compact subset $K$ of $G$.  
Thus $h$ is constant $\omega$-a.e.\ on $G$.   

In the general case 
there is a Borel measurable function $k$ on $G$ that is $\omega$-a.e.\ equal to $h$ (see e.g.\ \cite[Proposition 2.12]{Fol}).  
 By the above $k$, and hence $h$,  is constant $\omega$-a.e.. 
So $\mu' = c \omega$ for a constant $c \in [0,1]$, and thus $\mu = c \omega$ on bounded Borel functions.  Hence  $\varphi_\mu = c \varphi_\omega = c   \varphi_\mu + c \varphi_\nu$.      
It follows that $\mu$  is a constant times $\nu$ on $\mathcal{K}(G)$, hence on Borel sets
by the formulae above relating functionals and measures. 
Taking completions, $\mu$  is a constant times $\nu$.  
\end{proof}

\section{Characterizing measure algebras and abelian von Neumann algebras}
\label{chmav}

\subsection{Measure algebras and localizability} \label{rmal}

In this subsection we formally introduce the notion of a localizable measure algebra and show that it is once again Radon and
decomposable measure spaces that in some sense
 characterize this class. To fully comprehend the subtleties of this section,
we shall need a bit of background on Boolean algebras and their Stone
spaces. We proceed to summarise the essentials. In so doing we shall closely follow the exposition of \cite{Fremlin}. Fuller details may be found in 311A, 311E, 311F, 311H, 311I, 312M, 321J, 314P and 314S of \cite{Fremlin}.

\begin{definition} We define a Boolean ring $\Balg$, to be a ring $(\Balg,\boxplus,\boxdot)$ for which we have that
$a\boxdot a=a$ for every $a\in\Balg$. If
 the ring admits a multiplicative identity, we shall refer to $\Balg$ as a
Boolean algebra. Elements $a, b \in \Balg$ are called disjoint if $a \boxdot b = 0$, and a disjoint family of subsets has
pairwise disjoint members.

For Boolean algebras we define a Boolean homomorphism to be a ring homomorphism which preserves the multiplicative identity.
In fact the natural definitions of a homomorphism, an isomorphism, an ideal and a principal ideal all come from treating
a Boolean algebra as a ring with unit.
\end{definition}

It is not difficult to see that by definition all Boolean algebras are automatically commutative, and to conclude that $a\boxplus a=0$ for all $a\in \Balg$.

\begin{remark} Readers mystified by the fact that an object with so simple a definition carries the title of ``algebra'', may wish to verify that for any Boolean algebra $\Balg$, the operations $\vee$, $\wedge$ and $x\to x'$ defined by $$x\vee y= (x\boxplus y)\boxplus \left(x\boxdot y\right),\quad x\wedge y=x\boxdot y \mbox{ and }x'=1\boxminus x,$$ satisfy the conditions that
\begin{itemize}
\item $\vee$ and $\wedge$ are associative and commutative,
\item $x\wedge x=x$ and $x\vee x=x$,
\item $x\vee(x\wedge y)= x\wedge(x\vee y)$,
\item $(x\vee y)'=x'\wedge y'$ and $(x\wedge y)'=x'\vee y'$,
\item $x\wedge(y\vee z)=(x\wedge y)\vee(y\wedge z)$ and $x\vee(y\wedge z)=(x\vee y)\wedge(y\vee z)$,
\item $x\vee 0=x$ and $x\wedge 1=x$,
\item $x\wedge y=1$ and $x\vee y=0$ if and only if $y=x'$.
\end{itemize}
Conversely any object $\Balg$ equipped with operations $\vee$, $\wedge$ and $x\to x'$ which satisfy the above criteria,
becomes a Boolean algebra when the operations ``$+$'' and ``$\boxdot $'' are defined by $x\boxplus y=(x\vee y)\wedge(\neg (x\wedge y))$ and $x\boxdot y=x\wedge y$.
\end{remark}

\begin{remark} Each Boolean algebra $\Balg$ admits a natural ordering where we say that $a\leq b$ precisely when
$a\boxdot b=a$. This can easily be seen to be a partial ordering with least element 0 and greatest element 1. In fact under
this ordering $\Balg$ turns out to be a lattice with $a\vee b=\sup\{a,b\}$ and $a\wedge b =\inf\{a,b\}$ for all $a$,
$b\in\Balg$. It also clearly follows from the definition of this ordering that any Boolean homomorphism preserves this order.
Given $a\in \Balg$, one may then further show that the principal ideal $\Balg_a$ generated by $a$ can be shown to be
precisely $\Balg_a=\{b\in \B\colon b\leq a\}$.
\end{remark}

We pause to introduce the concepts that will form the foundation of our further analysis.

\begin{definition}
\label{mud}
\begin{itemize}
\item We say that a Boolean algebra $\Balg$ is \emph{Dedekind complete}, if with respect to the above ordering
every increasing net has a least upper bound. (By the argument in the third paragraph
of Section \ref{LinftyW1}), it follows that every subset will then have a least upper bound.) If by contrast we only have that every countable increasing net has an upper bound, then $\Balg$ is said to be Dedekind $\sigma$-complete (or just $\sigma$-complete).
\item A subset $C$ of a Boolean algebra $\Balg$ is called sequentially
order-closed (resp.\ order-closed) if any increasing sequence in $C$ (resp.\ an upward-directed net in $C$) has a least upper bound in $\Balg$ that belongs to $C$. By the same token a Boolean homomorphism which preserves suprema of increasing nets is said to be order continuous. If it only preserves suprema of countable increasing nets we say it is $\sigma$-order continuous.
\item A subset $\mathcal{I}$ of a Boolean algebra $\Balg$ is said to be an ideal of $\Balg$ if for all $x, y\in \mathcal{I}$ and $a\in\Balg$ we have that $x\vee y\in \mathcal{I}$ and $x\wedge a\in\mathcal{I}$. If an ideal is sequentially order-closed, we say it is a $\sigma$-ideal.
\end{itemize}
\end{definition}

It is an exercise to see that if $\mathcal{I}$ is an ideal of some Boolean algebra $\Balg$,
then $\Balg/\mathcal{I}$ will then again be a Boolean algebra. Indeed if $\Balg$ is $\sigma$-complete and $\mathcal{I}$ a $\sigma$-ideal, then $\B/\mathcal{I}$ is also $\sigma$-complete.

We are particularly interested in a specific subcategory of Boolean algebras, the so-called measure algebras.

\begin{definition}
\label{mud2}
\begin{itemize}
\item A Dedekind $\sigma$-complete Boolean algebra $\Malg$ is said to be a \emph{measure algebra} if it admits a function $\overline{\mu}:\Malg\to[0,\infty]$ with the properties that $\overline{\mu}(p)=0$ precisely when $p=0$ and that for any countable set $\{p_n\}$ of mutually disjoint elements it holds that $\overline{\mu}(\sup_n \, p_n)= \sum_{n=1}^\infty \overline{\mu}(p_n)$. The quantity $\overline{\mu}$ is referred to as a measure on $\Malg$, and is said to be 
 \emph{semifinite} if for every nonzero $p\in\Malg$, we can find a nonzero element $q$ with $q\leq p$ such that $\overline{\mu}(q)<\infty$.
\item A measure algebra $(\Malg, \overline{\mu})$ is said to be \emph{finite} if $\overline{\mu}(\I)<\infty$, and a
\emph{probability} algebra if $\overline{\mu}(\I)=1$.
\item Two measure algebras $(\Malg_0, \overline{\mu}_0), (\Malg_1, \overline{\mu}_1)$ are said to be isomorphic if there exists a Boolean isomorphism $\pi$ from $\A_0$ to $\Malg_1$ such that $\overline{\mu}_1\circ\pi = \overline{\mu}_0$.
\item A measure algebra $(\Malg, \overline{\mu})$ is said to be \emph{localizable} if 
$\Malg$ is Dedekind complete and $\overline{\mu}$ is semifinite.
\end{itemize}
\end{definition}

 \begin{example} We present an important example to illustrate the above concepts. Any measure space $(X,\Sigma, \mu)$ is a
Boolean algebra with the algebra operations given by $$E\wedge F=E\cap F,\quad E\vee F=E\cup F,\quad E'=X\setminus E,\quad \I=X,\quad 0=\emptyset$$for all $E,F\in \Sigma$. The collection $\mathcal{Z}(\Sigma)$ of sets of measure 0, can be shown to
be a $\sigma$-ideal of this Boolean algebra, which then ensures that the quotient $\Sigma/\mathcal{Z}(\Sigma)$ is again a
Boolean algebra with algebra operations given by
$$[E]\vee[F]=[E\cup F] , \; \; [E]\wedge[F]= [E\cap F], \; \; \text{and} \; \; [E]'=[X\setminus E] .$$
The action of $\mu$ canonically extends to $\Sigma/\mathcal{Z}(\Sigma)$, by means of the prescription
$\overline{\mu}([E])=\mu(E)$. The pair $(\Sigma/\mathcal{Z}(\Sigma), \overline{\mu})$ is then a measure algebra. From the
foregoing discussion, it is clear that $\Sigma/\mathcal{Z}(\Sigma)$ will be $\sigma$-complete whenever $(X,\Sigma, \mu)$ is,
and that $\overline{\mu}$ will be semifinite precisely when $\mu$ is. It follows that $(\Sigma/\mathcal{Z}(\Sigma), \overline{\mu})$ is a localizable measure algebra precisely when $(X,\Sigma, \mu)$ is a localizable measure space. For 
simplicity we will in future denote the quotient map $E\to [E]$ by $E\to \overline{E}$.
\end{example}

In view of the importance of the above, we formalise it as a definition.

\begin{definition}\label{madef}
The measure algebra $(\Malg, \overline{\mu})$ of a measure space $(X,\Sigma,\mu)$ is defined to be the Boolean algebra
$\Malg=\Sigma/\mathcal{Z}(\Sigma)$ where $\mathcal{Z}(\Sigma)$ are the sets of measure 0, equipped with the measure $\overline{\mu}$ where $\overline{\mu}$ is defined by $\overline{\mu}([E])=\mu(E)$ for some representative $E$ from $[E]\in \Sigma/\mathcal{Z}(\Sigma)$. We shall refer to such measure algebras as concrete measure algebras.
\end{definition}

\begin{example}\label{vnama}
A second example of a localizable measure algebra is given by $(\Pdb(\M), \tau_{|\Pdb(\M)})$ where $\M$ is an abelian von Neumann algebra, and $\tau$ a faithful normal semifinite trace on $\M$. The lattice of projections $\Pdb(\M)$ is easily seen to be a Boolean algebra when equipped with the operations
$p \boxplus q := p + q - 2pq$ and $p \boxminus q := pq$. Note that then $p\vee q = p + q - pq$ and $p\wedge q = pq$.
This Boolean algebra is Dedekind complete (see for example \cite[Proposition 5.1.8]{KR},
or alternatively Corollary \ref{awvn} and the measure space example above)
 and $\tau_{|\Pdb(\M)}$ is semifinite (see for example \cite[Definition 2.14 \& Proposition 2.13]{GL}).
More generally, a set of commuting projections
in a von Neumann algebra is a  Boolean algebra if it contains $0$, and is closed under $\perp$ (i.e.\ $p \mapsto p^\perp = I - p$), and infima (i.e.\ products $pq$), and it is a complete Boolean algebra if in addition it is closed under suprema of
increasing nets.
\end{example}

The importance of the two examples of localizable algebras presented above, stems from the functorial correspondence between the important objects we now define,
and localizable measure algebras. Establishing this correspondence is one of the main goals of the section.

\begin{definition}\label{defvnma} We call a couple consisting of a von Neumann algebra $\M$ endowed with a faithful normal semifinite weight
$\varphi$ a \emph{von Neumann measure algebra}. If indeed $\varphi$ is a (faithful normal) state, we shall refer to the pair $(\M,\varphi)$ as a \emph{von Neumann probability algebra}. \end{definition}

The choice of calling pairs $(\M,\varphi)$ of the above form von Neumann measure algebras rather than something like
\emph{noncommutative measure spaces} is deliberate for the sake of emphasising the correspondence between localizable measure algebras and abelian von Neumann measure algebras.

\begin{definition}\label{defvnmamor}
We call von Neumann measure algebras $(\M_1,\vf_{\M_1})$ and $(\M_2,\vf_{\M_2})$  \emph{isomorphic} if there exists a 
 $*$-isomorphism $\vniso\colon \M_1 \to \M_2$ such that $\vf_{\M_1}=\vf_{\M_2}\circ\vniso$ on $ (\M_1)_+$. 
A \emph{direct sum} of von Neumann measure algebras $(\M_i,\vf_i)$ is the von Neumann measure algebra $(\M,\vf_\M)$ such that $\M$ is the direct sum 
$\oplus_i \,  \M_i$ of $\M_i$ and $\vf_\M$ is given by $ \vf_\M (
(a_i))=\sum_i \,  \vf_i(a_i)$.
\end{definition}

One may use $W^*$-algebras in the last two definitions and
for some purposes this might be preferable (if one wishes to include the $L^\infty$ spaces of localizable measures).    However for operator algebraists,   the distinction between abelian $W^*$-algebras and 
 abelian von Neumann algebras is essentially cosmetic (see  Corollary \ref{awvn}).

To set the scene for what is to come, we note the following.
If $(X,\Sigma,\mu)$ is a measure space, then as a $C^*$-algebra, the projection lattice $\Pdb(\M)$ of
$\M=L^\infty(X,\Sigma,\mu)$ can be shown to be isomorphic to
$\Sigma/\mathcal{Z}(\Sigma)$. So if $\bar{\mu}$ is defined as
in Definition \ref{madef}, then $\bar{\mu}$ is a measure on $\Pdb(\M)$ with respect to which $(\Pdb(\M), \bar{\mu})$ is
clearly a measure algebra. We leave it as an exercise to check that the measure algebra of  $(X,\Sigma,\mu)$ is isomorphic in the
above sense to $(\Pdb(\M), \bar{\mu})$.   This makes it clear that $(X,\Sigma,\mu)$ is localizable  (resp.\ Dedekind)
in the sense of Section \ref{Ws2} 
if and only if its  measure algebra  is localizable in the sense of Definition \ref{mud2} (resp.\ Dedekind complete). Indeed
 some of these concepts from the early sections of our paper may well be deemed to be somewhat artificial from the viewpoint of someone well
versed in Boolean algebras, as they are  statements about the associated measure algebra. Moreover several results in the first
sections of our paper may be rephrased in measure algebra terms.  For example if one also uses Proposition \ref{Wextn} then one can see that Lemma
\ref{isdual} may be restated as Proposition \ref{prophomordcont} below.  Similarly, Theorem \ref{loc} (iv)  may be recast as
the statement that dualizability/localizability of a measure space may be characterized by its measure algebra being
isomorphic to the measure algebra of a decomposable measure space (see Theorem \ref{Radmeasalg} below).   The `essential suprema' in
Lemma \ref{mu} ff.\ correspond to suprema in the measure algebra.

\begin{remark}\label{homordcont} It is useful to note that a Boolean isomorphism $\pi$ between Dedekind complete Boolean
algebras $\Balg_0$ and $\Balg_1$, is automatically order
bicontinuous. To see this notice that the fact that both $\pi$ and
$\pi^{-1}$ preserves order, ensures that $(a_\alpha)\subset \Balg_0$ is an increasing net if and only if $(\pi(a_\alpha))$ is.
Since  $a_\alpha\leq \sup_\beta a_\beta$ for each $\alpha$,
 we have $\pi(a_\alpha)\leq \pi(\sup_\beta a_\beta)$, and hence that $\sup_\alpha \, \pi(a_\alpha) \leq \pi(\sup_\beta a_\beta)$. Repeating this argument for $\pi^{-1}$
instead of $\pi$ shows that $\sup_\beta a_\beta \leq \pi^{-1}(\sup_\alpha \, \pi(a_\alpha))$, and hence that
$\pi(\sup_\beta a_\beta) \leq \sup_\alpha \, \pi(a_\alpha)$. We therefore have $\sup_\alpha \, \pi(a_\alpha)= \pi(\sup_\beta a_\beta)$ as required.
\end{remark}

Any measure preserving Boolean isomorphism between measure algebras $(\Malg_0, \overline{\mu}_0)$ and
$(\Malg_1, \overline{\mu}_1)$, clearly preserves semifiniteness. It is not difficult to modify the argument at the end of the preceding remark to show that
measure preserving Boolean isomorphisms preserve Dedekind completeness.
 We record this observation as a proposition.

 \begin{proposition}\label{prophomordcont}
Let $(\Malg_0, \overline{\mu}_0)$ and $(\Malg_1, \overline{\mu}_1)$ be isomorphic measure algebras. Then $(\Malg_0, \overline{\mu}_0)$ is localizable if and only if $(\Malg_1, \overline{\mu}_1)$ is.
\end{proposition}

For any point set $X$, the power set 
$\mathcal{P} X$ may be realized as a Boolean algebra by the prescriptions $X=1$,
$\emptyset =0$, $E\boxplus F=E\Delta F$ (where $E\Delta F$ denotes the symmetric difference of $E$ and $F$), and
$E\boxdot F= E\cap F$. A deep fact regarding Boolean algebras, verified by Marshall Stone, is that every Boolean algebra admits
a representation as a subalgebra of the power set of some point set $Z$. We pause to record this fact.

\begin{theorem}[Stone's theorem]\label{StoneBoole}  Let $\Balg$ be a Boolean algebra, and let $Z$ be the set of ring
homomorphisms from $\Balg$ onto the two point Boolean algebra $\mathbb{Z}_2\equiv\{0,1\}$. Then the map
$\Balg\to\mathcal{P}Z:a\mapsto\widehat{a}$ given by $\widehat{a}=\{z:z\in Z,\,z(a)=1\}$, is an injective Boolean homomorphism
with $\widehat{1}_{\Balg}=Z$.
\end{theorem}

On the basis of the above theorem, the {\em Stone space} of a Boolean algebra $\Balg$ is defined to be the set $Z(\Balg)=Z$
of non-zero ring homomorphisms from $\Balg$ to 
$\mathbb{Z}_2$, with the map described in the above theorem referred to as the
\emph{Stone representation} of $\Balg$. The Stone space admits a canonical topology, which is elucidated by the following
theorem.

\begin{theorem}[Stone space topology] Let $Z$ be the Stone space of a Boolean algebra $\Balg$, and let
$$\mathfrak{T}=\{G:G\subseteq Z, \mbox{ for every }z\in G\mbox{ there exists }a\in\Balg \mbox{ such that } z\in\widehat{a}\subseteq G\}.$$ Then $\mathfrak{T}$ is a topology on $Z$, under which $Z$ is a compact 
Hausdorff space 
which is zero-dimensional, that is, the collection $\mathcal{O}$ of clopen subsets of $Z$ are a basis.
Moreover,  $\mathcal{O}=\{\widehat{a}:a\in\Balg\}$.
\end{theorem}

Proofs of the last results 
may be found
in \cite[311]{Fremlin}.   
Putting them  together, we see that every Boolean algebra is representable as  the Boolean algebra
of all clopen subsets of a zero-dimensional compact space.

 We may exploit the structure offered by the Stone space to establish the following very important result. We shall content ourselves with merely outlining the proof. Details may be found in \cite[314M, 321J]{Fremlin}.

\begin{theorem}[Loomis-Sikorski]\label{LSthm}
Every $\sigma$-complete measure algebra is isomorphic to a concrete measure algebra.
\end{theorem}

\begin{proof}[Proof outline--see
for example 321J in \cite{Fremlin} for more details if needed]
Given a $\sigma$-complete measure algebra $(\Malg,\overline{\mu})$, we pass to the Stone space $Z$, and let $\mathcal{O}$ be the algebra of clopen subsets of $Z$. This algebra is of course nothing but $\{\widehat{a}:a\in\Malg \}$. The collection
$\mathcal{I}$ of meager subsets of $Z$ turns out to be a $\sigma$-ideal of the Boolean algebra $\mathcal{P}Z$. Now let $\Sigma$ be the algebra $$\Sigma=\{C \in \mathcal{O}\colon\mbox{there exists }a\in\Malg\mbox{ such that } C \, \Delta \,
\widehat{a} \in \mathcal{I}\}.$$
It is an exercise to see that $\Sigma$ is a $\sigma$-algebra of subsets of $Z$ which contains the $\sigma$-ideal $\mathcal{I}$. The fact that $\Sigma$ is $\sigma$-complete and $\mathcal{I}$ a $\sigma$-ideal, ensures that $\Sigma/\mathcal{I}$ is $\sigma$-complete. We know that $\Malg$ and $\mathcal{O}$ are isomorphic as Boolean algebras. The next step is to show that $\mathcal{O}$ is isomorphic to the Boolean algebra $\Sigma/\mathcal{I}$. This is achieved by showing that for every $E\in \Sigma$ there is exactly one $F\in\mathcal{O}$ for which
$E \Delta F\in \mathcal{I}$. The action of this isomorphism then identifies $F$ with $[E]$. It follows that $\Malg$ and $\Sigma/\mathcal{I}$ are isomorphic as Boolean algebras.  What we still need is a measure on $(Z,\Sigma)$ for which $\mathcal{I}$ are the sets of measure zero. With $\theta$ denoting the isomorphism from $\Sigma/\mathcal{I}$ onto $\Malg$, the quantity $\nu(E)=\overline{\mu}(\theta([E]))$, where $E\in \Sigma$, is the measure we seek.

Consider the map $\pi:C\to\Malg$ where we say that
$\pi(C)=a$ precisely when $C \, \Delta \,
\widehat{a} \in \mathcal{I}$ for some $a\in \Malg$. Checking reveals that $\pi$ is a well-defined surjective homomorphism for which the kernel is precisely $\mathcal{I}$. The $\sigma$-completeness of $\mathcal{I}$ then ensures that $\pi$ is $\sigma$-order continuous. This fact in turn guarantees that the action of $\nu=\overline{\mu}\circ\pi$ on $C$ yields a well-defined measure on $(Z,C)$. Hence $(Z,C,\nu)$ is a measure space for which $\pi$ identifies $(\Malg,\overline{\mu})$ with the measure algebra of $(Z,C,\nu)$.
\end{proof}

\begin{proposition}\label{Wextn} Let $\M_1$ and $\M_2$ be  abelian $W^*$-algebras equipped with faithful normal semifinite traces $\tau_1$ and $\tau_2$.   Every Boolean homomorphism $\theta$ from $\Pdb(\M_1)$ to $\Pdb(\M_2)$ extends to a $*$-homomorphism $\mathcal{I} : \M_1 \to \M_2$. Moreover
$\theta$ is an isomorphism if and only if $\mathcal{I}$ is an isometric  isomorphism.
\end{proposition}

\begin{proof}
Define a mapping $\mathscr{I}$ on the span of  $\Pdb(\M_1)$ by setting $\mathscr{I}(s)=\sum_{i=1}^n \lambda_i \theta(e_i)$
for any ``simple function'' $s=\sum_{i=1}^n \lambda_i e_i$. (Note that when we use the term ``simple
function'' we understand $e_1,\dots,e_n$ to be mutually orthogonal projections.) Given another representation $\sum_{j=1}^m \alpha_j f_j$ of $s$ we may pass to a common refinement of the `sets' $e_1,\dots,e_n$ and $f_1,\dots,f_m$ to see that
$\mathscr{I}(\sum_{i=1}^n \lambda_i e_i)= \mathscr{I}(\sum_{j=1}^m \alpha_j f_j)$. So $\mathscr{I}$ is a well-defined bijective linear map from the simple functions of $\M_1$ to those of $\M_2$. The map
$\mathscr{I}$ clearly preserves adjoints, and squares of self-adjoint simple functions. Hence it is a
$*$-isomorphism. For any positive simple function $s=\sum_{i=1}^n \lambda_i e_i$, we have that
$\|s\|_\infty=\max(\lambda_1,\dots,\lambda_n)$. This ensures that then $$\|\mathscr{I}(s)\|_\infty=\|\sum_{i=1}^n \lambda_i \theta(e_i)\|_\infty\leq \|s\|_\infty
\, \|\sum_{i=1}^n\theta(e_i)\|_\infty =\|s\|_\infty.$$ By
Proposition \ref{sa3} 
 $\mathrm{span}(\Pdb(\M_1))$ is norm-dense in $\M_1$. Thus  $\mathscr{I}$ extends to a *-homomorphism from $\M_1$ into $\M_2$.

If $\theta$ is an isomorphism, then for the induced map on the simple functions we would similarly have that
$\|s\|_\infty\leq\|\mathscr{I}(s)\|_\infty$. Hence $\mathscr{I}$ is an isometric isomorphism between
the simple functions. In this case the simple functions of $\M_1$ of course map onto those of $\M_2$, and so the norm density of $\mathrm{span}(\Pdb(\M_i))$ in $\M_i$ ($i=1, 2$), ensures that here $\mathscr{I}$ extends to an isometric $*$-isomorphism from $\M_1$ into $\M_2$. \end{proof}

We are now ready to establish the main theorem of this subsection. This
will affirm the status of Radon measure
spaces as a class of well-described measure spaces unifying the concepts of `localizable' and `decomposable'.

\begin{theorem}\label{Radmeasalg} For a measure algebra $(\Malg,\overline{\mu})$, the following are equivalent:
\begin{enumerate}
\item $(\Malg,\overline{\mu})$ is localizable;
\item $(\Malg,\overline{\mu})$ is isomorphic to the measure algebra of a localizable measure space;
\item $(\Malg,\overline{\mu})$ is isomorphic to the measure algebra of a decomposable measure space;
\item $(\Malg,\overline{\mu})$ is isomorphic to the measure algebra of a Radon measure space $(X,\mathfrak{T}, \mathscr{B},\mu)$,
\end{enumerate}
\end{theorem}

\begin{proof} Any Radon measure space is decomposable by Theorem \ref{Fremlin 1.10}, and 
any decomposable measure space is localizable by
Theorem \ref{loc} (iv). 
 The implications (4)$\Rightarrow$(3)$\Rightarrow$(2) are therefore immediate, with the verification of (2)$\Rightarrow$(1) is an almost trivial exercise.   Then (1) $\Rightarrow$(2) by Theorem \ref{LSthm} and Proposition \ref{prophomordcont}.

It therefore remains to prove the implication (2)$\Rightarrow$(4).
The measure algebra of a finite measure space $(\Omega,\Sigma,\mu)$ is isomorphic  to the measure algebra
of a compact Radon measure space in the sense of
\ref{def-Radon}.   Indeed by the Loomis-Sikorski theorem \ref{LSthm}
 and its proof, up to  measure algebra isomorphism
 we may take $\Omega$ to be a Stone space,
and $\Sigma$ to
consist of $C \Delta F$ for $C$ clopen and $F$ meager.  Also such $F$ are  $\mu$-null.
Since $\mu$ is finite, items (ii) and (vi) in Definition \ref{def-Radon}
are trivial, as is (iii).  Since any open set $U$ equals $\bar{U} \Delta F$ for  the
meager set $F = \bar{U} \setminus U$, and since $\mu(C \Delta F) = \mu(C)$,
we also have items (iv) and (v) in that definition.   Items (ii)--(vi) are easily seen to remain
true for the completion of $\mu$, and hence we may assume that (i) also holds.
Thus we have a compact (finite) Radon measure space in the sense of
\ref{def-Radon}.

Let $(X,\Sigma,\nu)$ be a localizable measure space.
As in the proof of Theorem \ref{loc} (iii) we have $\M = L^\infty(X,\nu) \cong
\oplus_i \, \M p_i$ isometrically
via a measure preserving $*$-isomorphism,  with $\M p_i$
isometrically $*$-isomorphic to the $L^\infty$ algebra of a
finite measure space $(\Omega_i,\Sigma_i,\mu_i)$.   By the last paragraph
and Proposition \ref{Wextn} we may assume that
the latter is a compact (finite) Radon measure space.

Let $(\Omega, \Sigma_0, \mu)$ be the  `disjoint sum' of the compact Radon measure spaces
$(\Omega_i, \Sigma_i,\mu_i)$, as in the proof of Theorem \ref{loc} (iv).    This is clearly a complete measure space: for every  subset $E$ of a $\nu$-null set in $\Sigma_0$,
each $E\cap \Omega_i$ will be  measurable and null since $\mu_i$ is complete.
So $E \in  \Sigma_0$.    We give $\Omega$ the
disjoint sum topology (where a set is open if and only if
its intersection with every $\Omega_{i}$ is open).  This topology is locally compact,
and  is easily seen to be contained in $\Sigma_0$.
Items (ii) and (vi)   in Definition \ref{def-Radon} are clear if we take
the $F = \Omega_i$ there.   In (vi) we assume that $x \in \Omega_i$.
For  (ii), if $E \cap \Omega_i \in \Sigma$ for all $i$ it follows by
the decomposable definition that $E \in \Sigma$.
The validity of (iii) and (iv)  in
 \ref{def-Radon} is clear.  It is an
exercise to see that (v) holds too, using the fact that the subsets
$ \Omega_i$ form an open covering of any compact set, hence have a finite subcover.
Thus  $(\Omega, \Sigma_0, \mu)$ is a Radon measure space.
As in the proof of Theorem \ref{loc} (iv)
we have a measure preserving $*$-isomorphism $\M \cong \N$, for $\N = L^\infty(\Omega,\Sigma_0,\mu)$.
This induces measure algebra isomorphisms
between $(\Pdb(\M),\bar{\nu})$ and $(\Pdb(\N),\bar{\mu})$,
and hence between  the measure algebra of $(X,\Sigma,\nu)$
and the measure algebra
of the Radon measure space.
\end{proof}

Note that the Radon measure space constructed in the last proof is decomposable by construction.

\subsection{Abelian von Neumann algebras} \label{avna}

\begin{lemma} \label{sa2} If $K$ is a compact Stonean  space with a Radon probability measure $\mu$ then every nowhere dense set is $\mu$-null.  \end{lemma}

\begin{proof}  Since the closure of a nowhere dense set is nowhere dense we may assume that our nowhere dense set $E$ is
closed. Since the clopen sets are easily seen to form a basis, we can find an increasing net of clopen sets $U_t$ with union
$E^c$.  Let $f_t = \chi_{U_t}$ then $f_t \nearrow \chi_{E^c}$ pointwise.   Any $f \in {\rm Ball}(C(K))_+$ with $f \geq f_t$
for all $t$, equals $1$ on dense $E^c$, so is $1$.   So $1 = \sup_t \, \mu(U_t) \leq  \mu(E^c) \leq 1$, and hence  $\mu(E) = 0$.
\end{proof}

We apply the above technology to obtain Dixmier's
characterization of commutative von Neumann algebras.

\begin{theorem}[Dixmier] \label{Dix} If $C(K)$ has sufficiently many normal states and any bounded increasing net of positive elements
in $C(K)$ has a supremum, (or equivalently and by definition, $K$ is hyperstonean), then $C(K)$ is a $W^*$-algebra. Indeed
$C(K) \cong L^\infty(X, \mu)$ isometrically $*$-isomorphically for a localizable (even Radon in the  sense
of {\rm \ref{def-Radon}}) measure space
$(X,\mu)$.
\end{theorem}

\begin{proof} Let $\M = C(K)$ (or indeed a commutative unital $C^*$-algebra). For any normal state $\varphi$ on $\M$,
there exists a normal Radon probability measure $\mu = \mu_\varphi$ on $K$ with $\varphi(f) = \int_K \, f \, d \mu$. From a
graduate real variables course we know that every finite Radon measure has a topological support (see for example
\cite[p.\ 215]{Fol}). In our case this is a closed set $C$ with $\mu(K \setminus C) = 0$.  We claim that $C$ is clopen.
Clearly $C^{\circ}$ is clopen, since the closure of the open set $K \setminus C$, which  is clopen
by Lemma \ref{sa1}, is the complement of
$C^{\circ}$. If $x \in E = C \setminus C^{\circ}$ then $\mu(E) = \mu(K  \setminus C^{\circ}) > 0$.  However there are no open
sets within $E$, so that $E^c$ is open and dense, hence $\mu(E) = 0$ by Lemma \ref{sa2}. This contradiction shows that
$C = C^{\circ}$,  and this is clopen.

Take a maximal collection $U_i$ of such mutually disjoint clopen sets in $K$ which support normal measures, and let
$U = \cup_i \, U_i$. We claim that $U$ is open and dense in $K$. If $x \in K \setminus \bar{U}$, then by the Urysohn
lemma  we may choose a function $f \in C(K)_+$ vanishing on $\bar{U}$ and with $f(x) > 0$.  We may then further choose a normal measure $\nu$ with
$\psi_\nu(f) > 0$. If $C$  is the support of $\nu$, then $D = C \setminus \bar{U}$ is clopen, and $\nu_{|D}$ is a nonzero
(since $\int_{D} \, f \, d \nu = \psi_\nu(f) > 0$) is easily seen to be a normal  measure with support disjoint from the
$U_i$, contradicting the maximality. So $U$ is dense.

By Lemma \ref{sa1} the projections are norm dense in $\M$.
One may now check that the sum $\omega = \sum_i \, \omega_i$
of the  normal states above corresponding to the $U_i$, is a faithful normal semifinite trace  on $\M_+$.  For example,
if $f  \in \M_+$ with $\omega(f) = 0$, then $f = 0$ on every $U_i$, so that $f = 0$.
As in the last paragraph of the proof of Lemma
  \ref{WStone}, $\Pdb(\M)$ is Dedekind complete. Thus
$(\Pdb(\M),\omega)$ is a localizable measure algebra.
By Theorem \ref{LSthm} $(\Pdb(\M),\omega)$ is  isomorphic to the concrete measure algebra
of a localizable measure space $(X,\mu)$.
By Theorem \ref{Radmeasalg}  $(X,\mu)$ may be taken to be Radon if desired.
The proof technique  in Proposition
\ref{Wextn} then yields  $\M \cong L^\infty(X,\mu)$ isometrically $*$-isomorphically.
\end{proof}

\begin{corollary} \label{awvn} Every abelian $W^*$-algebra is isometrically $*$-isomorphic to an abelian von Neumann algebra, and also  to $L^\infty(X,\mu)$  for a localizable measure space $(X,\mu)$.
It  also possesses a faithful normal semifinite trace.
\end{corollary}

\begin{proof}   Writing the abelian $W^*$-algebra $\M = C(K)$
 as in Proposition \ref{sa3}, by Lemma \ref{WStone} this satisfies the hypotheses
of the last theorem.  So $\M \cong L^\infty(X,\mu)$ isometrically $*$-isomorphically for a localizable 
(even Radon) measure space $(X,\mu)$.   Also in the last proof  we constructed a faithful normal semifinite trace
$\omega$.
\end{proof}

Operator algebraists usually prove the last two results without 
using the Loomis-Sikorski theorem, by applying the method
in the proof of Theorem \ref{Dix} to obtain the states $\omega_i$ and their mutually
orthogonal support projections,
and then  decomposing $K$ and
$\M$ as in the proof of Theorem \ref{loc} (iv).  In Corollary \ref{awvn}
one begins with the weak* continuous states
(see for example \cite[Proposition 1.18.1]{Sakai} for more details).

\begin{theorem}\label{abelianvNA} For any $W^*$-algebra $\M$ equipped with a faithful normal semifinite trace $\tau$, the following are equivalent:
\begin{enumerate}
\item $\M$ is abelian.
\item $\M$ is $*$-isomorphic to $L^\infty(X,\Sigma,\mu)$ where $(X,\Sigma,\mu)$ is a localizable measure space.
\item $\M$ is $*$-isomorphic to $L^\infty(X,\Sigma,\mu)$ where $(X,\Sigma,\mu)$ is a decomposable measure space.
\item $\M$ is $*$-isomorphic to $L^\infty(X,\mathscr{B},\mu)$ where $(X,\mathfrak{T}, \mathscr{B},\mu)$ is a Radon measure space.
\end{enumerate}
In each of the cases {\rm (2)-(4)}, the $*$-isomorphism $\mathcal{I}$ can be chosen so as to satisfy $\tau=\int_X\mathcal{I}(\cdot)\, d\mu$.
\end{theorem}

\begin{proof} Any Radon measure space is decomposable  and
any decomposable measure space is localizable (see Theorem \ref{Radmeasalg}).
Thus (4)$\Rightarrow$(3)$\Rightarrow$(2)$\Rightarrow$(1).
We show  (1)$\Rightarrow$(4).
Any abelian $W^*$-algebra satisfies the hypotheses of Theorem \ref{Dix},
and by the last two  proofs
 $(\Pdb(\M),\tau)$ is  isomorphic via a measure preserving map
 $\theta$ to the concrete measure algebra
of a localizable measure space $(X,\mu)$.  We may take the latter to be Radon by  Theorem
\ref{Radmeasalg}.  Proposition
\ref{Wextn} gives $\M_1 = L^\infty(X,\mu) \cong \M$ isometrically
via a $*$-isomorphism $\mathcal{I}$.     Note that $*$-isomorphisms are
 automatically normal.  Let $\tau_2 = \psi_\mu$.
For any non-negative simple function
$s= \sum_{i=1}^n \lambda_i e_i$ of $\M_1$ we have
 $$\tau(\mathscr{I}(s))=\sum_{i=1}^n \lambda_i \tau_2(\theta(e_i))=\sum_{i=1}^n \lambda_i \tau_1(e_i) =\tau_1(s).$$ Given any positive $a\in \M_1$ we may select a sequence $(s_n)$ of non-negative simple functions increasing to $a$. The normality of the traces then ensures that $\tau_1(a)=\tau_2(\mathscr{I}(a))$.
  \end{proof}

\subsection{Comparing measure algebras and abelian von Neumann measure algebras}\label{ssmavn}

\begin{theorem}\label{abW}
If $\M, \N$ are abelian $W^*$-algebras and $\omega \in \M^*_+$
then \begin{enumerate}
\item [(i)] $\M$ has a unique predual.
\item [(ii)]
The functional $\omega$ is normal  if and only if  $\omega$ is
weak* continuous.
\item [(iii)]  $\omega$ is completely additive on projections if and only if $\omega$ is normal.
\item [(iv)] An (isometric) $*$-isomorphism $\M \to \N$ is a weak* homeomorphism. \end{enumerate}
\end{theorem}

\begin{proof}  These are well known
to operator algebraists, so we will just
give a good sketch of new proofs appropriate to the commutative case.
The interested reader will be able to fill in all details.
We will twice use our Radon-Nikodym theorem \ref{rn} as a main ingredient.
Indeed by the characterization of normal weights in that result,
(ii) holds if $\M$ is $L^\infty$ of a localizable space.
Thus (ii) holds in general 
by Theorem \ref{abelianvNA} if we can show
that the condition $\tau=\int_X\mathcal{I}(\cdot)\, d\mu$ there implies
that $\mathcal{I}$ is weak* continuous.   To do this we recall from
Theorem \ref{Dix} that $\tau$ is a sum $\sum_i \, \omega_i$ of weak* continuous functionals.
Let $p$ be a $\tau$-finite projection.
It is easy to show that multiplication by $p$ is weak* continuous on $\M$
(for example see bottom p.\ 16 in \cite{Sakai}),
and that therefore  $\tau_p = \sum_i \, \omega_i(p \,
\cdot)$ is a weak* continuous functional.   Using this,
if  $x_t \to x$ weak*  then from the relation $\tau=\int_X\mathcal{I}(\cdot)\, d\mu$ one sees
 $$\langle  \mathcal{I}(x_t) , \mathcal{I}(p) \rangle  = \int_X\mathcal{I}(x_t p) \, d\mu = \tau_p(x_t)
\to \tau_p(x) = \langle  \mathcal{I}(x_t) , \mathcal{I}(p) \rangle.$$ Hence $\mathcal{I}(x_t) \to \mathcal{I}(x)$ weak* (since the simple functions in $L^1(X,\mu)$ are dense).
So $\mathcal{I}$ is weak* continuous as desired.

For (i), it is an exercise that (ii) forces  the predual to be
 the (closed) span of the normal states.
Similarly  (ii) and the fact that $*$-isomorphisms
preserve suprema immediately imply (iv) (e.g.\ see \cite[Proposition 1.57]{GL}).

For (iii), we may again
assume that the $W^*$-algebra is $\M = L^\infty(X , \Sigma, \mu)$ for a localizable measure.
In view of Lemma \ref{lsigf0}, we only need to prove the ``only if'' part.
Define a measure $\nu$
by $\nu(E) = \omega(\chi_E)$.   Then $\nu \ll \mu$ as in the
proof of Lemma \ref{tol}.
By the proof of Lemma \ref{lsigf0}, $\omega$ preserves sup's of increasing nets of projections.
By Zorn's lemma there is a maximal projection $q$ in Ker$(\omega)$.
`Maximal' implies `maximum' here:
 if $r$ is a projection  in Ker$(\omega)$ then $q \leq r + q -rq \in {\rm Ker}(\omega)$ so $r \leq q$.
Let $s(\omega) = 1-q$.
If $E \in \Sigma$ with $0 \neq [\chi_E] \leq s(\omega)$ then $\nu(E)
= \omega(\chi_E) > 0$.  So
 $\nu$ has $\mu$-support.   By the Radon-Nikodym theorem
\ref{rn}, $\psi_\nu$ is normal.  However $\psi_\nu$  equals $\omega$ on   simple measurable functions, and both are continuous, so $\omega = \psi_\nu$ is normal.
\end{proof}

Armed with the above, we are now able to sharpen Proposition \ref{Wextn}.

\begin{proposition}\label{Wextn2} Let $\M_1$ and $\M_2$ be two abelian $W^*$-algebras equipped with faithful normal semifinite traces $\tau_1$ and $\tau_2$.  Let $\theta : \Pdb(\M_1) \to \Pdb(\M_2)$ be a Boolean homomorphism, and  $\mathcal{I} : \M_1 \to \M_2$ its  $*$-homomorphism
extension from Proposition {\rm \ref{Wextn}}.  Then
$\theta$ is order continuous if and only if $\mathcal{I}$ is normal; and in this case
 $\theta$ is measure preserving if and only if $\tau_1=\tau_2\circ\mathscr{I}$.
\end{proposition}

\begin{proof}
Given a set $\{e_i\}_{i\in I}$ of mutually orthogonal projections,
$\sum_{i\in I} \, e_i$ equals $\sup_F \, \sum_{i\in F} \, e_i$ where the supremum is taken over finite subsets $F$ of $I$.
If $\theta$ is normal,  then $$\mathcal{I}(\sum_{i\in I} \, e_i)=\theta(\sum_{i\in I} \, e_i)=\sup_F \, \theta(\sum_{i\in F} \, e_i) ,$$
which equals $\sup_F \sum_{i\in F}\theta(e_i)= \sum_{i\in I}\theta(e_i)=\sum_{i\in F}\mathcal{I}(e_i)$. For a normal state $\omega$ on $\M_2$ we  have  $\omega\circ\mathcal{I}(\sum_{i\in I}e_i)=\omega(\sum_{i\in I}\mathcal{I}(e_i))=\sum_{i\in I}\omega\circ\mathcal{I}(e_i)$.
So $\omega\circ\mathcal{I}$ is  a normal functional by
Theorem \ref{abW} (iii). Since the dual of $\mathcal{I}$ maps normal states to normal functionals,
 it is an exercise that $\mathcal{I}$ is normal.
The second assertion follows as in the proof of  Theorem \ref{abelianvNA} (4).   \end{proof}

Using the insights gleaned from Propositions \ref{Wextn} and \ref{Wextn2}, we arrive at the following remarkable
conclusion:

\begin{theorem}\label{vna=ma} The categories $\mathbf{AvNMA}$ of abelian von Neumann measure algebras and $\mathbf{LMA}$ of localizable measure algebras, are equivalent.
\end{theorem}

\begin{proof} Given two  objects $(\M_1,\tau_1)$ and $(\M_2,\tau_2)$ from $\mathbf{AvNMA}$, the natural morphisms for this
category are normal *-homomorphisms $\mathcal{I}:\M_1\to\M_2$ for which we have that $\tau_2\circ\mathcal{I}=\tau_1$ (see Definition \ref{defvnmamor}). Next
let $\mathbf{LMA}$ be the category of localizable measure algebras. The natural morphisms for this category are the measure
preserving order continuous Boolean homomorphisms. Now let $F$ be the mapping which sends each object $(\M,\tau)$ in
$\mathbf{AvNMA}$ to $(\Pdb(\M),\tau_{|\Pdb(\M)})$ in $\mathbf{LMA}$ and which sends a morphism $\mathcal{I}:(\M_1,\tau_1)\to(\M_2,\tau_2)$ to $F(\mathcal{I})=\mathcal{I}_{|\Pdb(\M_1)}$.
It is now not difficult to conclude from Propositions \ref{Wextn} and
\ref{Wextn2}, that $F$ sends $\mathcal{I}:(\M_1,\tau_1)\to(\M_2,\tau_2)$ to $F(\mathcal{I}):F(\M_1,\tau_1)\to F(\M_2,\tau_2)$
in such a way that
\begin{itemize}
\item $F(\mathrm{id}_{(\M_1,\tau_1)}) =\mathrm{id}_{F(\M,\tau)}$,
\item and $F(\mathcal{I}\circ\mathcal{J})= F(\mathcal{I})\circ F(\mathcal{J})$ for morphisms
$\mathcal{J}:(\M_1,\tau_1)\to(\M_2,\tau_2)$ and $\mathcal{I}:(\M_2,\tau_2)\to(\M_3,\tau_3)$.
\end{itemize}
In other words $F$ is a functor.

It is clear from Propositions \ref{Wextn} and \ref{Wextn2} that $F$ in fact induces a bijection from $\mathrm{Hom}((\M_1,\tau_1),(\M_2,\tau_2))$ to $\mathrm{Hom}((\Pdb(\M_1),\tau_1),(\Pdb(\M_2),\tau_2))$. Hence $F$ is both full and faithful.

In closing we note two facts. Firstly that by Theorem \ref{Radmeasalg}, any localizable measure algebra is isomorphic to $(\Sigma/\mathcal{Z}(\Sigma),\overline{\mu})$ for some localizable measure space $(X,\Sigma,\mu)$, and secondly that for such a measure space the pair $(L^\infty(X,\Sigma,\mu),\int_X
 \, \cdot \, d\mu)$ will by Theorem \ref{frec} and Corollary \ref{intno} then be an abelian von Neumann measure algebra. Since the projection lattice of $L^\infty(X,\Sigma,\mu)$ is a copy of $\Sigma/\mathcal{Z}(\Sigma)$, it follows that $F$ is essentially surjective.

This then suffices to ensure that the categories $\mathbf{AvNMA}$ and $\mathbf{LMA}$ are functorially equivalent.
\end{proof}

There is a version of the last result that does not require weights or measures on the algebras.
Note that the semifinite measures on a Dedekind complete Boolean algebra ${\mathcal B}$
are all `equivalent' in some sense, and constitute a group.
To see this fix one such semifinite
measure $\bar{\mu}$. Then  $({\mathcal B},\bar{\mu})$ is isomorphic to
a concrete measure algebra of a localizable measure space $(X,\Sigma,\mu)$,
with $\bar{\mu}$ corresponding to the restriction of the weight $\omega(f) =
\int_X \, f \, d \mu$,  by  Theorem \ref{LSthm}.   Let $\M$ be the associated
$W^*$-algebra, i.e.\ $L^\infty(X,\Sigma,\mu)$.
Semifinite measures on ${\mathcal B}$ correspond bijectively to
faithful normal semifinite weights on $\M_+$, as may be seen again by  Theorem \ref{LSthm}, by thinking of $\M$ as $C(Z)$ where $Z$ is the Stone space.
By the Radon-Nikodym theorem \ref{rn} such weights on $\M_+$
are all equivalent to $\omega$, and equal  $\omega_h$ for ($\mu$-a.e.\ unique)
strictly positive measurable functions
$h : X \to 
(0,\infty)$.  The ($\mu$-a.e.\ equivalence classes of the)
latter functions on $X$ clearly form a group.

It is therefore natural to define a
category {\bf LBA} consisting of the Dedekind complete Boolean algebras
that admit some (unspecified) localizable measure, together with the order continuous morphisms
of Boolean algebras.    Let {\bf AvNA} be the category of abelian $W^*$-algebras and normal $*$-homomorphisms.

\begin{theorem}  \label{eqca}  {\bf LBA} and {\bf AvNA} are equivalent  categories.    \end{theorem}

This follows from the proof of Theorem \ref{vna=ma}.  See also \cite{DP}; the above shows that
{\bf LBA} is equivalent to the
categories mentioned in the main theorem of that paper.

\begin{remark}
The projection lattice of a direct sum $(\oplus_i\M_i,\oplus_i\tau_i)$ of abelian von Neumann 
measure algebras is clearly of the form $(\prod_i\Pdb(\M_i),\tau)$, where here $\tau$ has the action $\tau((a_i)_{i\in I})=\sum_i \, \tau(a_i)$ (see \cite{Fremlin}, 322L). Moreover if each $(\M_i,\tau_i)$ is isomorphic to say $(L^\infty(X_i,\Sigma_i,\mu_i),\int\cdot \, d\mu_i)$, then $(\oplus_i\M_i,\oplus_i\tau_i)$ will be isomorphic to 
$$(L^\infty(\oplus_i(X_i,\Sigma_i,\mu_i)),\int\cdot  \, d(\oplus_i\mu_i))= (\oplus_i L^\infty(X_i,\Sigma_i,\mu_i),\oplus_i\int\cdot \, d\mu_i).$$ The action of the functor will then associate each $(\Pdb(\M_i),\tau_i)$ with $(\Sigma_i/\mathcal{Z}(\Sigma_i),\overline{\mu_i})$. Thus the functor takes direct sums of von Neumann measure algebras to 
{\em  simple products} of localizable measure algebras, which correspond to disjoint sums of measure spaces.
\end{remark}

\subsection{Commutative versions of theorems of Kadison and Pedersen} \label{kpb}

We  prove here  commutative versions of two deep theorems of Kadison and Pedersen respectively (Theorems 2.4.4 and 2.8.4 in Pedersen's text \cite{Ped}), and some results of Bade.
These may be viewed as  characterizations of `concrete' commutative von Neumann algebras, just as for example Kadison's theorem (see \cite{Kad2} and in particular reference [7] there) 
is usually viewed as a characterization of `concrete'
general von Neumann algebras. It is good to keep in mind that the results become much easier if the measure 
is finite (see the closing remark in this section).

\begin{theorem} \label{subloc} Suppose that we have a unital $C^*$-subalgebra $\N$ of $\M = L^\infty(X,\Sigma,\mu)$ for localizable
$\mu$.  The following are equivalent:
\begin{itemize}
\item [(i)]  $\N$ is norm-densely spanned by its projections,
and $\Pdb(\N)$ is order-closed in  $\Pdb(\M)$.
\item [(ii)]      $\N_{sa}$ is monotone closed in $\M_{sa}$ (that is,  $\N_{sa}$ is closed
under suprema of bounded sets in $\M_{sa}$).
\item [(iii)]  $\N$ is a weak* closed subalgebra of $\M$.
 \item [(iv)]  $\N = L^\infty(X,\Sigma_0,\mu_{|\Sigma_0})$
for a $\sigma$-subalgebra 
$\Sigma_0$ of $\Sigma$ which is essentially order-closed in $\Sigma$
(i.e.\  $\{ [\chi_E ] \in \Pdb(\M) : E \in \Sigma_0 \}$  is closed under suprema in $\Pdb(\M)$).
 \end{itemize}
 \end{theorem}

 \begin{proof}  (iii) $\Rightarrow$ (i) \
 These are well known for
 von Neumann algebras.   The first assertion  follows from for example Proposition \ref{sa3}.
  The second
 assertion follows since the supremum of an increasing net of projections is its weak* limit as we said
 in Lemma \ref{WStone}.

  (i) $\Rightarrow$ (iv) \  Let $\Sigma_0 = \{ B \in \Sigma : \chi_B \in \N \}$.  Such $\chi_B$ correspond
  to the
  projections in $\N$.  One easily checks that
   $\Sigma_0$ is a $\sigma$-subalgebra of $\Sigma$.
  We have $\N = L^\infty(X,\Sigma_0,\mu_{|\Sigma_0})$,
since these two $C^*$-subalgebras of $\M$
are the closed span of their projections.
(It is an exercise that for a $\sigma$-subalgebra $\Sigma_0$ of $\Sigma$, $L^\infty(X,\Sigma_0,\mu_{| \Sigma_0}) \subset L^\infty(X,\Sigma,\mu)$
isometrically and as Banach or $C^*$-algebras.)
Clearly $\{ [\chi_E ] \in \Pdb(\M) : E \in \Sigma_0 \}$  is closed under suprema.

(iv) $\Rightarrow$ (ii) \
Suppose that  $( f_t )$ is a net in $\N$ with $f_t  \nearrow f$ in $\M_+$.
If $s \in \Rdb_+$ write $\llbracket  f > s \rrbracket$ for the
 projection in $\M$ corresponding to $f^{-1}((s, \infty])$.  This may be viewed as the obvious spectral projection for $f$ in $\M$.
We claim that $\llbracket  f > s \rrbracket \in \N$.
 It is an exercise that $\llbracket  f > s \rrbracket = \vee_t \,  \llbracket  f_t  > s \rrbracket$
 (for example it follows from the fact that for any projection $r$, if $f_t \, r \leq s 1$ for all $t$ then  $f r \leq s 1$).
However $\vee_t \,   \llbracket  f_t  > s \rrbracket$ is a projection in $\N$.   This proves the claim.
By measure theory ($f$ is norm approximable from below by certain simple functions), $f \in \N$.

(ii) $\Rightarrow$ (iii) \
Given (ii), every normal state of $\M$
restricts to a normal state of $\N$.  Hence $\N$ is a  $W^*$-algebra by Theorem \ref{Dix}.
 The inclusion
map $i : \N \to \M$ is normal since given an increasing net $(x_t )$ in $\N$, with $x_t \nearrow x$ in $\M$,
then $x$ must be the supremum in $\N$.    Thus $i$ is weak* continuous
(this may be seen
using
for example
Theorem \ref{abW} (ii)), and $\N$ a  weak* closed
(von Neumann) subalgebra.
\end{proof}

\begin{corollary} \label{gg}  {\rm (Bade)} \ A concrete (commutative) $C^*$-algebra generated by an 
order-closed (hence Dedekind  complete) Boolean algebra $P$ of projections
on a Hilbert space is a von Neumann algebra whose projection lattice is $P$. \end{corollary}

\begin{proof}   Let $\N = C^*(P)$.  Since $\N$ is commutative, its
strong closure, or weak* closure, are easily seen to be  a commutative von Neumann algebra.
Hence by
Corollary \ref{awvn} we may assume  that $\N$ is a subalgebra of $\M = L^\infty(X,\Sigma,\mu)$ for localizable $\mu$.  Let
 $\Sigma_0 = \{ B \in \Sigma : \chi_B \in P \}$.   As in the proof of (iv) in the theorem, this is easily checked to be a $\sigma$-subalgebra of $\Sigma$.
  Clearly $P = \Pdb(L^\infty(X,\Sigma_0,\mu_{|\Sigma_0}))$, and so $\N = L^\infty(X,\Sigma_0,\mu_{|\Sigma_0})$.
Now apply that (iv) implies (iii) in the 
theorem.  \end{proof}

\begin{corollary} \label{gg2}
  Let $\M$ be an abelian von Neumann algebra
  and let  $P\subseteq \Pdb (\M)$. Then $P$ generates $\M$ as a von Neumann algebra if and only if $P$ generates   $\Pdb (\M)$  as an order-closed
  Boolean subalgebra  (of $\Pdb(\M)$).
\end{corollary}

\begin{proof} ($\Leftarrow$)  \ This direction is obvious.
Indeed the projections in the smallest von Neumann subalgebra $\N$ containing $P$ are in an order-closed Boolean algebra of projections  containing $P$, so containing $\Pdb (\M)$.
So $\N = \M$.

($\Rightarrow$)  \
Write $\M = L^\infty(X,\Sigma,\mu)$.
The intersection $Q$ of the order-closed Boolean algebras of projections
in $\M$ containing $P$, is the smallest order-closed Boolean algebra of projections
containing $P$.    By Corollary \ref{gg},
if $\N = C^*(Q)$ then $Q = \Pdb(\N)$ and
$\N$ is a von Neumann  subalgebra of $\M$.
Hence $\N = \M$ since $P$, and
$\N$, generate $\M$ as a von Neumann algebra.
\end{proof}

\begin{remark}  The results in this subsection become considerably easier for finite measures--in this case there is no need to use Dixmier's theorem \ref{Dix}.
 Indeed in this case (ii) implies (i) in
Theorem \ref{subloc} similarly to how
(iii) implied (i).   Hence (ii) also implies (iv), so that $\N$ there is a $W^*$-algebra because 
$\mu_{\vert \Sigma_0}$
is finite, and one may now finish (iii) as before.

Also regarding Theorem \ref{subloc}, it is interesting that the spectral condition in the noncommutative version of this result, namely in (iv) in
Theorem 2.8.4 in \cite{Ped},
 is automatic in the $\sigma$-subalgebra setting—one knows exactly what the spectral projections are in this case! 
 We also remark that  the statement
 in  332T  in \cite{Fremlin} is closely related to  part of Theorem \ref{subloc}.

Bade's result  \ref{gg} may be found in \cite [p.\ 405 and Corollary 4.2]{bade}
(see also Corollary 17 and Lemma 6 of
\cite[Chapter 17.3]{DS}).
\end{remark}

\section{$L^\infty$-spaces as $W^*$-algebras - 3}
\label{cmsvn}

Theorem \ref{frec} characterize  those measure spaces $(X,\Sigma,\mu)$ for which the space $L^\infty(X,\mu)$ is a $W^*$-algebra with
predual $L^1(X,\mu)$.
With some more refined technology and a deeper understanding of
abelian von Neumann algebras at our disposal, we are now ready to consider the related problem of characterizing measure
spaces $(X,\Sigma,\mu)$ for which $L^\infty(X,\mu)$ is a $W^*$-algebra with no assumption on the identity of a predual.
Recalling that dualizable measure spaces are
semifinite, there is a sense in which we are here trying
to see what parts of Theorem \ref{frec} survive when we dispense with semifiniteness.

\begin{theorem} \label{ded}  Let $\mu$ be a measure on a measurable space $(X,\Sigma)$, and set $\M = L^\infty(X,\Sigma,\mu)$.
The  following are equivalent:
\begin{itemize}
\item [(i)]    $\M$ is a $W^*$-algebra.
\item [(ii)]   The measure space $(X,\Sigma,\mu)$ is
Dedekind (that is, suprema in $\Pdb(\M)$ always exist), and integration considered as a weight
$\varphi$ on $\M_+$
is a sum (or is a supremum) of normal positive functionals on $\M$.
  \item [(iii)] The measure algebra of $(X,\Sigma,\mu)$ is Dedekind complete,
  and the `antisemifinite part' of $\mu$
    (defined after Lemma {\rm \ref{mu}}) is
the `trivialization' of a localizable measure. (`Trivialization' means that  the measure is redefined to be $\infty$ on any set that
had strictly positive measure).
 \item [(iv)]  $\mu \cong \nu$ for a localizable measure $\nu$ on $(X,\Sigma)$.
\end{itemize}
The localizable measure in {\rm (iv)} may be chosen to agree with $\mu$ on the `semifinite part' of $\mu$.
\end{theorem}

\begin{proof} (i)  $\Rightarrow$ (ii) \  Note that $\psi_\mu$ is a completely additive weight on $\M_+$ by
the last assertion of Lemma \ref{tol}, and indeed is normal by Theorem \ref{rnd2}.
Suprema in $\Pdb(\M)$ always exist
by Lemma \ref{WStone}.

Let $(S , \mu_{|S})$ be the semifinite part of $\mu$ as in Lemma \ref{mu} and the lines after it.
Let $e \in \Pdb(\M)$ correspond to $S$, and set $\N = \M e$.  Then $\N \cong L^\infty(S , \mu_{|S})$.  Let $\varphi_0 = \varphi_{| \N_+} = \int_{S} \,  \cdot \,
d \mu_{|S}$.
Then $(S , \mu_{|S})$ is localizable by Lemma \ref{mu}  (1).   Write  $\N \cong \oplus_i \, \M p_i$ as in  Theorem \ref{loc} (iii),
and define  $\varphi_i (x) = \varphi_{p_i}(x)$ for $x \in \M$.  These have mutually orthogonal support
projections.
Then  $\varphi_0$ and $\sum_i \, \varphi_i$ are normal weights on $\M_+$ which agree
on $\M(1-e)$ and on each $\M p_i$.  Hence $\varphi_0 = \sum_i \, \varphi_i$ on $\M_+$.

For any measurable subset $F$ of $X \setminus S$ with $\mu(F) > 0$, by the last assertion
of Lemma \ref{WStone}
there is a  normal state $\varphi_F$ on $\M$ that is strictly positive on $\chi_F$.
By considering $\chi_{X \setminus S} \cdot \varphi_F$ we may assume that $\varphi_F$ is supported on $e^\perp$ if desired.  We claim that on $\M_+$, $\varphi$ is the sum of the $\varphi_i$, plus the sum of countably many identical
 copies of each $\varphi_F$, for all $F$ as above.
 Since normal weights are determined by their action on projections,
 and since any sum of positive normal functions on $\M_+$ is a sup of positive normal functions,
 it suffices to check the claim on a projection $p \in \M$.   However $\varphi(p) = \varphi(p e) + \varphi(p e^\perp)$, so
 it suffices to check the claim on a nonzero projection dominated either by $e$ or $e^\perp$.    However both are now  easy exercises using the facts  above.

(ii) $\Rightarrow$ (i) \
This is  a simple rephrasing of Theorem \ref{Dix}.
Suppose that integration  is a sum or supremum of normal positive  functionals $\varphi_i$. If $f \in \M_+$ satisfies $\varphi_i(f) = 0$ for every $i$ then $\int_X \, f \, d \mu = 0$, so that $f =  0$ in $\M$. By the Theorem \ref{Dix}, $\M$ is a $W^*$-algebra.

(i) $\Rightarrow$ (iii) \  If $\M$ is a $W^*$-algebra then we saw above that the `semifinite part' $(S , \mu_{|S})$ of $\mu$ is localizable.
The `antisemifinite part' $\M(1-e)$ is a commutative $W^*$-algebra, so  by Corollary \ref{awvn}
 there exists
 a $*$-isomorphism $\pi : \M(1-e) \to L^\infty(\Omega,\Sigma_0,\nu')$, for a localizable measure space
 $(\Omega,\Sigma_0,\nu')$.
 The projections in $\M(1-e)$, which
correspond to the $\Sigma$-measurable subsets of $X \setminus S$, are taken by $\pi$ to
the (equivalence class of) characteristic functions of sets in $\Sigma_0$.
Define a measure $\nu$ on $X \setminus S$ by $\nu(F) = \int_\Omega \, \pi(\chi_F) \, d \nu'$ for $F \in \Sigma, F \subset
X \setminus S$.
Since $\nu'$ is semifinite, so is $\nu$.   Hence
 $\nu$ is  a localizable measure, and has trivialization  $\mu_{|X \setminus S}$.

(iii) $\Rightarrow$ (iv) \ Suppose that the semifinite part  $(S , \mu_{|S})$
is localizable, and the `antisemifinite part'
is the `trivialization' of a localizable measure  $\sigma$ on $X \setminus S$.
Define $\nu(E) = \mu(E \cap S) + \sigma(E \setminus S)$ for $E \in \Sigma$.
It is clear that $\nu$ is semifinite, and that $\mu$ and $\nu$ have the same null sets, so that $\mu \cong \nu$
and $\M =  L^\infty(X,\Sigma,\nu)$.   So $\nu$ is localizable.

(iv) $\Rightarrow$ (i) \ We have $\M =  L^\infty(X,\Sigma,\nu)$, and this is a $W^*$-algebra
(for example by Theorem \ref{frec} (ii)).
\end{proof}

We recall some interesting facts, a couple  of which will be used in the background in the remainder of this
section. A measure space $(X,\Sigma,\mu)$
is Dedekind  if and only if  $L^\infty_\Rdb(X,\Sigma,\mu)$ is boundedly complete
 (i.e.\ every bounded subset  has a supremum).
A proof of one direction of this may be found in the last paragraph of the proof of 
Lemma \ref{WStone}. For the other direction see 
for example 363M in  \cite{Fremlin}. 
As we said in an earlier section,
$L^\infty_\Rdb(X,\Sigma,\mu) = C(\Omega,\Rdb)$ for a compact space $\Omega$,
 the spectrum, or maximal ideal space of $L^\infty(X,\mu)$, which in this case is called the {\em Stone space}.
It is known for a compact  space $\Omega$  that
$C(\Omega,\Rdb)$ is boundedly complete 
 if and only if  $\Omega$ is Stonean (that is, compact and extremely disconnected).
  One direction of this is in
 Lemma \ref{sa1}, for the other direction see for example \cite[Theorem 2.3.3]{DalesDLS}.
 For interest's sake we remark that this happens  if and only if  $C(\Omega,\Rdb)$ is an injective Banach space (see for example 363R in \cite{Fremlin}),
and it is not hard to show that the latter is equivalent to $C(\Omega)$ being injective as a complex Banach space.
In this case $C(\Omega)$ is called a (commutative) $AW^*$-{\em algebra}.

Conversely if $\Omega$ is any Stonean space
 then it follows from Theorem \ref{LSthm} that there is a Dedekind measure space $(X,\Sigma,\mu)$
with $L^\infty_\Rdb(X,\Sigma,\mu) \cong C(\Omega,\Rdb)$.
Indeed as we just said, $C(\Omega,\Rdb)$ is boundedly complete,
so by the last paragraph of the proof of
Lemma \ref{WStone} the projections in $C(\Omega)$ form a Dedekind complete Boolean algebra.
Define $\bar{\mu}(p) = \infty$ for any nonzero projection $p$.
Theorem \ref{LSthm} then gives a Dedekind measure space $(X,\Sigma,\mu)$.   Using the last assertion of
Lemma \ref{sa1}, the
proof of Proposition \ref{Wextn}  shows that $L^\infty(X,\Sigma,\mu) \cong C(\Omega)$
 $*$-isomorphically.

Putting some of the above facts together, a
measure space is Dedekind if and only if the spectrum (or Stone space) $\Omega$ above is
Stonean, or equivalently
$C(\Omega)$ or its isomorphic copy $L^\infty(X,\mu)$ has the properties above.

\begin{example} It is evident from the equivalence of (i) and (ii) in Theorem \ref{loc} (or  \ref{frec}) that $\M = L^\infty(X,\mu)$ is a $W^*$-algebra  if 
$(X,\mu)$ is Dedekind and integration $\psi_\mu$ with respect to $\mu$ is a semifinite normal weight on $\M$.    However if $\mu$  is not semifinite, so that $\psi_\mu$  is
allowed to be infinite on `large pieces',  then $\M$ need not be a $W^*$-algebra.
 Indeed we give a counterexample to the question of whether  $L^\infty(X,\mu)$ is a $W^*$-algebra   if and only if we have that both the measure space $(X,\mu)$ is Dedekind  and integration $\psi_\mu$ is a normal weight.
 To do this one may take the canonical example of a commutative $AW^*$-algebra which is not a $W^*$-algebra,
 and assign it the trivial `always infinite' weight.

  More particularly,
consider the {\em Dixmier algebra} $D$, the quotient of the bounded Borel functions on $[0,1]$ by the closed ideal of bounded Borel functions vanishing on the complement of a meager Borel set \cite{DixS}.  Being a commutative $C^*$-algebra,
 $D = C(Y)$ for a compact space $Y$.   It can be shown that $D_{\rm sa}$ is boundedly complete
 (see for example \cite{DixS,Kad2}).  Hence by the facts summarized
 below  Theorem \ref{ded},
 $Y$ is Stonean, and there is a   Dedekind measure space $(X,\Sigma,\mu)$
with $L^\infty(X,\Sigma,\mu) \cong C(Y)$ $*$-isomorphically.
 However  $L^\infty(Y,\Sigma,\mu)$ is not a von Neumann algebra, since it is well known that
 $C(Y)$ is not a von Neumann algebra (see for example \cite{Kad2}).
 To see that the integral $\varphi$ with respect to $\mu$ is normal, note that by the lines just
 below Theorem \ref{ded},  $\mu$ only takes on values $0$ and $\infty$.
  For any nonnegative measurable simple function $s = \sum_{k=1}^n \, t_k \chi_{E_k}$ with $t_k > 0$, we see that $\int \, s \, d \mu$ is either $0$ or $\infty$. 
   So the integral of any nonnegative  measurable function is either $0$ or $\infty$.
  Suppose that $(f_t)$ is an increasing bounded nonnegative net in $L^\infty(Y,\Sigma,\mu)$.
Thus $(\varphi(f_t))$ is an increasing net in $\{ 0,\infty \}$.   If $\varphi(f_t) = 0$ for all $t$ then $f_t = 0$ $\mu$-a.e., and so $\sup_t \, f_t$ is $0$.
If $\varphi(f_t) = \infty$ for some $t$ then $\sup_t \, \varphi(f_t) = \varphi(\sup_t \, f_t)$. It is now clear that $\varphi$ is normal.
Clearly $\varphi$ is also completely additive in the sense defined at the beginning of Section
\ref{wol}.
\end{example}

This example also shows that for a commutative $AW^*$-algebra $\M$,  normality (or complete additivity) of a weight is not necessarily equivalent to being a supremum of normal positive functionals. (If it were then
Theorem \ref{ded} (ii) would imply that $D$ in the example above is a $W^*$-algebra.   Also,
$D$ is known to have no nontrivial normal positive functionals.)

\begin{remark} 
One may ask if 
the normal functionals  in  (ii) of Theorem \ref{ded} may be chosen to be of the form $\int_E f \, d \mu$ for measurable sets $E$ in $X$ of finite measure.    This is false in general, indeed the latter forces
semifiniteness 
(this is related to for example 322Ea in \cite{Fremlin}).
\end{remark}

\section{Classification of abelian von Neumann algebras} \label{MahS}

This last section differs from the earlier ones in its approach. Its aim  is to demonstrate the
possibility of performing a complete 
structural classification of abelian von Neumann algebras using purely operator-algebraic means. 
We will not  use results from earlier sections on the $L^\infty$ representation of abelian von Neumann algebras.  Indeed an `improved' such $L^\infty$ representation will be obtained
 as a consequence of the general classification theorem.

The presentation is based on Fremlin's exposition of Maharam's theorem (see \cite[Chapter 33]{Fremlin}). Nevertheless, only one key result is used explicitly without being proved (see the remark before
Theorem \ref{Mah-iso} below). For the rest, we either present operator-algebraic versions of the proofs, or, in a few cases, give independent proofs of the needed facts.

In this section  $\M$ stands for a von Neumann algebra, and $\A$ for an abelian von Neumann algebra.
   
It is `folklore' that the Maharam theorem, which classifies localizable measure algebras, can be seen as a theorem on the classification of abelian von Neumann algebras. But, as seen from  Theorem \ref{vna=ma}, the proper operator-algebraic counterpart of the notion of \emph{measure algebra} is that of \emph{von Neumann measure algebra} (see Definitions \ref{defvnma} and \ref{defvnmamor}). Hence, our aim will be to classify up to isomorphism all von Neumann measure algebras $(\M,\vf)$ for \emph{abelian} $\M.$

A subset $G$ of a Boolean algebra (resp.\ complete Boolean algebra) $\Balg$  \emph{generates} (resp.\ \emph{completely generates}) $\Balg$ if $\Balg$ is the smallest  Boolean subalgebra (resp.\ complete Boolean subalgebra) of $\Balg$ containing $G.$ We say that a complete Boolean algebra $\Balg$ is of \emph{Maharam type $\kappa$} if $\kappa$ is the smallest cardinal of any subset of $G$ completely generating $\Balg$. We denote the cardinal by $\kappa(\Balg).$ A complete Boolean algebra $\Balg$ is \emph{Maharam-type-homogeneous} if $\kappa(\Balg_a)=\kappa(\Balg)$ for any $a\in \Balg$ (with $\Balg_a$ the principal ideal generated by $a$).

We shall need the following simple lemma:

\begin{lemma}\label{Mah-gen}
  If $\Balg$ is a Boolean algebra completely generated by a finite set of generators, then it is finite. If, on the other hand,  $\Balg$ is a complete Boolean algebra of infinite Maharam type, and $G$  completely generates $\Balg$, then the Boolean subalgebra $\Balg_G$ of $\Balg$ generated by $G$ completely generates $\Balg$ and $\# \Balg_G=\# G$.
\end{lemma}

\begin{proof}
 First note that a finite set $G\subseteq \Balg$ generates a finite Boolean algebra (which is then also Dedekind complete). It is an exercise to check that if $C_0$ is a Boolean subalgebra of $\Balg$ and $c\in \Balg,$ then
 $C:=\{(a\cap c)\cup (b\setminus c)\colon a,b\in C_0\}$ is a Boolean subalgebra of $\Balg$ generated by $C_0$ and $c$ 
(see 312N in \cite{Fremlin}). Starting from $D_0:=\{0,1\}$ and adding elements one by one we can show that the Boolean subalgebra generated by a finite $G$ is also finite.

 If $G$ is arbitrary let $\Balg_G$  be the Boolean subalgebra of $\Balg$ generated by $G.$
 For all finite subsets $F\subseteq G,$ let $\Balg_F$  be the finite Boolean subalgebra of $\Balg$ generated by $F.$  For finite $F_1,F_2 \subseteq G$ we have $\Balg_{F_1}\cup \Balg_{F_2}\subseteq \Balg_{F_1\cup F_2},$ which implies $\Balg_G=\bigcup \{ \Balg_F\colon F\subseteq G \text{ finite} \}$. If $G$ is of infinite cardinality $\kappa,$ then the number of its finite subsets is also $\kappa,$ and the cardinality of $\Balg_G,$ being a union of $\kappa$ finite sets, is at most $\kappa.$ Since $G\subseteq \Balg_G,$ we have $\# \Balg_G=\kappa.$ If $G$ completely generates $\Balg,$ then  $G\subseteq \Balg_G\subseteq \Balg$ implies that $\Balg_G$ completely generates $\Balg,$ which ends the proof (cf.\  the proof of 331G in \cite{Fremlin}).
\end{proof}

The following is the most difficult
step in the proof of the Maharam theorem, and the only one which we shall
not prove (see 331I in \cite{Fremlin}). The proof relies on a rather tricky application of transfinite induction.

\begin{theorem}\label{Mah-iso}
  Let $(\Malg_1,\mu_1)$ and $(\Malg_2,\mu_2)$
  be Maharam-type-homogeneous measure algebras of the same Maharam type and with $\mu_1(1)=\mu_2(1)<\infty.$ Then they are isomorphic as measure algebras.
\end{theorem}

We call von Neumann algebras without minimal projections
\emph{non-atomic}, and those isomorphic to a direct sum of some (cardinal) number of copies of $\Cdb$ \emph{purely atomic}.
We will write the latter direct sum as $\Cdb^\kappa$ where $\kappa$ is the cardinal (usually
this is written as $\ell^\infty(\kappa)$).
We start the effort of classifying
abelian von Neumann measure algebras $(\A,\tau)$
 by reducing it to the case  of non-atomic abelian  von Neumann
 probability spaces. This step is not strictly necessary, and we could jump directly to $\kappa$-homogeneous von Neumann
  probability spaces. Nevertheless, it makes the proof slightly easier to follow and seems a natural step for an operator algebraist.

\begin{lemma}\label{atoms}
  An abelian von Neumann algebra $\A$  is a direct sum  of a non-atomic algebra  and a purely atomic algebra.
\end{lemma}

\begin{proof}
  If $p\in\p{\A}$ is minimal, then $\A p\simeq \Cdb.$
  Indeed the only projections in $\A p$ are $0$ and $p$, and a von Neumann algebra is generated by its projections (see Lemma \ref{sa1}), so that $\A p=\Cdb p.$ Let $\{p_i\}_{i\in I}$ be the family of all minimal projections in $\A$.  These  are mutually orthogonal (since $p_i p_j$ is a projection). Let $p=\sum_{i\in I}p_i.$ Then $\A p\simeq \Cdb^{\#(I)}$
  as for example in the proof of
  Theorem \ref{loc} (iii),
  and there are no minimal projections in $\A p^\perp.$
\end{proof}

\begin{proposition}\label{prob}
  An abelian non-atomic von Neumann measure algebra $(\A, \tau_\A)$ with
   $\tau_\A(\I_\A)=\infty$ is
  isomorphic to
  a direct sum of (abelian non-atomic)
  von Neumann  probability  algebras $(\A_i,\tau_i)$.
\end{proposition}

\begin{proof} Since $\A$ is abelian, any orthogonal family
 $(p_i)$ of projections in $\A$ summing to $\I_\A$ yields a decomposition of $\A$ into 
a direct sum $\oplus_i \, \A_i$ of von Neumann subalgebras
 (see for example the proof of
  Theorem \ref{loc} (iii)).
  We shall  find such a 
   family $(p_i)$ with $\tau_\A(p_i) = 1$ for every $i$. 
   Since 
   $\tau = \tau_\A$ is semifinite, there is a non-zero projection $p$ in $\A$ such that $\tau(p)<\infty.$ Since $p$ is not minimal, let us write it as $q+r$ with $q,r\in\p{\A}\setminus\{0\}$, and denote by $p_1$ one of $q$ and $r$ chosen so that $\tau(p_1)\leq 2^{-1}\tau(p).$ Further divisions lead to $p_k\in \p{\A}$ with $\tau(p_k)\leq 2^{-k}\tau(p).$ Hence, there are non-zero projections in $\A$ with arbitrarily small trace. Consequently, the set $\{e\in\p{\A}\colon \tau(e)\leq 1\}$ is non-empty, and its maximal element $f\in\p{\A}$ satisfies $\tau(f)=1.$  Indeed
   by semifiniteness
   $1-f$ dominates a $\tau$-finite projection, hence (by the above argument)
    a non-zero projection $q$ in $\A$ with trace $<
   1- \tau(f)$, if the latter is $> 0$.   Thus $\tau(f + q) < 1$, contradicting  the maximality of $f$.

   A small modification of the last argument shows that for any real  $t  > 0$ and
   projection $p$ in $\A$ with $\tau(p) \geq t$, there exists a
   projection $f \leq p$ with $\tau(f) = t$.

   Let $F$ be a maximal orthogonal family of non-zero projections of trace $1$ in $\M$, necessarily infinite. Let $g=\I_\A-\sum_{f\in F}{f}.$ 
   Obviously $\tau(g)<1$ by the maximality.
   Choose a countable subfamily $H=\{h_k\}$ from $F$. 
   By the last paragraph, we may for all $k$ find a projection $g_k\leq h_k$ such that $\tau(g_k)=\tau(g).$ Now 
   $$G=\{g+(h_1-g_1), g_1+(h_2-g_2), g_2+(h_3-g_3),\dots\}$$ is an orthogonal family of projections of trace $1$, and its members are orthogonal to all projections from $F \setminus H.$
 Moreover, their sum equals $\I_\A-\sum_{f\in F\setminus H} \, f.$ Hence, the orthogonal family $J:=G\cup (F\setminus H)$ consists of projections of trace $1$ summing to $\I_\A.$ Put $\tau_p:=\tau|\A p$ for all $p\in J.$ 
 The direct sum of the 
 von Neumann probability spaces $(\A_p,\tau_p)_{p\in J}$ is exactly $(\A,\tau).$
\end{proof}

There are several notions which we will need 
in order to formulate the theorem on isomorphism of noncommutative measure algebras.

\begin{definition}
  The \emph{decomposability number} $\dec{\M}$ of a von Neumann algebra $\M$ is the maximal cardinality of an orthogonal family of non-zero projections in $\M.$
  The \emph{decomposability number} $\decp{\M}{p}$ of a projection $p$ 
  in $\M$ is the decomposability number $\dec{\M_p}$ of the reduced von Neumann algebra $\M_p.$
  We say that $\M$ (resp.\ $p\in\p{\M}$) is \emph{countably decomposable} or \emph{$\sigma$-finite} if
  $\dec{\M}=\aleph_0$ (resp.\ $\decp{\M}{p}=\aleph_0$).
  If $(\M,\tau_{\M})$ is a von Neumann measure algebra, 
  with $\tau_\M$ a faithful normal semifinite trace on $\M$,  then 
  for any $p\in\p{\M}$  define the \emph{magnitude} of $p$ by
  $$\magn{\M}{p}=
  \begin{cases}
    \tau_{\M}(p)& \mbox{if }\tau_{\M}(p)<\infty,\\
    \decp{M}{p}& \text{otherwise}.
  \end{cases}$$
 \end{definition}

 Note that although the terms \emph{decomposability number} and \emph{magnitude}  are defined for 
von Neumann algebras, in 
the case that the algebra $\A$ considered is abelian, they have the same meaning (as defined in \cite{Fremlin}) in the Boolean algebra $\p{\A}.$

It is well known that $\M$ is $\sigma$-finite if and only if it admits a faithful normal state.
The one direction of this is obvious (by applying the state to the sum of the projections).  The other
direction  follows from a routine Zorn's lemma argument of the type we have seen before, on sets of mutually orthogonal projections that
are the supports of  normal states (or see for example \cite[Proposition 3.19]{tak1}).

\begin{lemma}\label{sfinvN}
  Let $\kappa$ be an infinite cardinal. If a von Neumann algebra $\M$ can be written as a direct sum of $\kappa$ $\sigma$-finite
  algebras, then
  $\dec{\M} = \kappa$.
\end{lemma}

\begin{proof}
  By assumption, there is an orthogonal family $\{z_j\}_{j<\kappa}$ of $\sigma$-finite central projections in $\M$ satisfying $\sum_{j<\kappa}z_j=\I.$
  By definition, $\kappa \leq \dec{\M}.$
   Let $P:=\{p_i\}_{i\in I}$ be any orthogonal family of projections in $\M.$ Then $Q:=\{p_iz_j\colon p_iz_j\neq0\}$ forms an orthogonal family.  For each $j<\kappa$ we have $1  \leq\#\{i\in I\colon p_iz_j\neq 0\}\leq \aleph_0$ by the $\sigma$-finiteness of $z_j$'s.  Consequently
  $\# P\leq \# Q \leq \kappa\cdot \aleph_0=\kappa,$
  so that $\dec{\M}\leq \kappa.$
\end{proof}
 
 \begin{definition}
Let $\A$ be an  abelian von Neumann algebra. We denote by $\kk{\A}$ the smallest cardinality of a set $P$ of projections in $\A$ such that $P$ \emph{generates} $\A$ as a von Neumann algebra.
 Since  any von Neumann subalgebra contains $\I$,  we may as well say that
 $P$ and $\I$ generate $\A$  as a von Neumann algebra.
 For this reason if $\A$ is 1 dimensional then we take $\kk{\A}=0.$
\end{definition}

When we  speak about generators, we 
will always mean projections. There is a good reason for this --- by the theorem below, a set of projections $P$ generates an abelian von Neumann algebra $\A$  if and only if it \emph{completely generates} the Boolean algebra $\p{\A}$.
That is, the smallest 
order-closed (hence Dedekind complete)
Boolean subalgebra of $\p{\A}$ containing $P$ equals $\p{\A}$. This coincides with  Fremlin's notion of \emph{complete} or $\tau$-\emph{generation} of Boolean algebras
 (see also \cite{Fremlin}, 331E). Additionally, we will
 avoid problems connected with generation of a von Neumann algebra
 by a finite number of operators (see \cite{sher}).

The general (not necessarily $\sigma$-finite) case of much of the following theorem has already been proved as 
Corollary \ref{gg2}. It is also contained in Bade's result (\cite[Theorem 3.4]{bade}, see also \cite[Lemma XVII.3.6]{DS}).
The following theorem can be also deduced from \ref{gg2}
   and a result one can find, for example, in \cite[Proposition VI.2a]{rick} or 316F (b) in \cite{Fremlin}; or from 331G (d) and (e) in \cite{Fremlin}.
(Item (b) below follows from the fact that the strong closure of a Boolean algebra of Hilbert space projections is a Boolean algebra, and is Dedekind complete.   This is an exercise.)
We decided to give it an independent proof
 that we have not seen in the literature, avoiding both Bade's proof and the proof of \ref{gg} (which uses an $L^\infty$ representation of $\M$), and helping keep this section independent of the earlier material.

\begin{theorem}\label{equiv}
  Let $\A$ be an  abelian non-atomic von Neumann algebra acting (non-degenerately) in a Hilbert space $H$ and let  $P\subseteq \p{\A}$. Then:
   \begin{enumerate}
     \item[(a)] If $\A$ is $\sigma$-finite and $P\subseteq \p{\A}$ generates $\A$ as a von Neumann algebra, then for any $p\in \p{\A}$ there is countable set $Q\subseteq P$ such that $p$ is in the complete Boolean subalgebra generated by $Q.$
       \item[(b)]  The Boolean subalgebra of  $\p{\A}$ generated by $P$ is
     strongly dense in $\p{\A}$.
     \item[(c)]  We have $\kk{\A}=\kappa(\p{\A}).$
   \end{enumerate}
\end{theorem}

 \begin{proof} Assume that  $P$ is a subset of $\p{\A}$ generating $\A$ as a von Neumann algebra. The first half of Lemma \ref{Mah-gen} implies that $P$  is infinite. Otherwise $\p{\A}$ would be finite, which is impossible for non-atomic $\A$. By the second half of the same lemma, we may assume that $P$ is a Boolean algebra.

Let $A$ be the set of all 
linear combinations of elements of $P.$ This is a $*$-algebra generating $\A,$ hence by von Neumann's double commutant theorem  $A$ is strongly dense in $\A$ (see, for example, \cite{Black}, I.9.1.1). Take an arbitrary $p\in\p{\A}.$  We will show that $p$ is in the smallest complete Boolean algebra
$\Balg$ containing $P$.  This implies that $\Balg=\p{\A}$ and   $\kk{\A}\geq \kappa(\p{\A}).$

By the Kaplansky density Theorem (see, for example, Theorem 5.3.5 and Corollary 5.3.6 in \cite{KR}), the positive part of the unit ball of $A$ is strongly dense in the positive part of the unit ball of $\A$. There is a directed set $D$ and a net $(x_d)_{d\in D}$ in $\A_+$ with $\|x_d\|\leq 1$ such that $x_d\to p$ strongly. We can assume that for all $d,$  $x_d=\sum_{i=1}^{m_d} \gamma_i^{(d)}p_i^{(d)},$ where $m_d\in\Ndb,\,0\leq  \gamma_i^{(d)}\leq1$ for  $i=1,\dots,m_d,$ with orthogonal families $\{p_i^{(d)}\}_{i=1}^{m_d}\subset P$.

Choose a vector $\xi \in H,  \|\xi\|=1$ that is separating for $\A$ (see, for example, 5.5.17 in \cite{KR}). Let, for each $n\in \Ndb,$ $d_n\in D$ be such that $\|x_d\xi-p\xi\|\leq \frac{1}{n2^n}$ for all $d\geq d_n.$ We can also assume that the sequence $d_n$ is increasing.
Put $q_n:=\chi_{[1/n,1]}(x_{d_n}).$ Note that $q_n=\sum\{p_i^{(d_n)}\colon \gamma_i^{(d_n)}\geq1/n\}\in P$ for all $n.$   Let $r_j :=\sup_{n\geq j} q_n$ for all $j\in\Ndb.$ Then the family $\{r_j\}$ is decreasing and $p\leq r_j$ for all $j$. In fact, otherwise for some $j_0\in \Ndb$ we would have $pr_{j_0}^\perp\xi\neq0,$ and for all $j\geq j_0,$ 
 \begin{equation*}
  \|x_{d_j}(p r_{j_0}^\perp)\xi\|=\|x_{d_j}r_{j}^\perp(p r_{j_0}^\perp)\xi\|
  \leq \|x_{d_j} q_{j}^{\perp}(p r_{j_0}^\perp)\xi\|
  \leq \frac{1}{j}\|q_j^\perp
  p r_{j_0}^\perp\xi\|
  \leq\frac{1}{j}.
  \end{equation*}
  On the other hand, 
  \begin{equation*}
  x_{d_j}(p r_{j_0}^\perp)\xi=(p r_{j_0}^\perp)x_{d_j}\xi\to (p r_{j_0}^\perp)p\xi\neq0
  \end{equation*}
   (as $\xi$ is separating for $\A$), which yields a contradiction.

Let $r:=\inf_{j\in \Ndb} r_j.$ Then $p\leq r.$ For all $j$ we have 
 \begin{equation*}
  \begin{aligned}
  \|(r-p)\xi\|^2&\leq \|r_j(r-p)\xi\|^2\leq \sum_{n\geq j}\|q_n(r-p)\xi\|^2 \\
  &\leq \sum_{n\geq j}n^2\|x_{d_n}q_n(r-p)\xi\|^2\leq \sum_{n\geq j} n^2\|x_{d_n}(r-p)\xi\|^2\\
  &\leq \sum_{n\geq j} n^2\|x_{d_n}(r-p)\xi-p(r-p)\xi\|^2=\sum_{n\geq j} n^2\|(r-p)(x_{d_n}-p)\xi\|^2\\
  &\leq  \sum_{n\geq j}\frac{n^2}{n^24^n}=\frac{1}{3\cdot 4^{j-1}}.
  \end{aligned}
 \end{equation*}
  Since $j$ was arbitrary, we get $(r-p)\xi=0,$ and as $\xi$ was separating for $\A$ we have $r=p.$ Thus $p$ is in the smallest complete Boolean algebra $\Balg_Q$ containing the countable set $Q:=\{q_n\}_{n\in\Ndb}\subseteq P.$   This ends the proof of (a).
  Moreover, $p$ is a strong limit of the decreasing sequence $r_j,$ and $r_j$ is a strong limit of
   the increasing sequence  $\sup_{j\leq n\leq  j+m} q_n$ as $m\to\infty.$
   This shows the density of the Boolean algebra generated by $P$ in $\p{\A},$ and also gives (b). 

Item (c) follows easily from (a).
\end{proof}

\begin{definition}\label{hom} Let $\kappa$ be a cardinal.
  We say that an abelian von Neumann algebra $\A$ is \emph{$\kappa$-homogeneous} or \emph{homogeneous of type $\kappa$} if for any non-zero projection $p$ in $\A$, we have $\kappa=\kk{\A}=\kk{\A_ p}.$  We say that a non-zero projection $p$ from $\A$ is $\kappa$-homogeneous if $\A p$ is $\kappa$-homogeneous. An abelian von Neumann algebra or a projection in such an algebra is called homogeneous if is $\kappa$-homogeneous for some cardinal $\kappa.$
\end{definition}

Note that since $\A$ is abelian, we can simply use $\A p$ instead of $\A_p,$ treating it as a von Neumann algebra with unit $p.$

The following facts will be useful:
\begin{lemma}\label{aleph0} 
  If  a von Neumann algebra $\A$ has no minimal projections, then $\kk{\A}\geq \aleph_0.$
  Indeed if $\A$ is homogeneous 
  then   $\kk{\A}$ is either $0$
  (and $\A=\Cdb\I$) or an infinite cardinal.
\end{lemma}

\begin{proof} 
  If $\kk{\A}< \aleph_0$ 
  then by Lemma \ref{Mah-gen} and Theorem \ref{equiv}(c), $\p{\A}$ is finite. Hence, $\A$ has a minimal projection.   If $\A$ is homogeneous then the minimal projection must be $\I$ and $\A=\Cdb\I.$
  \end{proof}
 
\begin{lemma}\label{hom-maj}
For any non-zero projection $p\in\A$ there is a non-zero projection $q\leq p$ such that $\A q$ is homogeneous.
\end{lemma}

\begin{proof}
  Let $K=\{\kk{\A r}\colon r\in\p{\A}, 0\neq r\leq p\}.$ Let $\kappa$ be the least member of $K.$ Let $q\in\p{\A}$ be such that $\kk{\A q}=\kappa.$ If $r\in\p{\A}, 0\neq r\leq q,$ then $\kk{\A r}\leq \kk{\A q}.$ In fact, if $G\subseteq \p{\A q}$ generates $\p{\A q},$ then the set
  $\{gr\colon g\in G\}$ generates $\A r.$ Hence $\kk{\A r}= \kk{\A q},$   and $\A q$ is homogeneous.
\end{proof}

\begin{lemma} \label{clas1}
  A ($\sigma$-finite, non-atomic) abelian von Neumann algebra  is a direct sum of ($\sigma$-finite, non-atomic) homogeneous von Neumann algebras.
\end{lemma}

\begin{proof}
  Take a maximal orthogonal family $\{q_i\}$ of projections from $\A$ such that $\A q_i$ is homogeneous. By the previous lemma, $\sum_i q_i=\I.$
\end{proof}

\begin{definition}
  The  \emph{homogeneous component of type $\kappa$} of an abelian von Neumann algebra $\A$ is the  ideal $\A \, e_\kappa$, where $e_\kappa$ is
the supremum   of
all non-zero projections $e\in\A$ such that   $\A e$ is homogeneous of type $\kappa.$
\end{definition}

Often we will also call $e_\kappa$  the homogeneous component of type $\kappa$.  We disregard the trivial components of $\A$, that is those $\kappa$ for which there are no non-zero $\kappa$-homogeneous projections in $\A.$

\begin{lemma}\label{hom-orth}
For a cardinal $\kappa,$ let $e_\kappa$ denote the homogeneous component of $\A$ of type $\kappa.$ If $\kappa\neq \kappa',$ then $e_\kappa\perp e_{\kappa'}.$ Moreover, $\sum_{\kappa}{e_\kappa}=\I$, when the sum is over cardinals.
\end{lemma}

\begin{proof}
 In fact,
 $$e_\kappa e_{\kappa'}=\sup\{ee'\colon e \ (\mbox{resp. }e') \mbox{ is homogeneous of type } \kappa \ (\mbox{resp. }\kappa')\}
= 0. $$
Indeed  if $e$ is homogeneous of type $\kappa$ and $e'$ is homogeneous of type $\kappa',$ then $e\perp e',$ otherwise $ee'$ would be both $\kappa$- and $\kappa'$-homogeneous, which is impossible.  Hence $\sum_{\kappa}{e_\kappa}=\I$,
taking into account also Lemma \ref{clas1}.
\end{proof}

Note that $e_\kappa$ can be non-zero only if $\kappa$ is infinite or zero,
 by Lemma \ref{aleph0}. Moreover, $e_\kappa$ is not necessarily $\kappa$-homogeneous (for example, $e_0$ is not $0$-homogeneous if $\A$ has at least two atoms). Nevertheless, we have:

\begin{lemma}\label{homcomp}
  If the homogeneous component of type $\kappa$ of an abelian $\sigma$-finite von Neumann algebra $\A$ is $\A$ itself, then $\A$ is $\kappa$-homogeneous.
\end{lemma}

\begin{proof}
  Denote by $P$ the set of $\kappa$-homogeneous projections in $\A$. By definition, $\sup P=\I.$ Note that if  $p\in P$ and $0\neq q\leq p,$ then $q\in P.$ By Zorn's lemma, there exists a maximal orthogonal family of $\kappa$-homogeneous projections $Q\subset P.$ If $e_0:=\sup Q\neq \I,$ then $q_0:=pe_0^\perp\neq 0$ for some $p\in P.$ Thus $q_0\in P$ and $q_0$ is orthogonal to all projections from $Q.$ This contradicts the maximality of $Q$. Hence $e_0=\I.$ Let $G_q$ denote, for each $q\in Q,$ a set of cardinality $\kappa$ generating $\A q.$ Since $Q$ is at most countable, and $\bigcup_{q\in Q}G_q$ generates $\A$ as a von Neumann algebra, we have $\kk{\A}\leq \kappa.$ If $0\neq r\in \p{\A},$ then $rp\neq 0$ for some $p\in P,$
  hence $\kappa= \kk{\A (rp)}\leq \kk{\A r}\leq \kk{A}\leq \kappa,$ and $\A$ is indeed $\kappa$-homogeneous.
\end{proof}

The most important part of Maharam's theorem is contained in the following result (formulated here in the language of von Neumann algebras):

\begin{theorem} \label{homog}
  Let $(\A,\tau_\A)$ and $(\B,\tau_B)$ be such that $\tau_\A(\I)=\tau_\B(\I)$ and, for some infinite cardinal $\kappa,$ both $\A$ and $\B$ are $\kappa$-homogeneous. Then $(\A,\tau_\A)$ and $(\B,\tau_B)$ are isomorphic as von Neumann measure algebras.
\end{theorem}

\begin{proof}
  This follows directly from Proposition \ref{Wextn} (see also
  Theorem \ref{vna=ma})
  and Theorem \ref{Mah-iso}.
\end{proof}

The  measure algebra version of the following result is due to Fremlin (see 332J
in \cite{Fremlin}).

\begin{theorem} \label{bigf}
Let $e_\kappa$ and $f_\kappa$ denote the homogeneous components of type $\kappa$ of algebras $\A$ and $\B$, respectively.
Von Neumann measure algebras $(\A,\tau_\A)$ and $(\B,\tau_B)$ are isomorphic if and only if, for every infinite cardinal $\kappa$, $\magn{\A}{e_\kappa}=\magn{\B}{f_\kappa}$ and, for each $0<\gamma < \infty,$
$$\#\{\text{minimal }p\in\p{\A}\colon \tau_\A(p)=\gamma\}=\#\{\text{minimal }q\in\p{\B}\colon \tau_\B(q)=\gamma\}.$$
\end{theorem}

\begin{proof}
($\Rightarrow$) \
  Let $\vniso$ denote the isomorphism. It is clear that $\Phi(e_\kappa)=f_\kappa,$ which implies $\decp{\A}{e_\kappa}=\decp{\B}{f_\kappa}.$
  Since  $\tau_A(p) =\tau_B(\Phi(p))$ for all $p\in\p{\A}$, we have
   $\tau_\A(e_\kappa)=\tau_B(f_\kappa)$ for all $\kappa.$ Hence $\magn{\A}{e_\kappa}=\magn{\B}{f_\kappa}.$ The isomorphism $\vniso$ takes minimal projections in $\p{\A}$ (atoms of $\p{\A}$) into minimal projections in $\p{\B}$ (atoms of $\p{\B}$), and if $\tau_A(p)=\gamma,$ then $\tau_B(\Phi(p))=\gamma,$ which shows the second statement.

      ($\Leftarrow$) \ Assume now that for all infinite cardinals $\kappa,$ $\magn{\A}{e_\kappa}=\magn{\B}{f_\kappa}$ and, for each $0<\gamma < \infty,$ $\#\{\text{minimal }p\in\p{\A}\colon \tau_\A(p)=\gamma\}=\#\{\text{minimal }q\in\p{\B}\colon \tau_\B(q)=\gamma\}$.
      It is clear that the assumptions guarantee that the noncommutative measure algebras $\A$ and $\B$ have the same number $\lambda$ of minimal projections.
      Hence the purely atomic parts of $\A$ and $\B$, together with the restrictions of the trace,
      are 
isomorphic as von Neumann measure algebras.

If for some infinite $\kappa$, $\magn{\A}{e_\kappa}=\magn{\B}{f_\kappa}<\infty,$ then the von Neumann algebras $\A e_\kappa,\B f_\kappa$ are $\sigma$-finite and constitute their own $\kappa$-homogeneous components. By Lemma \ref{homcomp}, they are both $\kappa$-homogeneous as von Neumann algebras, hence $(\A e_\kappa, \tau_\A|\A e_\kappa)$ and $(\B f_\kappa, \tau_\B|\B f_\kappa)$ are isomorphic as von Neumann measure algebras by Theorem \ref{homog}.

On the other hand if $\magn{\A}{e_\kappa}=\magn{\B}{f_\kappa}=\lambda$ for some infinite cardinal $\lambda,$ then
       by Proposition \ref{prob} and \ref{sfinvN}, 
$\A e_\kappa$ and $\B f_\kappa$ are direct sums of $\lambda=\dec{\A e_\kappa}=\dec{\B f_\kappa}$ noncommutative measure algebras $(\A p_i,\tau_\A |\A p_i)$ and $(\B q_i,\tau_\B |\B q_i),$ respectively,  with $\tau_\A(p_i)=\tau_\B (q_i)=1$ for all $i$. Note that $\A p_i$ and $\B q_i$ are $\kappa$-homogeneous components of themselves. Indeed, if, say, $p_i$ is not a $\kappa$-component of itself, then by Lemma \ref{hom-maj}, there would be a $\kappa'$-homogeneous projection $p\leq p_i$ with $\kappa'\neq\kappa.$ But $p\leq e_\kappa,$ so Lemma \ref{hom-orth} yields a contradiction. Again by Lemma \ref{homcomp}, $\A p_i$ and $\B q_i$ are $\kappa$-homogeneous. 
      By Theorem \ref{homog}, the von Neumann measure algebras $(\A p_i, \tau_\A|\A p_i)$ and $(\B q_i,\tau_\B|\B q_i)$ are isomorphic. The direct sums of all the isomorphisms produce an isomorphism of $(\A,\tau_\A)$ and $(\B,\tau_\B)$.
\end{proof}

\subsection{Infinite tensor products} \label{itp}
To properly demonstrate the von Neumann algebra version  of
the probabilistic case of Maharam's theorem we 
will use various notions of infinite tensor products: of Hilbert spaces, of $C^*$-algebras and of von Neumann algebras. 
To avoid misunderstanding, we take some time
to carefully  describe the tensor products.   For further background or  definitions if needed, we direct the reader to for example \cite[III.3.1]{Black} and 
\cite[Chapter 11]{KR}.
 In the definitions, the index set is arbitrary. For our purposes, it is most convenient to have an infinite cardinal $\kappa$ as the set of 
 indices. 
We freely use the notation introduced here in the proofs that follow. 
We also use
the notation `$\Subset$' to denote a finite subset.

We start with a family $\{(H_i,\xi_i)\}_{i\in K}$ of Hilbert spaces $H_i$ with a distinguished unit vector $\xi_i.$
For  $F\Subset K,$ we put $H_F=\bigotimes_{i\in F}H_i$  and $\xi_F=\bigotimes_{i\in F}\xi_i.$ For $F,G\Subset K$ with $F\subseteq G,$  we have a unique isometric linear mapping  $h_{GF}: H_F\to H_G$ such that $\bigotimes_{i\in F}\eta_i\mapsto \bigotimes_{i\in G}\zeta_i,$ where $\zeta_i=\eta_i$ for $i\in F$ and $\zeta_i=\xi_i$ for $i\in G\setminus F.$  Then $(\{H_F\}_{F\Subset K}, \{h_{GF}\}_{F\subseteq G\Subset K})$ form an inductive system of Hilbert spaces.
Note that $\bigcup h_F(H_F)$ is dense in
the Hilbert space inductive limit $H_K$ and there exists a unique unit vector $\xi_K\in H_K$ such that $h_F\xi_F=\xi_K$ for all $F\Subset K.$
We say that the pair $(H_K,\xi_K)$ is the infinite tensor product of $\{(H_i,\xi_i)\}_{i\in K}$.

We now define infinite tensor products of $C^*$-algebras \cite[\S 11.4]{KR}. 
Consider  a family
$\{A_i\}_{i\in K}$ of unital $C^*$-algebras. For  $F\Subset K,$ we put $A_F=\bigotimes_{i\in F}A_i,$ where the tensor product used is the minimal one. For $F,G\Subset K$ with $F\subseteq G,$  we have a unique $*$-isomorphism $\theta_{GF}: A_F\to A_G$   such that $\theta_{GF}(a_F)=a_F\otimes \I_{G\setminus F}.$ Then $(\{A_F\}_{F\Subset K}, \{\theta_{GF}\}_{F\subseteq G\Subset K})$ forms an inductive system of $C^*$-algebras. We denote by $(A_K,\{\theta_F\}_{F\Subset K})$ the inductive limit of the system. Note that $\bigcup_{F\Subset K} \theta_F(A_F)$ is uniformly dense in $A_K.$  We say that the pair $A_K$ is the infinite tensor product of $\{A_i\}_{i\in K}$ and write $A_K=\bigotimes_{i\in K} A_i.$

Finally, we will consider infinite tensor products of von Neumann algebras, or, more precisely, of von Neumann probability algebras. We start with a family $(\M_i,\vf_i)_{i\in K},$ where $\M_i$ are (necessarily $\sigma$-finite) von Neumann algebras and $\vf_i$ are faithful normal states on $\M_i.$ First, we treat $\M_i$ as unital $C^*$algebras, and  define their $C^*$-algebraic infinite tensor product $M_K$
as we did above.  There is a
unique state $\vf_K$ on $M_K$ such that $\vf_F(\bigotimes_{i\in F}a_i)=\prod_{i\in F}\vf_i(a_i)$ for any $F\Subset K$ and $a_i\in \M_i$ for $i\in F$ (see \cite[Proposition 11.4.6]{KR}). We call the state $\vf_K$ the \emph{product state} and denote it by
 $\bigotimes_{i\in K}\vf_i.$ Having the algebra $M_K$ and the state $\vf_K$ on it, we can build the GNS-representation $(\pi_{\vf_K}, H_{\vf_K}, \xi_{\vf_K})$ of $M_K$
  (see \cite[Section 4.5]{KR}), and define $\M_K:=\pi_{\vf_K}(M_K)''$ acting on $H_{\vf_K}$ with a cyclic and separating vector $\xi_{\vf_K}.$ Let $\overline{\vf}_K$ be the faithful normal  state on $\M_K$ given by the vector state $\xi_{\vf_K}.$ Then we call the pair $(\M_K,\overline{\vf}_K)$ the \emph{infinite tensor product of von Neumann probability algebras} $(\M_i,\vf_i)$ and denote it by $\overline{\bigotimes}_{i\in K}(\M_i,\vf_i).$

 Assume now that  each $\M_i$ acts standardly in a Hilbert space $H_i$.
 That is,  there exists a cyclic and separating vector $\xi_i$.   For each $i\in K,$ let $\varphi_i$ be
 the vector state corresponding to $\xi_i.$ Let $\pi_i$ denote the (identity) representation of $\M_i$ on $H_i.$ For all $F\Subset K$ we form the tensor product $\pi_F:=\bigotimes_{i\in F}\pi_i$ in an obvious way. We easily check (see \cite[11.5.30]{KR}) that there is a
 unique representation $\pi_K$ (denoted by $\bigotimes_{i\in K}\pi_i$) of $M_K$ on $H_K$ such that $\pi_K(\theta_F(a_F))h_F=h_F\pi_F(a_F)$ for all $a_F\in A_F,$ for all $F\Subset K.$ Moreover, the vector $\xi_K$ is cyclic for $\pi_K$ and the vector state $\omega_{\xi_K}$ is, in fact, the product state $\bigotimes_{i\in K}\omega_{\xi_i}.$ Hence $\vf_K=\omega_{\xi_K}$ and, by \cite[Proposition 4.5.3]{KR}, 
 $\pi_K$ is unitarily equivalent to the GNS representation $\pi_{\vf_K}.$ That means that $\M_K$ is generated by $\pi_K(M_K)$ and that $\overline{\vf}_K$ is the (extension of the) vector state $\xi_K.$ The vector $\xi_K$ is not only cyclic, but also separating. This follows, for example, from the fact that the vectors $\xi_i$ are all separating, hence cyclic for the algebras $\M_i'.$ Hence $\overline{\bigotimes}_{i\in K}(\M_i',\omega_{\xi_i}),$ when represented in $H_K$, has a cyclic vector $\xi_K.$ But $\overline{\bigotimes}_{i\in K}(\M_i',\omega_{\xi_i})=\M_K'$ (see \cite{tak1} and III.3.1.2 in \cite{Black}), hence
 $\xi_K$ is separating for $\M_K.$ This means that the representation $\pi_K$ is in fact faithful, so that we can identify $M_K$ as a subset of $\M_K.$

We use the definitions above in the following special case.
In the sequel, we will use the space $\Cdb^2$  in three different guises: as a Hilbert space, as a $C^*$-algebra, and as a von Neumann algebra.  We fix an infinite cardinal $\kappa,$ take $K:=\kappa,$ and then for all $i< \kappa,$ put $H_i:=H$  and $\xi_i:=\xi,$ where $H=\Cdb^2$  and $\xi=(\frac{1}{\sqrt{2}},\frac{1}{\sqrt{2}}).$ We write $(H_\kappa,\xi_\kappa)=(H,\xi)^{\otimes\kappa}.$ Next,
$A_i:=A$ for all $i<\kappa,$ where $A=\Cdb^2.$ We denote by $A_\kappa$ the $C^*$-algebraic tensor product of $\kappa$ copies of $A$. Finally, $\A:=\Cdb^2, \, \tau$ is the normalized trace on $\A,$ and $\M_i:=A,\, \tau_i:=\tau$ for all $i<\kappa,$ with $(\A_\kappa,\tau_\kappa)$ being the infinite tensor product of $\kappa$ copies of the von Neumann measure algebra $(\A,\tau).$
It follows from the above that $\A$ acts standardly in $H$,
that is $\xi$ is both cyclic and separating for the action.
 The latter implies, in particular,
 that $\xi_\kappa$ is separating and cyclic for
 $\A_\kappa$, and $\tau_\kappa$ is a faithful normal tracial state on $\A_\kappa$.

The following proposition gives much important information about
 our `building blocks' for the classification of abelian von Neumann algebras. The proof of its last part follows closely the corresponding part in \cite[331K]{Fremlin}.

\begin{proposition}\label{akhom}
  For an infinite cardinal $\kappa,$ the von Neumann algebra $\A_\kappa$ is abelian, $\sigma$-finite, non-atomic and $\kappa$-homogeneous.
\end{proposition}

\begin{proof}
It is obvious that the $C^*$-algebra $A_\kappa$ is abelian, and so the continuity of multiplication on bounded subsets of $\A_\kappa$  together with the Kaplansky density theorem  (see, for example, \cite[Proposition 1.36 (2) and Theorem 1.60]{GL})  implies that $\A_\kappa$ is abelian. It is $\sigma$-finite,
indeed it admits a faithful normal tracial state $\tau_\kappa$, and a
cyclic and separating vector $\xi_\kappa$.

Let $p_j$, with $j<\kappa$, be projections of the form
$((1,0)) \otimes \I_{\kappa\setminus \{j\}}$ (that is the `elementary tensor' which  is
$\I_2$ in each component except the $j$th, which is $(1,0) \in \Cdb^2$).
For all $F\Subset \kappa$ the set
$P:=\{p_j\}_{j\in F}$ generates $\A_F\b{\otimes}\Cdb\I_{\kappa\setminus F}$. Since $\bigcup_{F\Subset \kappa}\theta_F(A_F)$
is uniformly dense
in $A_\kappa,$ it is strongly dense in $\A_\kappa$.
So the family $\{p_j\}_{j< \kappa}$ generates $\A_\kappa.$ Hence $\kk{\A_\kappa}\leq\kappa.$ If  $p\in\p{\A_\kappa}$ is  non-zero, then by Theorem \ref{equiv}(b), with $P$ as above, there is a countable subset $J$ of $\kappa$ such that for some non-zero $q \in \p{\A_J}$ we have $p=q \b{\otimes} \I_{\kappa\setminus J}.$

Suppose now that $p\in \p{\A_\kappa}$ is minimal. For each $j<\kappa$ we have $0\neq p=pp_j+pp_j^\perp,$ hence (at least) one of $pp_j$ and $pp_j^\perp$ is non-zero. Let $q_j$ be either $p_j$ or $p_j^\perp,$ with the provision that $pq_j\neq 0.$ Since $p$ is minimal and  $pq_j\leq p,$ we have $p\leq q_j$ for all $j<\kappa.$ Let $F\Subset \kappa$ be such that $\tau_\kappa(p)>1/2^{\# F},$ and let $q_F=\prod_{j\in F} q_j.$ Then
$$\tau_\kappa(q_F)= 1/2^{\# F} < \tau_\kappa(p)\leq \tau_\kappa(q_F),$$ a contradiction. Hence there are no minimal projections in $\A_\kappa.$

We check the equality $\kk{\A_\kappa}=\kk{\A_\kappa} p$ for any non-zero $p\in\p{\A_\kappa}.$  Indeed, if $G$ is a set of generators for $\A_\kappa,$  then $\{gp\colon g\in G\}$ is a set of generators for $\A_\kappa p,$ and hence the inequality $\kk{\A_\kappa}\geq \kk{\A_\kappa} p$ holds. In the other direction, note that no minimal projections occur in neither $\A_\kappa,$ nor $\A_\kappa p.$ By Lemma \ref{aleph0}, we have $\aleph_0\leq \kk{\A_\kappa p}\leq \kk{\A_\kappa}\leq \kappa.$ Consequently, if $\kappa=\aleph_0,$ then $\kk{\A_\kappa p}= \kk{\A_\kappa}.$ Otherwise, by Theorem \ref{equiv}(b), $p=q\b{\otimes} \I_{\kappa\setminus J}$ for some countable set $J$ and $q\in \A_J,$  so that $\A_\kappa p=\A_J q\overline{\otimes}\A_{\kappa\setminus J}.$ If $G$ generates $\A_\kappa p,$ then $\{g\I_{\kappa\setminus J}\colon g\in G\}$ generates $\A_{\kappa\setminus J},$ so that the isomorphism of $\A_\kappa$
and $\A_{\kappa\setminus J}$ implies $\kk{\A_\kappa}=\kk{\A_{\kappa\setminus J}}\leq \kk{\A_\kappa p},$ and again $\kk{\A_\kappa p}= \kk{\A_\kappa}.$ Hence once we prove that $\kk{\A_\kappa}\geq \kappa,$
the proof of the proposition will be finished.

Observe first that for any sequence $(i_k)_{k\in\Ndb}$ of (pairwise) distinct
indices from $\kappa$ we have $\bigwedge_{k\in \Ndb} p_{i_k}=0$ and $\bigvee_{k\in \Ndb} p_{i_k}=\I.$ In fact, the
normality  of $\tau_\kappa$ on the unit ball implies that
$$\tau_\kappa(\bigwedge_{k\in \Ndb} p_{i_k})=\lim_{n\to\infty}\tau_\kappa(\bigwedge_{k\leq n} p_{i_k})=\lim_{n\to\infty}2^{-n}=0,$$
and we get the same result if we replace
each $p_{i_k}$ with $p_{i_k}^\perp.$
Now, for a  non-zero $q_0\in \p{\A_\kappa}$  and $\ve>0,$ put
$$U(p,\ve):=\{q\in \p{\A_\kappa}\colon \|(q-p)\xi_\kappa\|\leq \ve\}.$$
There is $\delta>0$ such that
$$
\#\{i<\kappa\colon p_i\in U(p,\delta)\}<\infty.
$$
Otherwise, there would be an infinite  sequence of distinct indices $i_k$ such that  $p_{i_k}\in
U(p,2^{-k})$, and we would obtain
\begin{align*}
  1 &=\tau_\kappa(\I)=\tau_\kappa(p)+\tau_\kappa(p^\perp) \\
   & =\tau_\kappa(p(\bigwedge_{k\in \Ndb} p_{i_k})^\perp)+\tau_\kappa((\bigvee_{k\in \Ndb} p_{i_k})p^\perp) \\
  &\leq \sum_{k\in \Ndb} \tau_\kappa(pp_{i_k}^\perp)+\sum_{k\in \Ndb}\tau_\kappa(p_{i_k}p^\perp)\\
 &= \sum_{k\in \Ndb} \tau_\kappa(p(p-p_{i_k}))+\sum_{k\in \Ndb}\tau_\kappa(p_{i_k}(p_{i_k}-p))\\
  &= \sum_{k\in \Ndb} \|p(p-p_{i_k})\xi_\kappa\|^2+\sum_{k\in \Ndb}\|p_{i_k}(p_{i_k}-p)\xi_\kappa\|^2\\
  &\leq \sum_{k\in \Ndb} \|(p-p_{i_k})\xi_\kappa\|^2+\sum_{k\in \Ndb}\|(p_{i_k}-p)\xi_\kappa\|^2\\
  &\leq 2\sum_{k\in \Ndb} 4^{-k}=\frac{2}{3}<1,
\end{align*}
a contradiction.

Let now $Q\subseteq \p{\A_\kappa}$ be such that $Q$ generates $\A_\kappa$ as a von Neumann algebra and that $\# Q=\kk{\A_\kappa}.$ By Lemma \ref{Mah-gen}, we can assume that $Q$ is a Boolean algebra. For any non-zero $p\in\p{\A_\kappa}$ there exists $\delta>0$ such that the set $\{j<\kappa\colon p_j\in U(p,\delta)\}$ is finite. Take $n\in\Ndb$ such that $2^{-n+1}\leq \delta.$ By Theorem \ref{equiv}(b), there is a projection $q\in Q$ such that $q\in U(p,2^{-n}).$
If
 $\| (r-q) \xi_\kappa \| \leq 2^{-n}$ then
 $$\| (r-p) \xi_\kappa \| \leq 2^{-n} + \| (p-q) \xi_\kappa \| \leq 2^{-n+1}\leq \delta.$$
So $U(q,2^{-n})\subseteq U(p,\delta).$
Hence the number of indices $j<\kappa$ such that $p_j\in U(q,2^{-n})$ is finite. Moreover, $p\in U(q,2^{-n}).$
Let $\cU$ be the family of sets of 
 the form $U = U(q,2^{-n})$ with $q\in Q$ and $n\in \Ndb,$ for which the set of indices $J_U:=\{j<\kappa\colon p_j\in U\}$ is finite.
 Then $\cU$  covers $\p{\A_\kappa}\supseteq P,$  where $P:=\{p_i\}_{i<\kappa}.$  Since $\sup_{U\in\cU}J_U\leq\aleph_0,$ we have
\[
\kappa=\#\kappa\leq \#\cU\cdot \aleph_0=\#\cU\leq \kk{\A_\kappa}\cdot\aleph_0=\kk{\A_\kappa},
\]
which ends the proof.
\end{proof}

\begin{corollary}\label{Mah-prob}
 Let $\A$ be an abelian, non-atomic and $\kappa$-homogeneous von Neumann algebra, and let $\tau_\A$ be a faithful normal trace on $\A$ with $\tau_\A(\I)=1.$ Then $(\A,\tau_\A)$, as a von Neumann measure algebra, is isomorphic to $(\A_\kappa,\tau_\kappa).$
\end{corollary}
 
\begin{proof}
This follows directly from Theorems \ref{akhom} and \ref{homog}.
\end{proof}

The following theorem (together with Theorem \ref{bigf}) gives 
a complete classification, up to isomorphism, of 
abelian von Neumann measure algebras (c.f.\ \cite[332C]{Fremlin}):

\begin{theorem}[Maharam's theorem for von Neumann measure algebras] \label{vnmacase}
 Let $(\A,\tau_\A)$ be an 
abelian von Neumann measure algebra.
 For $\kappa=0,$ we put $\A_0:=\Cdb$ and let $\tau_0$ be the normalized trace on $\A_0.$ Then there are families $(\kappa_i)_{i\in I}$ and $(\gamma_i)_{i\in I}$ such that each $\kappa_i$ is either $0$ or an infinite cardinal, and each $\gamma_i$ is a strictly positive real number, and such that
 \[
 (\A,\tau_\A) \simeq \oplus_{i\in I} \, (\A_{\kappa_i},\gamma_i \, \tau_{\kappa_i}).
 \]
 as von Neumann measure algebras.
\end{theorem}

\begin{proof}
Let $P$ be the collection of all minimal projections in $\A,$ and let $q:=\sup P.$
By Lemma \ref{atoms},  $\A q$ is purely atomic, isomorphic to $\Cdb^\lambda$ with $\lambda:=\# P,$ and $\A q^\perp$ is non-atomic. If $\tau_\A(q^\perp)=\infty,$ then we use Proposition \ref{prob} to decompose $q^\perp$ into a sum of non-zero finite-trace projections. By Lemma \ref{clas1}, we can find an orthogonal family $\{z_j\}_{j\in J}$ of homogeneous
 finite-trace
projections with $\sum_{j\in J}z_j=q^\perp$.
Put $I:=P\cup J$ and $z_p:=p$ for $p\in P.$ Then $\{z_i\}_{i\in I}$ form an orthogonal family of homogeneous
finite-trace projections in $\A$ with $\sum_{i\in I}z_i=\I.$

Put $\gamma_i:=\tau_\A(z_i)$ for each $i\in I.$ If $i\in P,$ then clearly
$$(\A z_i,\tau_\A|\A z_i)=(\A p,\tau_\A|\A p)\simeq(\Cdb,\gamma_i\tau_0).$$
For each $i\in J$ there is by Theorem \ref{Mah-prob}, an infinite cardinal $\kappa_i$ such that
$$(\A z_i,\frac{1}{\gamma_i} \, \tau_\A|\A z_i)\simeq (A_{\kappa_i}, \tau_{\kappa_i}) .$$
  Putting
 $\kappa_i:=0$ for each $i\in P,$ we have for each $i\in I$ an isomorphism $\Phi_i$ between $(\A z_i,\tau_\A|\A z_i)$ and $(\A_{\kappa_i},\gamma_i \, \tau_{\kappa_i}).$ The direct sum of those isomorphisms yields the desired isomorphism of $(\A,\tau_\A)$ with $\oplus_{i\in I}\, (\A_{\kappa_i},\gamma_i \, \tau_{\kappa_i}).$
 \end{proof}

Since every abelian von Neumann algebra has a faithful normal semifinite trace (see
Corollary \ref{awvn}  or  \cite[Theorem VII.2.7]{tak1}), we have:

\begin{corollary} \label{vncase} An abelian von Neumann algebra is isomorphic to a direct sum of
$\Cdb^\lambda$ for some cardinal $\lambda$, and a direct sum of $\A_\kappa$
for some infinite cardinals $\kappa$.
\end{corollary}

\begin{remark}  Together Theorems \ref{vnmacase} and \ref{bigf} have a myriad of consequences and applications,
just as in the measure algebra case in \cite{Fremlin}.   Many of the latter applications may be found in
Chapter 33 in that reference, 
for example after the statements of the analogous results and in the \emph{Notes} sections there.  We will not take the time here to spell  out a sample of these, but the interested reader could look in \cite{Fremlin}  for details
in that parallel case. 
Note that Theorem  \ref{bigf} may be used to refine Theorem \ref{vnmacase} in various ways, using the extra information about the magnitudes $\magn{\A}{e_\kappa}$ and Lemma \ref{sfinvN}.
For example one may `clean up' the `summand clumps' associated with
each cardinal.   For example for each distinct 
infinite cardinal $\kappa$ there are, up to measure isomorphism, two cases:  one may either 1)\  take all of the associated scaling factors $\gamma_i$ to be 1, or 
2)\ replace all associated  $\gamma_i$ by a single $\gamma \neq 1$.
   (See the relationship between $\magn{\A}{e_\kappa}$ and $\gamma_i$ in 332K (a) in \cite{Fremlin}.) 
In particular,  an abelian von Neumann probability algebra is isomorphic to a finite or countable direct sum
 $\oplus_n \, (\A_{\kappa_n}, \gamma_n)$, with the cardinals $\kappa_n$ infinite and distinct if they
 are nonzero, and with $\sum_n \, \gamma_n = 1$.
 If in addition the predual is separable (or under some equivalent `countability' condition)
then  $\kappa_n$ is $0$ or
 $\aleph_0$ 
 for all $n$.

We remark that there is another famous and similar sounding theorem of  Kuratowski for Polish spaces up to Borel space isomorphism \cite{Sri}.
  In particular any two uncountable standard Borel spaces are
Borel isomorphic.

One may also for example apply these results to the question
of characterizing $*$-isomorphisms  $\pi : L^\infty(X,\A,\mu) \to L^\infty(Y,\B,\nu)$,
for localizable  measure spaces.  These are all induced as in Section 7 (see
for example Proposition \ref{Wextn})
from a measure algebra isomorphism between the
measure algebras of $(X,\A,\mu)$ and $(Y,\B,\nu_h)$, where
$\nu_h$ is as in the Radon-Nikodym theorem with 
 $h : Y \to (0,\infty)$ 
 $\B$-measurable.  Such a measure algebra isomorphism  may clearly be described in terms of the Theorems \ref{bigf} and \ref{vnmacase}.   Note that $h$ is the Radon-Nikodym derivative of the measure $\nu'(E) = \bar{\mu}(\pi^{-1}(\chi_E))$ with respect to $\nu$ (see the discussion above Theorem \ref{eqca}).
The Maharam
decompositions of $\nu$ and $\nu'$ may differ very slightly.   The homogeneous components are the same, but note 
for example that $\oplus_{n=1}^\infty \, \A_\kappa$ is $*$-isomorphic to $\A_\kappa$ 
for infinite $\kappa$.
 \end{remark}

\subsection{Infinite product measure spaces}
In this subsection the reader will find a simpler, more familiar, and seemingly more attractive way of representing the `building block' tensor products as $L^\infty$-spaces over Cantor cubes $\{0,1\}^\kappa.$ The advantage 
 for operator algebraists 
 of the `building blocks'  $\A_\kappa$ used in  the previous subsection is that they are prone to natural non-commutative generalization (or `quantization', as some like to call the process). Namely, replacing the commutative 
  $\Cdb^2$-algebras a couple of paragraphs above
Proposition \ref{akhom} with two-by-two matrices, and fixing a trace or a state on such matrices, results in a plethora of  hyperfinite factors of various kinds (see, for example, \cite[Theorem 12.3.8]{Kad2}).

First, we review some facts about
the infinite product of Radon probability  measures on compact spaces.
See for example \ p.\ 230--231 in \cite{Fol} for full details of the general setting. In our case 
things are simpler since the Radon probability  measure spaces are just
the `Bernoulli' probability space $(\{ 0,1 \} , \Sigma, \mu)$, where $\Sigma:=\cP (\{0,1\})$ and $\mu(\{0\})=\mu(\{1\})=\frac{1}{2}.$
For a cardinal $\kappa$ we define $(\Omega_i, \Sigma_i, \nu_i) = (\{ 0,1 \} , \Sigma
, \mu)$ for all $i < \kappa$.
Set  $\Omega_\kappa = \{0,1\}^\kappa$ with the product topology.
A {\em measurable cylinder} is a product $C = \Pi_i \, E_i$ of
 nonempty sets $E_i \in \Sigma_i$ such that $E_i = \Omega_i$ for all but finitely many $i$.
These cylinders are the basis for the product topology.  The complement of a cylinder $C$ is a
finite union of cylinders, so $C$ is clopen.
 Define $\rho(C) = \frac{1}{2^k}$ if exactly $k$ of the $E_i$ above are not $\{ 0, 1 \}$.
The span ${\mathcal S}$ of the characteristic functions of the measurable cylinders is a
norm dense unital $*$-subalgebra of $\C(\Omega_\kappa)$.  It is written as $C_F(X)$ in
\cite[Theorem 7.28]{Fol}, and it is proved that there is a
positive contractive linear functional
$\sigma_\kappa$ on $\C(\Omega_\kappa)$ such that
$\sigma_\kappa( \chi_C ) = \rho(C)$ for every measurable cylinder $C$.
By the Riesz representation theorem  there is a Radon (in the sense of  Section
\ref{Rad}) probability measure $\nu_\kappa$ on a complete $\sigma$-algebra
$\Sigma_\kappa$ on $\Omega_\kappa$,  such that
$\sigma_\kappa(f) = \int_{\Omega_\kappa} \, f \, d \nu_\kappa$ for all $f \in \C(\Omega_\kappa)$.
We call $\nu_\kappa$ the {\em product measure}.   It clearly agrees with $\rho$ on
measurable cylinders.
Since $\nu_\kappa$ is finite it is an exercise that $\nu_\kappa$ is both inner and outer regular.
In fact $\Sigma_\kappa$ is the completion of the Borel $\sigma$-algebra. (Indeed
by a well known argument in measure theory, inner and outer  regularity forces any $E \in \Sigma_\kappa$ to be  sandwiched between
an $F_\sigma$ and a $G_\delta$ set of the same measure).

Another common way to define the product measure $\nu_\kappa$ 
(see for example Volume 2 in \cite{Fremlin}) is to first form the natural
outer measure $\rho^*$ (see \cite[Proposition 1.10]{Fol}) from the set function
$\rho$ above.   Then $\nu_\kappa$ equals the restriction $\mu$ of $\rho^*$
to the $\sigma$-algebra  on $\Omega_\kappa$  coming from Caratheodory's theorem
applied to the outer measure $\rho^*$ (see 1.11 in  \cite{Fol}).
This fact is hard to find
in the literature, but since we do not really need it we shall only
prove  that $\mu = \nu_\kappa$ in the case
where $\kappa$ is countable (for the uncountable case see for example (c) in the proof of
\cite[Theorem 415E]{Fremlin}). 
Since cylinders are open, it is clear from Caratheodory's construction
that $\mu$ is outer regular and complete. 
By taking complements, and since $\Omega_\kappa$ is compact, $\mu$ is inner regular. 
If  $\kappa$ is countable then $\Omega_\kappa$ is second countable by topology.
Hence any open set in $\Omega_\kappa$
is in $\Sigma$ (since the basis of cylinders are in  $\Sigma$).  So $\mu$ is a Radon measure. 
By the argument at the end of the last paragraph, the domain of $\mu$ is the completion of
the Borel $\sigma$-algebra.
Integration with respect to $\mu$ agrees with $\sigma_\kappa$ on the dense set ${\mathcal S}$, so agrees on $\C(\Omega_\kappa)$.  By the `uniqueness' in the Riesz representation theorem,
 $\mu = \nu_\kappa$ on Borel sets, hence on all of $\Sigma_\kappa$.

\begin{proposition}\label{cont}
 There exists a  $*$-isomorphism $\theta :
 A_\kappa=(\Cdb^2)^{\otimes \kappa}
\to \C(\{0,1\}^\kappa)$
  such that $\int \, \theta(a) \, d \nu_\kappa = \varphi_\kappa(a)$ for
  $a \in A_\kappa$, where $\varphi_\kappa$
 is the product state in the paragraphs above Proposition {\rm \ref{akhom}}.
\end{proposition}

\begin{proof}
We use some of the notation established above and in the first
paragraphs of the proof of
  Proposition \ref{akhom}, such as ${\mathcal S}$ and
 $\sigma_\kappa$.
 For $F \Subset \kappa$, $A_F$ is the algebra generated by 1 and $p_j$ for $j \in F$. It is well known that $(\ell^\infty_2)^{\otimes (n)}  \cong
  \ell^\infty( \{ 0,1 \}^{(n)})$. This gives the canonical isomorphism of $A_F$ onto $C(\{ 0,1 \}^F)$, which takes the generators $p_j$ to the characteristic function of the cylinder in $\{ 0,1 \}^F$ with $\{ 1 \}$ as the $j$th slice. Composing this  isomorphism with the canonical isometry $C(\{ 0,1 \}^F) \to C(\Omega_\kappa)$ (namely, composition with the projection  $\pi_F$), gives a canonical
 faithful, hence isometric, $*$-homomorphism $A_F \to C(\Omega_\kappa)$. Since $\cup_{ F \Subset \kappa} \, A_F$ is dense in $A_\kappa$, it is an exercise that we obtain an isometric $*$-homomorphism $\theta : A_\kappa \to C(\Omega_\kappa)$. The range of $\theta$ contains the dense set ${\mathcal S}$, and so $\theta$ is onto.

Observe that $\varphi_k$ agrees with 
$\sigma_\kappa \circ \theta$ on $A_F$. (To see this one may here use the fact that $\varphi_\kappa$ is multiplicative on products of the $p_i$ for $i < \kappa$, by the product state
property). So by density these maps agree everywhere.  \end{proof}

\begin{corollary}
\label{AkapaL}
  The pair $(\A_\kappa,\tau_\kappa)$ is isomorphic, as a von Neumann measure algebra, to $(L^\infty(\Omega_\kappa,\Sigma_\kappa,\nu_\kappa),\sigma_{\kappa}), \, \int \, \cdot \, d \lambda )$ where $\sigma_{\kappa}$ is given by
  $\sigma_{\kappa}(a)=  \int_{\Omega_\kappa} \, a\, d\nu_\kappa$ for $a\in L^\infty(\Omega_\kappa,\Sigma_\kappa,\nu_\kappa).$
\end{corollary}

\begin{proof}
Since $(\A_\kappa, \tau_\kappa)$ was obtained from $A_\kappa$ via the GNS-construction with
 $\varphi_\kappa,$
and since $\varphi_\kappa = \sigma_\kappa \circ \theta$  by the last result,
it is enough to check that the GNS-construction with $\sigma_\kappa$ leads from $\C(\Omega_\kappa)$  to $L^\infty(\Omega_\kappa,\Sigma_\kappa,\nu_\kappa)$ acting by multiplication on $L^2(\Omega_\kappa,\Sigma_\kappa,\nu_\kappa).$    However this relation between
the GNS construction
applied to an integral, and $L^\infty$ of the associated measure,
is a general fact about Radon probability measures.
We  prove this general fact in our case (see also \cite[Theorem III.1.2]{tak1}).   We suppress the subscript 
$\kappa$ below, writing $\nu_\kappa$ as $\nu$, etc.  Note that $H_{\sigma}$, obtained as the completion of $\C(\Omega)$
 with the inner product $\langle f , g \rangle_\sigma = \int_{\Omega} \, f\bar{g}\, d \nu,$
is $L^2(\Omega, \nu)$. This follows from the density of $\C(\Omega)$ in $L^2(\Omega, \nu)$ 
for a Radon measure (see for example\ \cite[Proposition 7.9]{Fol}).
Let $\pi_{\sigma}$ be the GNS-representation of $\C(\Omega)$ on $H_{\sigma}$. 
Now $L^\infty(\Omega, \nu)$
is  isometrically isomorphic and
 weak* homeomorphic to the canonical  `multiplication'
 von Neumann algebra on $L^2(X,\mu)$ (see Theorem \ref{lsigf1}). 
 Since $\pi_{\sigma}(\C(\Omega))$ is a subset of the latter, 
 it is enough to show that  $\C(\Omega)$ is 
weak* dense in $L^\infty(\Omega, \nu)$. 
This follows from the uniqueness in the Riesz representation theorem: if $f \in L^1(\Omega, \nu)$ 
and the measure  $f \, d \nu$ annihilates $\C(\Omega)$ then $f d \nu = 0$ and $f = 0$.

The measure isomorphism follows by weak* density and the relation  $\varphi_\kappa = \sigma_\kappa \circ \theta$ on $A_\kappa$.
 \end{proof}

\begin{corollary} \label{leb}
  For $\kappa=\aleph_0,$ the 
  pair $(\A_\kappa,\tau_\kappa)$ is isomorphic, as a von Neumann measure algebra, to $(L^\infty([0,1],\Lambda,\lambda),$ where $\lambda$ is the Lebesgue measure on $[0,1]$ with the $\sigma$-algebra $\Lambda$ of Lebesgue-measurable subsets of $[0,1].$
\end{corollary}

\begin{proof}
The complete proof can be found in \cite[254K]{Fremlin}. The main idea is the (well known)
 construction of a measure preserving bijection between 
  $\Omega_\kappa$ and $[0,1]$
  by putting first $\Phi(a)=\sum_{i=0}^{\infty}2^{-i-1}a_i$ for $a=(a_i)_{i<\kappa},$ and then mending it  in a countable number of points to make it a bijection.
\end{proof}

The energetic reader may show, using associativity of the product of cardinals, that the last corollary holds
for $\kappa> \aleph_0$ with Lebesgue measure $\lambda$ replaced by its power $\lambda^\kappa$.

\bigskip

{\em Acknowledgments.}   We thank  Adam Paszkiewicz for several conversations and
inputs, and  Alex Bearden, Lander Besabe, Arianna Cecco, 
and Cedric Arhancet for some reading of parts of this manuscript and comments. Our huge debt to D.\ H.\ Fremlin and his extraordinary magnum opus \cite{Fremlin} will be clear almost throughout. 
particularly his Volume 3 of course (the standard text on measure algebras).


\begin{thebibliography}{888}
\bibitem{bade} W. G.  Bade, \textit{Weak and strong limits of spectral operators.} Pacific J. Math. 4 (1954), 393–-413.
\bibitem{Black}  B. Blackadar,  \textit{Operator algebras. Theory of $C^*$-algebras and von Neumann algebras.} Encyclopaedia of Mathematical Sciences, 122.
Springer-Verlag, Berlin, 2006. ISBN: 978-3-540-28486-4.
\bibitem{Cohn} D. L. Cohn, \textit{Measure Theory:Second Edition}, Birkh\"auser Advanced Texts Basler Lehrb\"ucher, Springer, 2013.
\bibitem{DalesDLS} H. G. Dales, F. K. Dashiell, A. T.-M. Lau, and D. Strauss,
{\em Banach spaces of continuous functions as dual spaces,}
CMS Books in Mathematics/Ouvrages de Math\'ematiques de la SMC. Springer, Cham, 2016.
\bibitem{DJT} J. Diestel, H. Jarchow, and A. Tonge, {\em Absolutely summing operators}, Cambridge Studies in Adv. Math. 43, Cambridge University Press, Cambridge, 1995.
\bibitem{DixS}  J. Dixmier, {\em  Sur certains espaces consid\'er\'es par M.H. Stone,} Summa Brasiliensis Mathematicae {\bf 2} (1951), 151--182.
\bibitem{DS} N. Dunford and J. T. Schwartz, {\em Linear operators. Part III. Spectral operators.} Reprint of the 1971 original. Wiley Classics Library, Wiley-Interscience, John Wiley \& Sons, Inc., New York, 1988.
\bibitem{Fol} G.B. Folland, {\em Real Analysis: Modern Techniques and Their Applications,} 2nd Edition, Wiley, 2007.
\bibitem{Fremlin} D. H.  Fremlin,  {\em  Measure theory,} Torres Fremlin, Colchester, 2003-2004.
\bibitem{FremlinMA} D. H. Fremlin, {\em Measure algebras,}  Chapter 22 in Vol. 3 of \textit{The Handbook of Boolean Algebras}, ed. by J. D. Monk with R. Bonnet, North-Holland, 1989.
\bibitem{GT} E. Gardella and H. Thiel, {\em Isomorphisms of algebras of convolution 
operators,}  Ann. Sci. Ecole Norm. Sup., 2018 preprint, to appear. 
\bibitem{GL} S. Goldstein and L. E. Labuschagne, {\em Notes on noncommutative $L^p$ and Orlicz spaces,}
Lodz University Press, 2020. 
\bibitem{haag-Nw} U. Haagerup, {\em Normal weights on $W^*$-algebras,} J.\ Funct.\ Analysis {\bf 19} (1975), 302--317.
\bibitem{KR} R. V.  Kadison and J. R. Ringrose, \textit{Fundamentals of the theory of operator algebras. Vol. I.
Elementary theory.} Reprint of the 1983 original. Graduate Studies in Mathematics, 15. American Mathematical Society, Providence, RI, 1997. ISBN: 0-8218-0819-2.
\bibitem{Kad2} R. V. Kadison, {\em The von Neumann algebra characterization theorems,}  Expo. Math. \textbf{3} (1985), 193--227.
\bibitem{maharam} D. Maharam, {\em On homogeneous measure algebras}, Proc.\ Nat.\ Acad.\ Sci.\ U.S.A.\ \textbf{28}(1942), 108-111.
\bibitem{OR} S. Okada and W. Ricker, {\em Classes of localizable measure spaces,} Positivity and noncommutative analysis, 425--469, Trends Math., Birkh\"auser/Springer, Cham, 2019.
\bibitem{DP} D. Pavlov, {\em Gelfand-type duality for commutative von Neumann algebras,} revision of 21 June 2020, arXiv:2005.05284v2
\bibitem{Ped} G. K. Pedersen, {\em  $C^*$-algebras and their automorphism groups,} Academic Press, London, 1979.
\bibitem{Pedeu} G.K. Pedersen, {\em The existence and uniqueness of the Haar integral on a locally compact topological group,} unpublished manuscript  (2000). 
\bibitem{PT} G. K. Pedersen and M. Takesaki, {\em  The Radon-Nikodym theorem for von Neumann algebras,} Acta Math. {\bf 130} (1973), 53--87.
\bibitem{JP} J. Peterson, {\em Notes on operator algebras,} April 2020.
\bibitem{rick} W. J. Ricker, {\em Operator algebras generated by commuting projections: a vector measure approach,} Lecture Notes in Mathematics 1711, Springer-Verlag, Berlin, 1999.
\bibitem{Sakai} S. Sakai, \textit{$C^*$-Algebras and $W^*$-Algebras}, Classics in Mathematics, Springer-Verlag, Berlin Heidelberg, 1998.
\bibitem{Sal} D. Salamon, \textit{Measure and Integration}, preprint, 2020.
\bibitem{Segal} I. E. Segal, {\em  Equivalences of measure spaces,} American J. Math.\ \textbf{73} (1951), 275--313.
\bibitem{sher} Sherman, D. \textit{On cardinal invariants and generators for von Neumann algebras}.
 Canad. J. Math.  64  (2012),  no. 2, 455--480.
\bibitem{Sri} S. M. Srivasatava, {\em  A course on Borel sets,}
Graduate Texts in Mathematics, 180. Springer-Verlag, New York, 1998.
\bibitem{tak1}   M. Takesaki, {\em Theory of operator algebras I}, Springer, New York, 1979.
\end{thebibliography}
\end{document}